%% file: neurips_2026.tex
\documentclass{article}

\PassOptionsToPackage{compress}{natbib}
\usepackage[preprint]{neurips_2026}
\usepackage[utf8]{inputenc} % allow utf-8 input
\usepackage[T1]{fontenc}    % use 8-bit T1 fonts
\usepackage{hyperref}       % hyperlinks
\usepackage{url}            % simple URL typesetting
\usepackage{booktabs}       % professional-quality tables
\usepackage{amsfonts}       % blackboard math symbols
\usepackage{nicefrac}       % compact symbols for 1/2, etc.
\usepackage{microtype}      % microtypography
\usepackage{xcolor}         % colors
\usepackage{booktabs}
\usepackage{tabularx}
\usepackage{makecell}
\input{header}

\usepackage{hyperref}
\hypersetup{
    colorlinks=true,
    linkcolor=blue,
    citecolor=blue,
    urlcolor=blue
}

\usepackage[capitalise,noabbrev,nameinlink]{cleveref}
\Crefname{assumption}{Assumption}{Assumptions}
\Crefname{fact}{Fact}{Facts}
\title{Adaptive Inference for Functionals of M-Estimands}

\author{%
  James Leiner  \\
  Carnegie Mellon Univeristy\\
  \texttt{jleiner@stat.cmu.edu} \\
  \And
  Aurélien Bibaut \\
  Netflix\\\
  \texttt{abibaut@netflix.com} \\
    \And
  Nathan Kallus \\
  Netflix and Cornell Tech \\
  \texttt{kallus@cornell.edu} \\
      \And
  Aaditya Ramdas \\
  Stanford University \\
  \texttt{aramdas@stanford.edu} \\
    \And
  Koulik Khamaru \\
  Rutgers University \\
  \texttt{kk1241@rutgers.edu} \\
}

\begin{document}

\maketitle

\begin{abstract}
\input{abstract}
\end{abstract}

\section{Introduction}

\input{intro}

\section{Problem Setup} \label{sec:setup}
\input{setup}
\section{Main Results} \label{sec:method}
\input{methodology}

\section{Experiments} \label{sec:experiments}
\input{experiments}

\section{Conclusion} \label{sec:conclusion}
\input{conclusion}
\newpage
\bibliographystyle{chicago}
\bibliography{ref}

\appendix
\input{appendix}

%%%%%%%%%%%%%%%%%%%%%%%%%%%%%%%%%%%%%%%%%%%%%%%%%%%%%%%%%%%%

%%%%%%%%%%%%%%%%%%%%%%%%%%%%%%%%%%%%%%%%%%%%%%%%%%%%%%%%%%%%

%\newpage
%\input{checklist.tex}

\end{document}

%% file: header.tex
\usepackage{mathtools}
\usepackage{bbm}
\usepackage{amsthm}
\usepackage{enumitem}
\usepackage{amsfonts}
\usepackage{amsmath}
\usepackage{amssymb}

\newcommand{\Filt}{\mathcal{F}_{t-1}}

\newcommand{\E}{\mathbb{E}}
\newcommand{\R}{\mathbb{R}}
\newcommand{\Var}{\mathrm{Var}}

\newcommand{\pto}{\overset{p}{\to}}
\newcommand{\dto}{\overset{d}{\to}}
\newcommand{\norm}[1]{\left\lVert#1\right\rVert}

\newcommand{\indep}{\perp\!\!\!\perp}

\newtheorem{fact}{Fact}
\newtheorem{theorem}{Theorem}
\newtheorem{lemma}{Lemma}

\newtheorem{proposition}{Proposition}
\newtheorem{assumption}{Assumption}
\newtheorem{remark}{Remark}
\newtheorem{example}{Example}

\DeclareMathOperator*{\argmin}{argmin}

%% file: abstract.tex
Reinforcement learning and contextual bandit algorithms have become increasingly common in sequential decision-making applications. When these methods are deployed in high-stakes domains, there is growing interest not only in learning effective policies, but also in conducting statistical inference for quantities learned under adaptive data collection. However, classical procedures applied naively in these settings can fail: even when estimators are unbiased, their variance becomes path-dependent and as a result may not be asymptotically normal. A growing literature has emerged to ameliorate this problem, but solutions tend to be problem specific and often rely on correct specification of a working model. In this work, we develop a unified framework for constructing asymptotically valid confidence intervals to cover smooth functionals of nonparametric M-estimands under adaptive sampling. Under Neyman orthogonality, we provide two novel methods for performing inference: (1) a self-normalized statistic based on the realized quadratic variation of the influence function and (2) a statistic using a plug-in estimate of the conditional variance based on reweighted influence function increments. Our results allow for flexible nonparametric estimation of nuisance parameters and remain valid under model misspecification. Our theory is  supported by a simulation study for a dynamic pricing application which demonstrates that this method can produce asymptotically valid confidence intervals where standard methods fail. 

%We identify a previously underexplored regime in which the asymptotic variance converges to a random limit. In this setting, classical inference fails, but a self-normalized statistic based on realized quadratic variation satisfies a central limit theorem without explicit variance estimation. For fully adaptive settings where variance may not stabilize to even a random quantity, we introduce a plug-in estimator of the time-varying conditional variance based on reweighted influence function increments, enabling valid inference under minimal assumptions. Our theoretical results are supported by a simulation study for a dynamic pricing application which demonstrates that standard methods can be severely miscalibrated in adaptive regimes while our approach remains valid. 

%% file: intro.tex
Reinforcement learning algorithms have become widespread due to their state-of-the-art performance in a wide variety of areas including health interventions \citep{info:doi/10.2196/jmir.7994}, personalized web recommendations \citep{afsar2022reinforcement}, public policy \citep{https://doi.org/10.3982/ECTA17527,10.1093/jeea/jvad067} and clinical trials \citep{https://doi.org/10.1002/sim.3720,info:doi/10.2196/18477}. The widespread adoption of these approaches means that many modern datasets no longer obey an $i.i.d.$ assumption, and attempting to apply classical statistical techniques to these data results in biased parameter estimates \citep{NEURIPS2019_65b1e92c} and invalid inferential guarantees. 

A classical result \citep{1982laiwei} demonstrates that a sufficient condition for ordinary least squares parameters to concentrate in the linear model $y_t = \beta^{T}x_{t} + \epsilon_t$ with adaptively chosen $x_t$ is for the covariance matrix $\sum_{t=1}^{T} x_{t}x_{t}^{T}$ to concentrate around a sequence of deterministic matrices $\{A_{T}\}_{T\ge1}$. In contextual bandit problems, this condition typically requires that the sampling policy converge to a deterministic stationary policy independent of the data trajectory. In cases where the optimal policy is not unique or there are many close ties, common bandit algorithms (e.g. $\epsilon$-greedy, Thompson sampling, UCB) may not converge \citep{pmlr-v80-deshpande18a,NEURIPS2020_6fd86e0a,khamaru2021near} and classical approaches to inference applied naively will consequently have non-Gaussian limiting distributions. 

Motivated by these results, a recent literature has developed to provide inferential guarantees on adaptively collected data. One set of approaches assumes that the sampling policies are known and appropriately reweighs observations to reduce bias or stabilize variance \citep{2019hadad,kato2021efficient,bibault2021postcontext,zhang2021mestimators,cook2023semiparametric,guo2025statisticalinferencemisspecifiedcontextual,leiner2026adaptiveoffpolicyinferencemestimators}. An alternative line of work assumes that the sampling policies are unknown and instead uses an online debiasing approach to ensure valid inference \citep{pmlr-v80-deshpande18a,NEURIPS2023_a399456a,khamaru2021near}. 

In the above cases, the inferential target is usually either a finite-dimensional parameter in a well-specified parametric model (e.g. least squares) or a smooth functional of the conditional regression function (e.g. average treatment effect). Despite these differences, the proposed estimators share a common structure: they combine importance weighting or martingale arguments with bias correction terms that resemble influence function adjustments. However, these constructions are typically developed in a problem-specific manner, with separate analyzes for each estimand and modeling framework. As a result, there is currently no general framework that unifies these approaches in the way that semiparametric efficiency theory \citep{bickel1993efficient} and double machine learning \citep{10.1111/ectj.12097,10.1093/biomet/asaa054} do in the $i.i.d.$ setting.

Moreover, many existing methods impose restrictive assumptions on the underlying model or data-generating process. In particular, several approaches either assume correctly specified parametric models or rely on limited forms of nuisance estimation, precluding the use of flexible  nonparametric estimators. Methods that allow model misspecification are typically confined to parametric Z estimators and require strong conditions such as convergence of the sampling policy to a deterministic limit \citep{guo2025statisticalinferencemisspecifiedcontextual}, accurate estimation of the conditional moments of a score function \citep{leiner2026adaptiveoffpolicyinferencemestimators}, or longitudinal settings where batch sizes can tend to infinity \citep{zhang2023statistical}. On the other hand, methods that do allow flexible estimation of nuisance parameters \citep{10.1214/24-AOS2485} do not provide coverage guarantees under misspecification.  These limitations restrict the applicability of current techniques in modern machine learning settings, where models are often misspecified and nuisance components are estimated using black box regression approaches.
\begin{table}[t]
\centering
\caption{List of target estimands studied in prior work that are specific instantiations of this framework. We note that this is not exhaustive and additional targets are discussed in Appendix~\ref{appendix:prior}.}
\label{tab:examples}

\scriptsize
\setlength{\tabcolsep}{2pt}
\renewcommand{\arraystretch}{1.05}

\begin{tabularx}{\linewidth}{
@{}
>{\raggedright\arraybackslash}p{0.18\linewidth}
>{\raggedright\arraybackslash\footnotesize}p{0.22\linewidth}
>{\centering\arraybackslash}p{0.24\linewidth}
>{\centering\arraybackslash}X
@{}
}
\toprule

\textbf{Target estimand}
&
\textbf{Prior Work}
&
$\theta_{P_e}$
&
$\Psi_e(\theta_{P_e})$
\\

\midrule

Average treatment effect
(Appendix~\ref{subsec:ATE})
&
\cite{2019hadad,bibault2021postcontext,kato2021efficient,cook2023semiparametric}
&
$\displaystyle
\E[Y\mid A=a,C=c]
$
&
$\displaystyle
\E_{P_C}\!\left[
\theta_{P_e}(1,C)-\theta_{P_e}(0,C)
\right]
$
\\

\addlinespace[2pt]

Parametric $M$-estimator
(Appendix~\ref{subsec:M-estimator})
&
\cite{zhang2021mestimators,guo2025statisticalinferencemisspecifiedcontextual,leiner2026adaptiveoffpolicyinferencemestimators}
&
$\displaystyle
\argmin_{\theta\in\Theta}
\E_{P_e}[\ell_\theta(C,A,Y)]
$
&
$\displaystyle
v^\top\theta_{P_e},
\quad \text{fixed } v
$
\\

\addlinespace[2pt]

Partial linear regression
(Appendix~\ref{subsec:partial_glm})
&
\cite{10.1214/24-AOS2485}
&
$\displaystyle
\argmin_{\theta}
\E_{P_e}\!\left[
\bigl(
Y-\theta^\top X- \eta_{P_e}
\bigr)^2
\right]
$
&
$\displaystyle
v^\top\theta_{P_e},
\quad \text{fixed } v
$
\\

\addlinespace[2pt]

Value of an optimal policy
(Appendix~\ref{subsec:optimal_policy})
&
\cite{10.1214/15-AOS1384}
&
$\displaystyle
\E[Y\mid A=a,C=c]
$
&
{\footnotesize
$\displaystyle
\sup_{\pi\in\Pi}
\E_P\!\left[
\int_{\mathcal A}
\pi(a\mid C)\theta_{P_e}(a,C)\,\mu(da)
\right]
$
}
\\

\bottomrule
\end{tabularx}
\end{table}

\textbf{Summary of Contributions} 
 To address these limitations, we develop a unified framework for inference on smooth functionals of nonparametric M-estimands under adaptive data collection. Our approach builds on recent advances in automatic debiased machine learning \citep{https://doi.org/10.3982/ECTA18515,vanderlaan2026automaticdebiasedmachinelearning}, which show that in the i.i.d. setting, smooth functionals of M-estimands can be represented via a Riesz representer that links the target functional to the curvature of the underlying loss. For adaptively collected data, we provide a general construction of one-step estimators that use this representation, enabling valid inference with flexible black box parameter and nuisance estimation. Specifically:
\begin{enumerate}
    \item We provide a \textbf{general method} for performing inference on smooth functionals of non-parametric $M$-estimands, including a first order expansion (\cref{thm:pe_expansion}) and a generic recipe for evaluating the resulting influence function on adaptively collected datasets. 
    \item We provide an \textbf{asymptotic normality} result (\cref{thm:clt}) that uses a novel plug-in estimate (\cref{thm:var_estimate})
    of the (time-varying) conditional variance of the influence function. Crucially, our variance estimator is based on reweighted estimated influence function increments and does not require training additional models or data splitting in contrast to prior work. 
    \item In prior work, variance was either assumed to stabilize to a deterministic limit or assumed not to converge. We identify an \textbf{intermediate regime} where variance converges to a stable but non-deterministic quantity. In this case, we provide an additional central limit theorem (\cref{thm:self-normalized}) based on \emph{self-normalized} increments of the influence function that does not require bespoke variance estimation techniques. 
\end{enumerate}
\textbf{Paper Organization} In \cref{sec:setup} we introduce the contextual bandit problem and define our class of target estimands. In \cref{sec:method} we derive the corresponding efficient influence function and show that an appropriately defined one-step estimator is asymptotically normal. In \cref{sec:experiments}, we demonstrate the utility of our framework using a dynamic pricing application. We provide concluding thoughts in \cref{sec:conclusion}. All proofs are deferred to the Appendix.

\paragraph{Notation} \label{appendix:notation}
 We define $[n] := \{1,...,n\}$ for a positive integer $n$. We denote $\lfloor x \rfloor$ for some $x\in \R$ to refer to the largest integer less than or equal to  $x$. For a probability measure $P$, we write $Pf := \int f(z)\, dP(z)$.

%For a function $f_{\theta} : \R^{d} \rightarrow \R$, we refer to $\dot{f}_{\theta}$ as the gradient of $f$ with respect to $\theta$, $\ddot{f}_{\theta}$ as the Hessian matrix with respect to $\theta$, $\dddot{f}_{\theta}$ as the third derivative with respect to $\theta$, and so on. We denote $e_{j}$ as the $j$-th standard basis vector in $\R^{d}$. For two matrices $A,B$, $A \succeq B$ means that $A-B$ is positive semidefinite. For a sequence of observations $\{Z_{t}\}_{t=1}^{T}$, we define the empirical measure $P_T := \frac{1}{T}\sum_{t=1}^T \delta_{Z_t}$. 

When $\Theta$ and $\Upsilon$ are normed linear spaces and 
$f_{\theta,\eta} : \mathcal{Z} \to \mathbb{R}$ is a function indexed by $(\theta,\eta) \in \Theta \times \Upsilon$, 
we denote the directional derivative with respect to the $\theta$-argument by
$$\partial_\theta f_{\theta,\eta}(z)[h]
:=
\left.\frac{d}{d\epsilon} f_{\theta+\epsilon h,\eta}(z)\right|_{\epsilon=0},$$
for $h \in \Theta$. Similarly, we define the second-order derivative as $$\partial_\theta^2 f_{\theta,\eta}(z)[h,u]
:=
\left.\frac{\partial^2}{\partial \epsilon \partial \delta}
f_{\theta+\epsilon h+\delta u,\eta}(z)\right|_{\epsilon=\delta=0}.$$
Mixed derivatives are defined in the natural way, e.g.
$$\partial_\eta \partial_\theta f_{\theta,\eta}(z)[h,u]
:=
\left.\frac{\partial^2}{\partial \delta \partial \epsilon}
f_{\theta+\epsilon h,\eta+\delta u}(z)\right|_{\epsilon=\delta=0}.$$
When needed, we assume that $f_{\theta,\eta}$ is Fréchet differentiable in $\theta$.

%% file: setup.tex
We consider the contextual bandit problem over $T$ rounds. At each round $t=1,...,T$, a context $C_t \in \mathcal{C}$ is observed, an action $A_t \in \mathcal{A}$ is selected, and then an outcome $Y_t \in \R$ is observed. We denote the history up to time $t-1$ as $\mathcal{H}_{t-1} := \{C_i,A_i,Y_i \}_{i=1}^{t-1}$ and let $\mathcal{F}_{t-1}:=\sigma(\mathcal{H}_{t-1})$ be the corresponding filtration. At each round $t$, the observer selects a policy $\Pi_{t}$, measurable with respect to $\mathcal{F}_{t-1}$ and then samples $A_t \sim \Pi_{t}(\cdot\mid C_t)$. In certain contexts, it may be more  useful to collapse $C_t$ and $A_t$ into a single object which we will label as $X_{t} := (C_t,A_t) \in \mathcal{C} \times \mathcal{A}$. Similarly, we will label $Z_{t} = (C_t,A_t,Y_t) \in \mathcal{Z}$ with $\mathcal{Z} : = \mathcal{C} \times \mathcal{A} \times \mathbb{R}$ as a compact notation when referring to the observation tuple at step $t$.  We denote $P_{t}$ as the joint distribution of $Z_t$ conditional on $\mathcal{F}_{t-1}$.  As is standard in the literature, we assume that the marginal distribution of contexts and conditional distribution of rewards \emph{given} particular choices of actions and rewards are fixed across time. The dependence across time arises solely through the policy $\Pi_t$, which may depend on past history.

\begin{assumption}
\label{assumption:distribution}
The data generating process has the following properties
\begin{enumerate}[label=(\alph*)]
    \item (Exogenous contexts).  The contexts $\{C_{t}\}_{t=1}^{T}$ are $i.i.d.$ with $C_{t} \sim P_{\mathcal{C}}$ and $C_{t} \indep \mathcal{F}_{t-1}$ for all $t \in [T]$. 
    \item (Fixed conditional reward distribution). There exists a conditional distribution $P_{Y}(\cdot \mid c, a)$ such that for all $t \in [T]$,
$$Y_{t} \mid C_{t}, A_{t} \sim P_{Y}(\cdot \mid C_t,A_t),$$
and $Y_{t} \indep \mathcal{F}_{t-1} \mid  C_t,A_t$.
\end{enumerate}
\end{assumption}

Our goal will be to conduct inference on functionals of a non-parametric $M$-estimand. We mirror the framework taken in \cite{vanderlaan2026automaticdebiasedmachinelearning} but with the target and inference procedures adapted to the sequential, adaptive setting. Let $\Pi_{e}(\cdot\mid c)$ be an evaluation distribution that we assume to be \emph{fixed} and independent of history. We denote $P_{e}$ as the corresponding law where $C \sim P_{C}$, $A\sim \Pi_{e}(\cdot \mid C)$, and $Y \mid C,A \sim P_Y(\cdot \mid C,A)$. We note that at each $t$, $P_{e}$ shares the same marginal distribution of contexts and outcome mechanism as $P_{t}$, differing only in the distribution of actions. 
\begin{assumption} \label{assumption:densities}
There exists a common dominating measure $\mu$ such that for all $t$, $\Pi_{t}(\cdot\mid c), \Pi_{e}(\cdot \mid c) \ll \mu(\cdot \mid c)$ with corresponding densities $\pi_{t}(\cdot \mid c)$ and $\pi_{e}(\cdot \mid c)$. We assume these densities are known and $\Pi_e(\cdot\mid c)\ll\Pi_t(\cdot\mid c)$ almost surely.
\end{assumption}
Define a loss function $\ell : \ \Theta \times \Upsilon \times \mathcal{Z} \rightarrow \R$ written as $\ell_{\theta,\eta}(z)$, where $(\Theta,\norm{\cdot}_{\Theta})$ and $(\Upsilon,\norm{\cdot}_{\Upsilon})$ are normed linear spaces. Define the $M$-estimand,
\begin{equation} \label{eqn:m-estimand}
\theta_{P_e} := \argmin_{\theta \in \Theta} \E_{P_{e}} \left[  \ell_{\theta,\eta_{P_{e}}}(Z)\right],
\end{equation}
where $\eta_{P_e}$ is a nuisance function. Our parameter of interest will be a smooth functional of the $M$-estimand defined as 
\begin{equation} \label{eqn:target}
    \Psi_{e}(\theta_{P_{e}}) := \E_{P_{e}}\left[m_{\theta_{P_e}} (Z) \right],
\end{equation}
for some measurable function $m : \Theta \times \mathcal{Z}  \to \R$. We note that this framework is general enough to cover most target estimands studied in the adaptive inference literature, including: projection parameters for linear models and parametric $M$-estimators, partial linear regression parameters, and average treatment effect. See \cref{tab:examples} and Appendix~\ref{appendix:prior} for an extended discussion.  

%% file: methodology.tex
We first derive the influence function at $P_{e}$ and establish a local first-order expansion of $\Psi_e$.
\begin{assumption}[Regularity  of the loss at evaluation law] \label{assumption:reg_pe}
We assume that $\theta_{P_{e}}$ as defined in \cref{eqn:m-estimand} exists and is unique. Moreover, the following conditions hold:
\begin{enumerate}[label=(\alph*)]
\item (Differentiability) The loss $\ell_{\theta,\eta}(z)$ is twice Fr\'echet differentiable in
$\theta$ in a neighborhood of $\theta_{P_e}$ and the functional $m_\theta(z)$ is once Fr\'echet differentiable in $\theta$ in a neighborhood of $\theta_{P_e}$. \label{cond:differentiable}
\item (First order condition).  $\E_{P_e}\!\left[
\partial_{\theta} \ell_{\theta_{P_e},\eta_{P_e}}(Z)[h]
\right] = 0$ for all $h \in \Theta$.\label{cond:first_order}
\item (Neyman orthogonality) $  \label{cond:neyman}
\E_{P_e}\!\left[\partial_\eta \partial_\theta \ell_{\theta_{P_e},\eta_{P_e}}(Z)[h,u]\right] = 0$ for all $h \in \Theta$ and $u \in \Upsilon$.
\item (Positive definite Hessian) Let $H_{P_e}(u,h):=
\E_{P_e}\!\left[\partial_{\theta}^{2} \ell_{\theta_{P_e},\eta_{P_e}}(Z)[h,u]\right]$ for $h,u \in \Theta$. We assume there exists a constant $c>0$ such that $c\|h\|_\Theta^2\le H_{P_e}(h,h)$
for all $h \in \Theta$. \label{cond:hessian}
\end{enumerate}
\end{assumption}
Condition~\ref{cond:differentiable} guarantees sufficient smoothness of the loss and functional so that a first order expansion can be taken. Condition~\ref{cond:first_order} is a standard first order condition defining the target estimand. Condition~\ref{cond:neyman} imposes Neyman orthogonality, which ensures that first-order perturbations in the nuisance parameter do not affect the target. We note that this is not a substantial restriction, as most common loss functions can be orthogonalized without changing the underlying risk. Condition~\ref{cond:hessian} requires the Hessian be positive definite, ensuring local identification, and that the curvature of the loss induces a well-behaved inner product structure. Finally, we assume an envelope condition (\cref{assumption:envelope}) to bound higher order remainder terms in our first order expansion but defer its statement to Appendix~\ref{subsec:deferred_assumptions}.

Under Assumptions~\ref{assumption:reg_pe} and~\ref{assumption:envelope}, the bilinear form $H_{P_e}$ defines an inner product on $\Theta$ and the map $\E_{P_e}\!\left[\partial_{\theta} m_{\theta_{P_e}}(Z)[h]\right]$ is a bounded linear functional with respect to the induced norm.Write $\|h\|_{H_{P_e}}^2:=H_{P_e}(h,h)$ and let $\bar{\Theta}$ be the completion of $\Theta$ in this norm. Lemma~\ref{lemma:hilbert_completion} justifies extending the original norm and the derivative maps in their direction arguments to $\bar{\Theta}$. By the Riesz representation theorem, there is a unique element $\alpha_{P_e}\in\bar{\Theta}$ that solves the Hessian equation
\[E_{P_e}\!\left[\partial_{\theta} m_{\theta_{P_e}}(Z)[h]\right] =H_{P_e}(\alpha_{P_e},h) \text{ for all } h \in \Theta.\]
In particular, 
\[
\alpha_{P_e}
=
\arg\min_{\alpha\in\bar{\Theta}}
\left\{\frac{1}{2}
\E_{P_e}\!\left[\partial_{\theta}^{2} \ell_{\theta_{P_e},\eta_{P_e}}(Z)[\alpha,\alpha]\right] - E_{P_e}\!\left[\partial_{\theta} m_{\theta_{P_e}}(Z)[\alpha]\right]
\right\}.
\]
An extended discussion of this is deferred to \cref{appendix:proofs}. This representation allows us to express first-order changes in $\Psi$ in terms of the curvature of the loss, and forms the basis for the expansion below.

\begin{theorem}[One-step expansion at the evaluation law]
\label{thm:pe_expansion}
Suppose Assumptions~\ref{assumption:reg_pe}-\ref{assumption:envelope} hold. Let
$
(\bar\theta,\bar\eta,\bar\alpha)\in \Theta\times\Upsilon\times\Theta$
be in a neighborhood of
$(\theta_{P_e},\eta_{P_e},\alpha_{P_e})$. Let $
\delta_{\mathrm{lin}} := \mathbf{1}\{m_\theta \text{ is not linear in }\theta
\text{ at } \theta_{P_e}\}$ and $\delta_{\mathrm{quad}} := \mathbf{1}\{\ell_{\theta,\eta_{P_e}}(\cdot) \text{ is not quadratic at } \theta_{P_e}\}$. Then,
\begin{align*}
\Psi_e(\bar\theta)-\Psi_e(\theta_{P_e})
&=
 \E_{P_e}\left[\partial_\theta\ell_{\bar\theta,\bar\eta}[\bar\alpha]\right]
-
H_{P_e}[\bar\alpha-\alpha_{P_e},\,\bar\theta-\theta_{P_e}]  \\
& \qquad +  (\delta_{\mathrm{lin}} + \delta_{\mathrm{quad}})
O\!\left(
\|\bar\alpha\|_\Theta
\|\bar\theta - \theta_{P_e}\|_\Theta^2
\right)\\
& \qquad +
O\left(\|\bar\alpha\|_\Theta\left(\|\bar\theta-\theta_{P_e}\|_\Theta
\|\bar\eta-\eta_{P_e}\|_\Upsilon+\|\bar\eta-\eta_{P_e}\|_\Upsilon^2\right)\right).
\end{align*}

\end{theorem}
Motivated by this expansion, we denote 
\begin{equation} \label{eqn:oracle_EIF}
\varphi_{P_e}(Z)
:=
m_{\theta_{P_e}}(Z)
-
\Psi(P_e)
-
\partial_\theta \ell_{\theta_{P_e},\eta_{P_e}}(Z)[\alpha_{P_e}].
\end{equation}
By construction, $\E_{P_e}[\varphi_{P_e}(Z)] = 0$. We note that this function coincides with the efficient influence function of $\Psi$ at $P_e$ in the nonparametric model. 
\iffalse
Let $\{P_\varepsilon : \epsilon \in (-\varepsilon_0,\varepsilon_0)\}$ be a regular parametric submodel passing through $P_e$ at $\varepsilon=0$, with densities $p_\varepsilon$ with respect to a common dominating measure. Let
$
s(Z)
:=
\left.\frac{d}{d\epsilon} \log p_\epsilon(Z)\right|_{\epsilon=0}$
denote the associated score function, which satisfies $E_{P_e}[s(Z)] = 0$. It follows that
\[
\E_{P_e}[\varphi_{P_e}(Z)\, s(Z)]
=
\left.\frac{d}{d\varepsilon}\Psi(P_\varepsilon)\right|_{\varepsilon=0},
\]
so that $\varphi_{P_e}$ represents the pathwise derivative.
\fi 
\paragraph{Change of measure from adaptive data to evaluation policy} 
\label{subsec:adaptive_weights}
We move now to writing a first order expansion for an \emph{empirically realizable} functional, obtained by replacing population expectation under $P_{e}$ with an empirical operator. Consider estimating $\Psi_{e}(\theta_{P_e})$ by taking a weighted average over the sample, for instance, $\frac{1}{T}\sum_{t=1}^{T}w_{t}m(\bar{\theta})$ for some $\bar{\theta} \in \Theta$. Then,
\begin{align*}
\left(\frac{1}{T}\sum_{t=1}^{T} w_{t} m_{\bar\theta}(Z_t)\right) - \Psi_{e}(\theta_{P_e}) &= \underbrace{\left(\frac{1}{T}\sum_{t=1}^{T} w_{t} m_{\bar\theta}(Z_t)\right) - \Psi_{e}(\bar{\theta})}_{\text{Empirical process term}} + \underbrace{\left(\Psi_{e}(\bar{\theta}) - \Psi_{e}(\theta_{P_e})\right)}_{\text{Plug-in error}}.\\ 
\end{align*}
The second term can be analyzed by applying \cref{thm:pe_expansion}. For the first term, we note that it introduces first order bias unless $\E\left[\frac{1}{T}\sum_{t=1}^{T} w_{t} m(\bar{\theta})(Z_t)\right] = \E_{P_e}[m(\bar{\theta})(Z_t)]$. To eliminate this bias, we can choose $w_{t} := \frac{\pi_{e}(A_t|C_t)}{\pi_{t}(A_t|C_t)}$ as the inverse propensity weights because for any integrable function $f$, 
\begin{align*}
\E\left[ w_{t}f(C_t,A_t,Y_t) | \mathcal{F}_{t-1}\right] &= 
\E\left[\E\left[w_{t}f(C_t,A_t,Y_t) |\mathcal{F}_{t-1}, C_{t}\right]\mid\mathcal F_{t-1}\right] \\
&= \int f(c,a,y)P_Y(dy\mid c,a)\pi_e(a\mid c)\mu(da\mid c)P_C(dc) \\
&=\E_{P_{e}}\left[ f(C,A,Y)\right]. 
\end{align*}
This allows us to replace any expectation involving $f$ defined relative to the evaluation policy, such as $\E_{P_e}[f(Z)]$, with the empirically realizable form $\frac{1}{T}\sum_{t=1}^T w_t \E \left[f(Z_t)  \mid \Filt \right]$.

\subsection{One-step estimator and asymptotic normality results} \label{subsec:onestep}
Our asymptotic results rely on the fact that a suitably rescaled influence function admits a martingale structure. Let $\varphi_{t} := \varphi_{P_e}(Z_t)$ and note that $\E[w_t \varphi_t \mid \mathcal{F}_{t-1}] = P_e \varphi_{P_e} = 0$. This implies that $\{w_{t} \varphi_t\}_{t=1}^T$ is a martingale difference sequence with respect to the filtration $\{\mathcal F_{t-1}\}_{t=1}^{T}$. Define
\begin{equation} \label{eqn:cond_var}
   \sigma_t^2
:=
\Var(w_{t}\varphi_t\mid \mathcal F_{t-1})
=
\E[w_{t}^{2}\varphi_t^2\mid \mathcal F_{t-1}].
\end{equation}
If $T^{-1}\sum_{t=1}^{T}\sigma_{t}^{2} \overset{p}{\to} \sigma^{2}$ for a deterministic constant $\sigma^{2}>0$, one may apply a martingale central limit theorem \citep{3f41bc42-f438-35b7-8504-3723ce0fea5f,dvoretzky1972,hall2014martingale} to show that under regularity conditions $T^{-1/2} \sum_{t=1}^{T} w_{t} \varphi_{t} \overset{d}{\to}N(0,\sigma^{2})$. However, in adaptive settings, the conditional variance often will not converge to a deterministic constant, rendering naive variance estimation inappropriate. Alternatively, one can define the martingale difference sequence as $\zeta_{t} := \sigma_{t}^{-1}w_{t} \varphi_{t}$ and construct a \emph{sequence} of plug-in estimators $\{\hat{\sigma}_{t}\}_{t=1}^{T}$ targeting the conditional variance. This strategy of stabilizing observations using conditional variance estimates was initially used in the non-adaptive setting by \cite{10.1214/15-AOS1384} and later work has applied it to martingale difference sequences \citep{2019hadad,kato2020standardized,bibault2021postcontext,leiner2026adaptiveoffpolicyinferencemestimators}. Of course, constructing such estimators is nontrivial; we postpone this to \cref{subsec:cond_var}.

Supposing that a suitable estimate of $\hat{\sigma}_{t}$ exists, let $(\hat\theta_t, \hat\eta_t, \hat \alpha_t)$ denote estimators of $(\theta_{P_e}, \eta_{P_e}, \alpha_{P_e})$. Letting $w_{t} = \frac{\pi_{e}(A_{t}|C_t)}{\pi_{t}(A_t|C_t)}$ and $\hat{A}_{T} = \sum_{t=1}^T \hat \sigma_t^{-1} w_t $, we define our one-step estimator as
\begin{equation} \label{eqn:onestep}
\hat\Psi_T^{\mathrm{os}}
:=\hat{A}_{T}^{-1}
\sum_{t=1}^T \hat \sigma_t^{-1} w_t \left(m_{\hat\theta_t}(Z_t)-\dot\ell_{\hat\theta_t,\hat\eta_t}(Z_t)[\hat \alpha_t]\right).
\end{equation}
One-step estimation is a standard semiparametric technique for debiasing naive plug-in estimators. The correction term involving $-\dot\ell_{\hat\theta_t,\hat\eta_t}(Z_t)[\hat \alpha_t]$ arises from the first-order expansion in \cref{thm:pe_expansion} and removes the leading bias induced by nuisance estimation. After reweighting and stabilization, the estimator behaves asymptotically like an average of mean-zero influence function increments; we formalize this asymptotic normality result in \cref{thm:clt}. 
\begin{theorem}[CLT for one-step estimator] \label{thm:clt}
Let $(\hat\theta_t, \hat\eta_t, \hat \alpha_t,\hat{\sigma}_{t})$ denote $\mathcal{F}_{t-1}$ measurable estimators of $(\theta_{P_e}, \eta_{P_e}, \alpha_{P_e}, \sigma_t)$ and define $\hat\Psi_T^{\mathrm{os}}$ as in \cref{eqn:onestep}. Recall that $w_{t} = \frac{\pi_{e}(A_{t}|C_t)}{\pi_{t}(A_t|C_t)}$ and  $\varphi_{t} := \varphi_{P_e}(Z_t)$ as defined in \cref{eqn:oracle_EIF}. Define $\hat{A}_{T} := \sum_{t=1}^{T}\hat\sigma_{t}^{-1} w_t $. Assume
\begin{enumerate}[label=(\thetheorem\alph*), ref=(\thetheorem\alph*), itemsep=0pt, topsep=2pt]
\item $\frac{1}{T}\sum_{t=1}^T \|\hat\theta_t-\theta_{P^e}\|_\Theta^4=o_p(T^{-1})$, $\frac{1}{T}\sum_{t=1}^T \|\hat\eta_t-\eta_{P^e}\|_\Upsilon^4=o_p(T^{-1})$, and $\frac{1}{T}\sum_{t=1}^T \|\hat \alpha_t-\alpha_{P_e}\|_\Theta^4=o_p(T^{-1})$;  \label{cond:rates}
\item For any $\epsilon >0$, $\frac{1}{T} \sum_{t=1}^T
\E\!\left[w_t^{2}\varphi_t^2\mathbf{1}\{|w_t\varphi_t| > \varepsilon \sqrt{T}\}
\mid \mathcal{F}_{t-1}\right] \xrightarrow{p} 0$; \label{cond:lindeberg}
\item $\sup_{t\in [T]}\sup_{{c,a} \in \mathcal{C}\times \mathcal{A}}  \frac{\pi_{e}(a\mid c)}{\pi_t(a\mid c)} = O(T^{1/3})$ almost surely; \label{cond:weight}
\item $\frac{1}{T}\sum_{t=1}^{T}  \sigma_{t}^{2}/\hat{\sigma_{t}}^{2} -1 =o_{p}(1)$ and    $\min_{t\in [T]}\hat{\sigma}_{t}^{2} >c$ for some $c>0$ almost surely. \label{cond:consistency}
\end{enumerate}
Then under Assumptions~\ref{assumption:distribution}--\ref{assumption:envelope},
\begin{equation}
    \frac{ \hat{A}_T}{\sqrt T}(\hat\Psi_T^{\mathrm{os}} - \Psi_e(\theta_{P_e})) \xrightarrow{d} N(0,1). 
\end{equation}
\end{theorem}
\begin{remark}
In the case where the loss is quadratic and the functional is linear with respect to $\theta$ (i.e. $\delta_{\text{lin}} = \delta_{\text{quad}}=0$ as defined in \cref{thm:pe_expansion}) and there is no nuisance parameter (e.g. $\Upsilon =\{ 0\}$), condition~\ref{cond:rates} can be simplified to a requirement that  $\frac{1}{T}\sum_{t=1}^{T}\E_{P_e}\!\left[\partial_{\theta}^{2} \ell_{\theta_{P_e},\eta_{P_e}}(Z)[\hat\theta_{t} - \theta_{P_e},\hat\alpha_{t} - \alpha_{P_e}]\right] =o_p(T^{-1/2})$.
\end{remark}
Requiring that all estimators be $\mathcal{F}_{t-1}$ measurable is natural in the sequential setting and ensures that the empirical process terms arising from parameter and nuisance estimation errors form martingale difference sequences. %so that a martingale central limit theorem can be applied.  
In the absence of this condition, one would instead need uniform control over the underlying function class, which would require stronger assumptions, such as bracketing entropy bounds. However, such conditions are often difficult to verify for flexible, nonparametric estimators, which are explicitly allowed in our framework.

Condition~\ref{cond:rates} is stated in terms of time-averaged higher-order moments of the estimation errors. This reflects the fact that, in the sequential setting, the analysis requires control of products of time-indexed estimation errors. The stated moment conditions provide a convenient sufficient condition for controlling such products via repeated applications of the Cauchy--Schwarz inequality. Intuitively, the condition should be interpreted as requiring that large estimation errors arising early in the sequential process do not dominate the time average. In practice, we recommend using an initial burn-in period of length $T_0=\lfloor aT\rfloor$ for some fixed $a\in(0,1)$ to train pilot estimates and excluding these observations when computing $\hat\Psi_T^{\mathrm{os}}$. A simple implementation is then to freeze the nuisance estimators after the burn-in period by setting $(\hat\theta_t,\hat\eta_t,\hat\alpha_t):=(\hat\theta_{T_0+1},\hat\eta_{T_0+1},\hat\alpha_{T_0+1})$ for all $t>T_0$ and then computing $\hat\Psi_T^{\mathrm{os}}$ using only observations from times $T_0+1$ through $T$. In this special case, condition~\ref{cond:rates} reduces to the more familiar requirement that each estimator converges at rate $o_p(T^{-1/4})$. More generally, nuisance estimators may be updated periodically after the burn-in period provided the resulting estimators continue to satisfy \ref{cond:rates}. While updating the nuisance estimators at every time step is theoretically valid under these conditions, doing so may be computationally burdensome in practice. We therefore recommend batch updates as a practical compromise between statistical and computational efficiency.

Condition \ref{cond:lindeberg} is a standard Lindeberg-type condition ensuring that no single observation dominates the asymptotic variance. In the present setting, this amounts to controlling the tails of the increments $\varphi_{t}$. This condition can be verified under mild moment assumptions, for example by requiring either bounded increments of the efficient influence function or a bounded fourth moment. Condition~\ref{cond:weight} requires that the exploration rate  does not diverge from the evaluation policy too quickly.

An intermediate regime arises when the conditional variance converges to a random variable that depends on the realized trajectory but not on the horizon $T$. In this setting, we are able to avoid per-round variance estimation entirely and can instead use the realized quadratic variation of the influence function in order to create an asymptotically normal test statistic. 

\begin{theorem}[CLT for self-normalized one-step estimator]
\label{thm:self-normalized} Let $(\hat\theta_t, \hat\eta_t, \hat \alpha_t)$ denote $\Filt$ measurable estimators of $(\theta_{P_e}, \eta_{P_e}, \alpha_{P_e})$ Choose $m_{T} \to\infty$ with $m_{T}= o(T)$ and set $(\bar\theta,\bar\eta,\bar\alpha) :=(\hat\theta_{m_T},\hat\eta_{m_T},\hat\alpha_{m_T})$. Recall that $w_{t} = \frac{\pi_{e}(A_{t}|C_t)}{\pi_{t}(A_t|C_t)}$ and  $\varphi_{t} := \varphi_{P_e}(Z_t)$ as defined in \cref{eqn:oracle_EIF}
and set $\bar{\Psi}_{T} = m_T^{-1}\sum_{t=m_T+1}^{2m_T}w_{t}\left(m_{\bar\theta}(Z_t)-\partial_\theta \ell_{\bar\theta,\bar\eta}(Z_t)[\bar\alpha]\right)$. Define $A_T = \sum_{t=1}^{T}w_{t}$ and let 
\begin{equation} \small
    \hat{S}_{T} := A_{T}^{-1} \sum_{t=1}^{T} w_{t}\left( m_{\hat\theta_t}(Z_t)-\partial_\theta \ell_{\hat\theta_t,\hat\eta_t}(Z_t)[\hat\alpha_t]\right),
\end{equation}
\begin{equation}
\hat{V}_{T} :=A_{T}^{-2}  \sum_{t=2m_{T}+1}^{T}  w_{t}^{2}\left( m_{\hat\theta_t}(Z_t)-\partial_\theta \ell_{\hat\theta_t,\hat\eta_t}(Z_t)[\hat\alpha_t]-\bar{\Psi}_{T}\right)^{2}.
\end{equation}
\iffalse and 
\begin{equation}
    \hat{V}_{T} := \frac{1}{T-m_{T}} \sum_{t=m_{T}+1}^{T} w_{t}^{2}\left( m_{\bar\theta}(Z_t)-\partial_\theta \ell_{\bar\theta,\bar\eta}(Z_t) -\bar{\Psi}_{T}\right)^{2}. 
\end{equation} \fi
Assume conditions~\ref{cond:rates}, \ref{cond:lindeberg}, and \ref{cond:weight} of \cref{thm:clt} hold. Furthermore, assume that
\begin{enumerate}[label=(\thetheorem\alph*), ref=(\thetheorem\alph*), itemsep=0pt, topsep=2pt]
\item 
$\|\bar\theta-\theta_{P^e}\|_\Theta=o_p(m_{T}^{-1/4})$, $ \|\bar\eta-\eta_{P^e}\|_\Upsilon=o_p(m_{T}^{-1/4})$, and $ \|\bar \alpha-\alpha_{P_e}\|_\Theta=o_p(m_{T}^{-1/4})$;
\item $\sup_{t\in [T]}\sup_{{c,a} \in \mathcal{C}\times \mathcal{A}}  \frac{\pi_{e}(a\mid c)}{\pi_t(a\mid c)}  = o(m_{T}^{1/2})$ almost surely; \label{cond:exploration_self}
\item  $\frac{1}{T} \sum_{t=1}^{T} \E[ w_{t}^{2}\varphi_{t}^{2}\mid \Filt] \overset{p}{\to} \kappa$ to some random variable $\kappa$ which is finite and non-zero almost surely. \label{cond:finite_rv}
\end{enumerate}
Then under Assumptions~\ref{assumption:distribution}--\ref{assumption:envelope}, $
\hat{V}_{T}^{-1/2} \left(\hat{S}_{T} - \Psi_{e}(\theta_{P_e})\right) \overset{d}{\to} N(0,1)$.
\end{theorem}
Note that the first $2m_{T}$ observations are only used to create $\bar{\Psi}_{T}$ which centers the empirical variance and is not a primary driver of asymptotic rates. Condition~\ref{cond:exploration_self} implies that the rate $m_{T}$ is allowed to grow at is driven by the rate at which the importance weights are allowed to inflate. If the weights are uniformly bounded over time, $m_{T}$ can be chosen to grow at $\log \log T$ rates or slower. On the other hand, if weights grow at $T^{1/3}$ rates, $m_{T}$ must be chosen to grow at $T^{2/3}$.  We note a few example cases where condition~\ref{cond:finite_rv} is satisfied, but where $\kappa$ \emph{is not} a deterministic constant in the following examples. Intuitively, a sufficient condition for this to hold is that there is a single asymptotic regime that dominates over the sample with only sublinear rounds of adaptive sampling.

\begin{example}[Deterministic sublinear exploration]
Assume that an experimenter chooses the policy $\pi_{t}(\cdot |c)$ adaptively the first $T_0 = \lfloor T^{\alpha} \rfloor$ for $\alpha \in (0,1)$ rounds and then freezes it in all subsequent rounds. In this case, $\Pi_{t} = {\Pi}_{T_0}$ for all $t > T_0$. 
In this case, the conditional variance is constant for all $t>T_{0}$ and will converge to ${\sigma}_{T_0}^{2}$ so long as $\sigma_{t}$ is uniformly bounded across $t \in [T]$. Nonetheless, it is still a \emph{random quantity}. 
\end{example}

\iffalse
\begin{example}[Freeze policy after stopping time]
Suppose there exists an almost surely finite stopping time $\tau$ such that $\pi_{T} = \pi_{\tau}$ for all $t > \tau$. It follows that $\sigma_{t}^{2} = \sigma_{\tau}^{2}$ for all $t > \tau$. Thus, the variance is constant over time on the terminal block and will converge. Nonetheless, the policy is frozen according to a random quantity $\sigma_{\tau}^{2}$.
\end{example}
\fi 
\begin{example}[Policy switch at stopping times]
For each $T$, denote a set of stopping times such that
$0 < \tau_{1,T}< \tau_{2,T} < ... < \tau_{K_{T}} < T$. Suppose that policies change at stopping times and are piecewise constant between stopping times. That is $\Pi_{t} = \Pi_{\tau_{k,T}} $ for all $t \in [\tau_{k,T}, \tau_{k+1,T})$. In our setting, the variance within each block $[\tau_{k,T}, \tau_{k+1,T})$ is constant and almost surely bounded. Therefore, a sufficient condition for the asymptotic variance to be an almost surely finite random variable is that $T - \max_{k \in [K_{T}]} \lvert \tau_{k+1,T} - \tau_{k,T} \rvert =o_{p}(T)$.
\end{example}

\subsection{Estimating Conditional Variances} \label{subsec:cond_var}
In cases where condition \ref{cond:finite_rv} does not hold, whch can occur where there is adaptive sampling in every round $t\in [T]$, we must instead rely on \cref{thm:clt} which requires us to construct a sequence of $\mathcal{F}_{t-1}$ measurable estimators $\{\hat{\sigma}_{t}\}_{t=1}^{T}$. A natural strategy might be to directly estimate the functional,
$$ \text{Var}\left(w_{t} \varphi_{t }\mid \mathcal{F}_{t-1}\right) = \E\left[w_{t}^{2} \varphi_{t}^{2} \mid\mathcal{F}_{t-1}  \right] - \E[w_{t} \varphi_{t}\mid \mathcal{F}_{t-1}]^{2} = \E\left[w_{t}^{2} \varphi_{t}^{2} \mid\mathcal{F}_{t-1}  \right].$$
This approach would be burdensome because of the dependence of the estimand on several nuisance parameters: $\theta_{P_e}, \eta_{P_e}, \alpha_{P_e}, P_{C}, P_{Y|C,A}$. Instead, we construct $\hat\sigma_t$ by reweighting past observations to emulate draws from the evaluation policy at time $t$. This effectively reduces the problem to controlling a sequence of empirical averages over a policy class, which we handle via a maximal inequality.

We choose to  target the following sequence of observable quantities:
\begin{align*}
    \text{Var}\left(w_{t} \hat\varphi_{t }\mid \mathcal{F}_{t-1}\right) = \E\left[w_{t}^{2} \hat\varphi_{t}^{2} \mid\mathcal{F}_{t-1}  \right] - \E[w_{t} \hat\varphi_{t}\mid \mathcal{F}_{t-1}]^{2} 
= \E\left[w_{t}^{2} \hat\varphi_{t}^{2} \mid\mathcal{F}_{t-1}  \right] - \E_{P_e}[\hat\varphi_{t}]^{2}.
\end{align*} 
At the oracle EIF, $\E_{P_e}[\varphi_{t}] =0$ so the second term should be negligible. Our goal then is finding a sufficiently good estimate for $\E\left[w_{t}^{2} \hat\varphi_{t}^{2}\mid \mathcal{F}_{t-1}\right]$. Now, note that for any $s<t$, $\hat{\varphi}_{s}(\cdot)  :=
m_{\hat{\theta_s}}(\cdot)
-
\hat{\Psi}_{s}
-
\dot\ell_{\hat\theta_s,\hat\eta_s}(\cdot)[\hat \alpha_s])$ is composed of only $\mathcal{F}_{s-1}$ measurable known plug-in terms. Recalling that
\[
\left\{\, f\left(Z_s\right) - \E\!\left[f(Z_s)\mid \mathcal F_{s-1}\right] \,\right\}_{s=1}^t
\]

is a martingale difference sequence for any $\mathcal{F}_{s-1}$ measurable function $f$, consider some counterfactual policy $\pi$ and choose $f(z) :=\frac{ \pi_{s}(a \mid c)  }{ \pi(a \mid c)  } w_s^2\hat{\varphi}_{s}^{2}(z) =\frac{ \pi_{e}^{2}(a \mid c)  }{ \pi(a \mid c) \pi_{s}(a \mid c) } \hat{\varphi}_{s}^{2}(z) $. Intuitively, this map will change the measure to target the expectation of the reference policy because
\begin{align*}
    \E\left[f(Z_s) \mid \mathcal{F}_{s-1}\right] &:= \E\left[\frac{ \pi_{e}^{2}(A_s\mid C_s)  }{ \pi(A_s \mid C_s) \pi_{s}(A_s \mid C_s) } \hat{\varphi}_{s}^{2}(Z_s) \mid \mathcal{F}_{s-1}\right]\\
    &= \E_{A\sim \pi}\left[\frac{\pi_{e}^{2}(A \mid C) }{\pi^{2}(A \mid C) }\hat{\varphi}_{s}^{2}(Z) \right].
\end{align*}
With this motivation in mind, a tempting way to target the quantity $\E\left[w_{t}^{2} \hat\varphi_{t}^{2}\mid \mathcal{F}_{t-1}\right]$ would be to define the estimate
\begin{equation} \label{eqn:var_estimate}
    \hat{\sigma}^{2}_{t} : = \frac{1}{t-1}\sum_{s=1}^{t-1} \frac{ \pi_{e}^{2}(A_{s} \mid C_s)  }{ \pi_{t}(A_s \mid C_s) \pi_{s}(A_s \mid C_s) } \hat{\varphi}_{s}^{2}(Z_s).
\end{equation}
We justify the use of this estimator in \cref{thm:var_estimate}. The key technical challenge apart from needing to derive uniform bounds over the policy class is that we are using prior influence function increments $\hat\varphi_{s}$ to target $\hat\varphi_{t}$ so we need to demonstrate these quantities converge sufficiently fast. 
\begin{theorem} \label{thm:var_estimate}
Assume all conditions of \cref{thm:clt} are satisfied aside from condition~\ref{cond:consistency}. Assume there exists a $c>0$ such that $\min(\sigma_{t} ,\hat{\sigma}_{t}) >c$ for all $t \in[T]$ almost surely. Recall that $\hat{\varphi}_{s}(\cdot)  :=
m_{\hat{\theta_s}}(\cdot)
-
\hat{\Psi}_{s}
-
\dot\ell_{\hat\theta_s,\hat\eta_s}(\cdot)[\hat \alpha_s]$, where $\hat{\Psi}_{s}$ is any $\mathcal{F}_{s-1}$ measurable estimator such that $
\frac1T\sum_{s=1}^T(\hat\Psi_s-\Psi_e)^2=o_p(T^{-1/2})
$ and $\sup_{s \in[T]}|\hat\Psi_{s}| = O_p(1).$\footnote{Note that we provide an explicit construction of an estimator $\hat\Psi_t$ fulfilling these requirements in Appendix~\ref{appendix:cond_var_detail}.} Moreover assume that there  exists a class of policies $\Pi_{T}$ such that $\pi_t \in \Pi_{T}$ almost surely for all $t\in[T]$ and 
\begin{enumerate}[label=(\thetheorem\alph*), ref=(\thetheorem\alph*)]
\item \label{cond:thm4entropy} There exists $p \in [0,2)$ such that for all $\epsilon \in (0,1)$, 
\[ \log N(\epsilon, \log \Pi_{T}, \|\cdot\|_\infty)
\lesssim
\begin{cases}
\log(e/\epsilon), & p = 0, \\
\epsilon^{-p}, & p \in (0,2),
\end{cases},\] 
where $\log \Pi_{T} := \{ (c,a) \mapsto \log \pi(a \mid c) : \pi \in \Pi_{T}\}$.
\item $\sup_{\pi \in \Pi}\sup_{{c,a} \in \mathcal{C}\times \mathcal{A}} \frac{\pi_{e}(a\mid c)}{\pi(a\mid c)} =O(T^{1/8})$.  \label{cond:thm4overlap}
\end{enumerate}
Then, if  $\hat{\sigma}^{2}_{t}$ is constructed as in \cref{eqn:var_estimate} with $\hat{\sigma}_1^2$ initialized as a fixed positive finite constant, condition \ref{cond:consistency} is satisfied. That is, 
$\frac{1}{T}\sum_{t=1}^{T}  \sigma_{t}^{2}/\hat{\sigma_{t}}^{2} \overset{p}{\to} 1$. 
\end{theorem}
Condition \ref{cond:thm4entropy} controls the complexity of the policy class through  an entropy bound; in practice, such conditions are satisfied by policy classes driven by parametric or other low-complexity models.  We emphasize that this restrictions only applies to the \emph{policy} class, nuisance and parameter estimates can still rely on flexible black box approaches. Condition~\ref{cond:thm4overlap}'s stronger overlap requirement relative to Condition~\ref{cond:weight} is the price of uniform empirical-process control over $\Pi$. However, it remains weaker than the assumptions in previous adaptive $M$-estimation work \citep{zhang2021mestimators,guo2025statisticalinferencemisspecifiedcontextual,leiner2026adaptiveoffpolicyinferencemestimators}, which require importance weights that are uniformly bounded in time. %by a constant.

%% file: experiments.tex
\label{subsec:dynamic_pricing}
We discuss a dynamic pricing application as a prototypical example of the utility of our approach. We consider a sequential pricing problem in which, at each time period \(t \in [T]\), the decision-maker selects a price \(p_t \in \mathcal{A}\) and then observes a context \(C_t \in \mathcal{C}\) corresponding to customer attributes and a reward \(Y_t \in \{0,1\}\) corresponding to either customer renewal ($Y_t = 1$) or attrition ($Y_t=0$).\footnote{This setup differs slightly from the standard contextual bandit formulation, since the action space is restricted to context-independent pricing policies. We focus on this setting because fully personalized pricing (i.e., price discrimination based on customer attributes) is often infeasible in practice due to regulatory, fairness, or operational constraints.} We model renewal decisions using a random utility framework. Let $\varepsilon_{0t}, \varepsilon_{1t} \overset{iid}{\sim} \mathrm{Gumbel}(0,1)$ and define, 
\[
U_t^{1}
=
g^\star + \tau g_0(C_t) - \beta p_t + \varepsilon_{1t} \text{, } \qquad U_t^{0}
=
\varepsilon_{0t}.
\]
where $U_t^{1}$ denotes the utility of renewing and $U_t^{0}$ denotes the utility of attrition for individual $t$. The parameter \(\beta > 0\) governs price sensitivity and is the primary target of inference. The function \(g_0(X)\) captures baseline heterogeneity in demand, while \(g^\star\) controls the overall level of demand. $\varepsilon_{0t}$ and  $\varepsilon_{1t}$ are used to enforce idiosyncratic variation in the utility function of individuals. The observed outcome is 
\[
Y_t = \mathbf{1}\{U_t^{1} \ge U_t^{0}\} =\mathbf{1}\{g^\star + \tau g_0(C_t) - \beta p_t + \varepsilon_{1t} - \varepsilon_{0t}\ge 0 \}.
\]
The difference of two independent standard Gumbel random variables follows a logistic distribution. Consequently, $\varepsilon_{1t} - \varepsilon_{0t} \sim \mathrm{Logistic}(0,1)$, and the renewal probability is
$
\Pr(Y_t = 1 \mid C_t, p_t)
=
\Lambda\!\left(
g^\star + \tau g_0(C_t) - \beta p_t
\right)$,
where \(\Lambda(u) = (1 + e^{-u})^{-1}\) is the sigmoid function.
\paragraph{Sampling policies.}
We primarily consider two sampling policies. The first, which we label
\emph{revenue maximization}, uses past data to learn a price that maximizes
expected revenue over the population:
\[
p^\star
=\arg\max_{p \in \mathcal{A}}\E_{P_C}\left[
p \cdot \Lambda\big(g^\star + \tau g_0(C) - \beta p\big)\right].
\]
The decision-maker uses an $\epsilon$-greedy allocation: the price with
the highest estimated revenue receives probability $1-\epsilon$, while
the remaining probability $\epsilon$ is divided equally among the
other prices.

The second policy, which we term \emph{model auditing}, uses a planning
model with price sensitivity $\beta_{\rm ref}=1$. At each policy update,
the analyst tests $H_0:\beta=\beta_{\rm ref}$ at a nominal 5\%
significance level. If the null is rejected, prices are sampled
uniformly to collect broader evidence about demand. Otherwise, the
price recommended by the planning model receives probability
$1-\epsilon$, and the remaining probability $\epsilon$ is divided
equally among the other prices. We set $\epsilon=0.1$ throughout.
Further details and additional sampling rules appear in
Appendix~\ref{appendix:pricing_details}.

\paragraph{Implementation details.}
We use $\mathcal A=\{0.5,1,\ldots,5\}$ and independent attributes
$C_j=B_jG_j$, with $B_j\sim\mathrm{Bernoulli}(1/2)$ and $G_j\sim N(0,1)$.
The fitted model is
\[
 m_{\alpha,\beta}(C,p)=\Lambda\{H(C)^\top\alpha-\beta p\},\qquad
 H(C)=(1,C_1C_2,\sin(C_2)C_3,C_4^2,C_5C_6)^\top,
\]
estimating all coefficients while keeping price sensitivity constant.
For revenue maximization and auditing, true nuisance coefficients are
$\alpha_0=(2.5+\log(1.5),0.5,0.5,0.25,-0.125)^\top$.
Correct specification uses $\beta=1$; misspecification replaces it in the
true utility by $\beta(C)=b+1.2\tanh(C_1)$, with $b\approx1.426$.
All methods target the logistic-loss projection $\beta_e\approx1$ under
uniform evaluation prices. Some implementations uses different intercept
and sensitivity parameters that tie optimal revenues at prices 2 and 2.5,
preserving $\beta_e\approx1$.
Each setting has 2000 replications at $T=50000$. One uniform pilot uses
20\% of observations to fit and initialize the methods requiring it;
the remaining 80\% enter their estimating equations. GLM and IPW regression
use all $T$ observations.

\paragraph{Results.}
We compare equally weighted GLMs, IPW-weighted GLMs (labeled IPW),
AdapTZ~\citep{10.1214/24-AOS2485}, self-normalization
(\cref{thm:self-normalized}), and stabilized IPW
(\cref{thm:clt,thm:var_estimate}). Both GLM regressions use sandwich
standard errors. AdapTZ performs well under correct specification,
while ordinary IPW achieves near-nominal coverage under revenue
maximization even under misspecification. Model auditing induces
undercoverage for IPW, and heterogeneous demand also leads to
undercoverage for AdapTZ, outside its common-coefficient assumptions;
stabilized IPW maintains near-nominal coverage across these settings.
\begin{figure}[htbp]
\centering
\includegraphics[width=\linewidth]{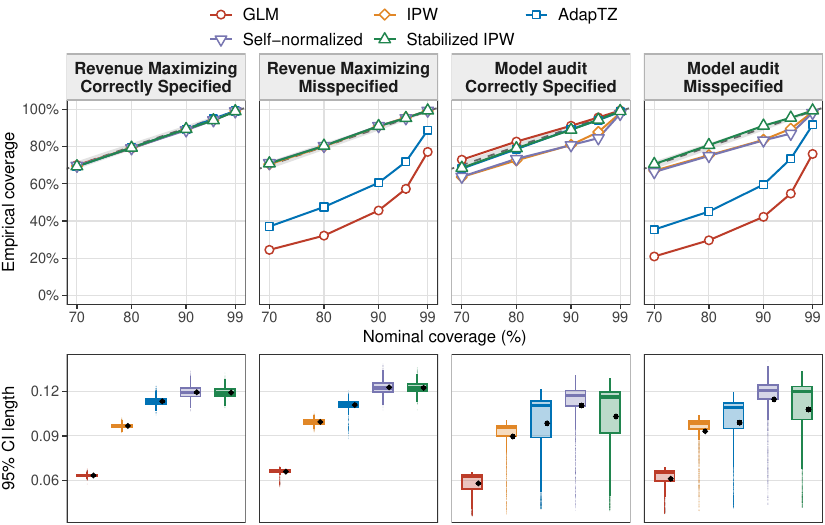}
\caption{Dynamic pricing: empirical coverage (top) and 95\% CI length
(bottom), $T=50000$, 2000 replications. Coverage bars are 95\% Wilson
intervals; shading marks two Monte Carlo standard errors around the diagonal.
CI panels share a scale and follow the legend order. Boxes show medians and
interquartile ranges; black diamonds and bars show means and 95\% Monte
Carlo intervals. All valid intervals are retained.}
\label{fig:coverage_pricing}
\end{figure}

%% file: conclusion.tex
We have developed a unified framework for inference on smooth functionals of nonparametric M-estimands under adaptive data collection. Our two asymptotic normality results---one based on a novel reweighted plug-in variance estimator and one based on self-normalized influence function increments---provide valid confidence intervals across a range of adaptive sampling regimes. Our results permit flexible estimation of nuisance parameters and are valid under model misspecification. 

Several directions remain open.  Generalizing beyond inverse propensity reweighting to handle arbitrary time-varying loss reweighting would allow us to unify a larger subset of work in the adaptive inference literature. Moreover, although we have shown that our constructions are asymptotically normal, it is an open question whether they achieve semiparametric efficiency bounds. Finally, our results assume that sampling policies are known, a natural question is under what conditions the sampling policies can be estimated while preserving asymptotic normality. 

\subsubsection*{Acknowledgments}
Aaditya Ramdas acknowledges the support of NSF Grant DMS-2610252. 

%% file: appendix.tex
\section{Common Target Estimands and Relationship to Prior Work} \label{appendix:prior}
In this section, we note how this framework can be used to target estimands considered in a variety of precursor work in the adaptive and contextual bandit settings. \cref{tab:positioning_review} summarizes some of the key differences in the assumptions of this work and precursor methods.

\subsection{Average Treatment Effect} \label{subsec:ATE}

Let $A \in \{0,1\}$ be a binary treatment. Let $\Theta := L_{2}(P_e)$ and the nuisance space be trivial ($\Upsilon = \{0\}$) as there is no nuisance parameter. Assume a square loss function with $\ell_{\theta}(c,a,y) =(y - \theta(a,c))^{2}$. Then the $M$-estimand is $\theta_{P_e}(a,c) := \E[Y|A=a,C=c]$. Letting $m_{\theta}(c,a,y) : = \theta(1,c) - \theta(0,c)$, the target becomes
$$ \Psi_{e}( \theta_{P_e}) = \E_{P_e}\left[m_{\theta_{P_e}}(C,A,Y)\right] = \E_{P_C}\left[ \theta_{P_e}(1,C) - \theta_{P_e}(0,C)\right].$$
This generalizes straightforwardly to any smooth functional of the conditional regression function, subsuming the frameworks of \cite{bibault2021postcontext} and \cite{2019hadad}.

\subsection{Parametric $M$-estimators}\label{subsec:M-estimator}
Let $\Theta = \R^{d}$ and the nuisance space be trivial (e.g. $\Upsilon = \{0\}$) as there is no nuisance parameter. Choose any loss $\ell_{\theta}(z)$ and let $\theta_{P_e} = \argmin \E_{P_e}\left[\ell_{\theta}(C,A,Y)\right]$. For any fixed $v \in \R^{d}$ (such as a basis vector $e_j$ to target a coordinate of $\theta$), we can then let $m_{\theta} = v^{T}\theta$ and consider $\Psi_{e}(\theta_{P_e}) = v^{T} \theta_{P_e}$. This subsumes the frameworks of \cite{leiner2026adaptiveoffpolicyinferencemestimators} and \cite{guo2025statisticalinferencemisspecifiedcontextual}. 

We note that the approach of \cite{zhang2021mestimators} cannot be put into this form as their construction uses square root inverse propensity weights. Under correct specification, the target estimand would be unchanged, but under misspecification, their procedure covers the time-varying random target 
$$\argmin_{\theta \in \Theta} \sum_{t=1}^{T}\E\left[\sqrt{\frac{\pi_{e}(A_t \mid C_t)}{\pi_{t}(A_t\mid C_t)}} \ell_{\theta}(Z_t) \mid C_t ,\Filt\right].$$
Generalizing our method to cover arbitrary time-varying weightings of the loss function would be an interesting extension for future work. 
\subsection{Partial Linear and Generalized Linear Regression} \label{subsec:partial_glm}
We partition the covariates $X = (W,V)$, where $W\in \R^{d}$ enters the model linearly and $V$ enters the model non parametrically. Consider the model
$$Y = g\left( W^{T} \theta  + \eta_{P_e}(V)\right) + \epsilon,$$
where $g: \R \to \R$ is a known inverse link function and $\E[\epsilon \mid C,A] = 0$. Then $\Theta = \R^{d}$ and $\Upsilon = L_{2}(P_C)$ and $\ell_{\theta,\eta}(c,a,y) = \left( y - g(w^{T}\theta + \eta(v))\right)^{2}$ where $(w,v)$ are components of (c,a). We then have $\theta_{P_e} = \argmin_{\theta \in \R^{d}} \E_{P_e} \left[\ell_{\theta,\eta_{P_e}}(C,A,Y)\right]$. For any fixed $v \in \R^{d}$, we can then let $m_{\theta} = v^{T}\theta$ and consider $\Psi_{e}(\theta_{P_{e}}) = v^{T} \theta_{P_e}$. This corresponds to the setting of \cite{10.1214/24-AOS2485}. 

\subsection{Value of an Optimal Policy over a Policy Class}  \label{subsec:optimal_policy}
We first start with a naive solution. Let $\Theta := L_{2}(P_e)$ and the nuisance space be trivial ($\Upsilon = \{0\}$) so there is no nuisance parameter.  Assume a square loss function with $\ell_{\theta}(c,a,y) =(y - \theta(a,c))^{2}$. As before, the $M$-estimand is $\theta_{P_e}(a,c) := \E[Y|A=a,C=c]$. Letting $\Psi_e(\theta_{P_e}) := \E_{C} \left[\max_{a\in \mathcal{A}} \theta_{P_e}(a,C) \right]$ correspond to the expected value of the optimal policy. Note that this functional is differentiable under a condition that ensures the optimal action given each context $C$ is unique almost surely. 

However, this construction can be restrictive in that it requires the optimal action to be unique almost surely. This condition can be relaxed by optimizing over a policy class. Let $\Pi$ denote a class of policies, define $\theta(a,c) = \E[Y \mid A=a, C=c]$, and consider the functional
\[
\Psi_e(\theta)
=
\sup_{\pi \in \Pi}
\E_{P_C} \left[
\int_{\mathcal{A}} \pi(a \mid C)\,\theta(a,C)\,\mu(da)
\right].
\]
This formulation avoids requiring a unique optimal action almost surely. Instead, it suffices that the population optimization problem over $\Pi$ admits a unique maximizer $\pi^\star$.

Under standard regularity conditions, the envelope theorem implies that for any perturbation $h \in \Theta$,
\[
\partial_\theta \Psi_e(\theta)[h]
=
\E_{P_C} \left[
\int_{\mathcal{A}} \pi^\star(a \mid C)\, h(a,C)\,\mu(da)
\right].
\]
The Riesz representer $\alpha_{P_e}$ can therefore be defined by the relation
\[
\E_{P_C} \left[
\int_{\mathcal{A}} \pi^\star_{\theta}(a \mid C)\, h(a,C)\,\mu(da)
\right]
=
E_{P_e} \big[ \partial_\theta^2 \ell_{\theta_{P_e}}(Z)[\alpha_{P_e}, h] \big]
\quad \text{for all } h \in \Theta.
\]

Under the square loss $\ell_\theta(c,a,y) = (y - \theta(a,c))^2$, we have
\[
E_{P_e} \big[ \partial_\theta^2 \ell_{\theta_{P_e}}(Z)[\alpha, h] \big]
=
2\, E_{P_e}[\alpha(A,C)\, h(A,C)].
\]
Matching terms yields the Riesz representer
$
\alpha_{P_e}(a,c)
=
\frac{1}{2}\,\frac{\pi^\star(a \mid c)}{\pi_e(a \mid c)}$. The corresponding efficient influence function is therefore
\[
\varphi_{P_e}(Z)
=
\int_{\mathcal{A}} \pi^\star(a \mid C)\, \theta_{P_e}(a,C)\,\mu(da)
-
\Psi_e(\theta_{P_e})
+
\frac{\pi^\star(A \mid C)}{\pi_e(A \mid C)}
\big( Y - \theta_{P_e}(A,C) \big).
\]

We note that this construction can potentially be extended to settings where the optimal policy is non-unique, although the value functional need not be pathwise differentiable in this case. \cite{10.1214/15-AOS1384} characterize conditions for pathwise differentiability of the optimal policy value and develop inference procedures even when these conditions fail. More recently, \cite{whitehouse2026inferenceoptimalpolicyvalues} develop an approach based on softmax smoothing to obtain inference on optimal policy values without requiring uniqueness, while \cite{whitehouse2026inferencevaluepolicyimprovement} establish asymptotically normal inference for the value of policy improvement over a restricted policy class under adaptive sampling, also allowing non-unique optima.

% Requires: booktabs, tabularx, makecell
\begin{table}[t]
\centering
\small
\setlength{\tabcolsep}{2.8pt}
\renewcommand{\arraystretch}{1.08}
\begin{tabularx}{\linewidth}{
  @{}>{\raggedright\arraybackslash}p{.22\linewidth}
  >{\centering\arraybackslash}p{.135\linewidth}
  >{\centering\arraybackslash}p{.082\linewidth}
  >{\centering\arraybackslash}p{.07\linewidth}
  >{\centering\arraybackslash}p{.085\linewidth}
  >{\centering\arraybackslash}p{.17\linewidth}
  >{\raggedright\arraybackslash}X@{}
}
\toprule
Work & Target & Misspec. & Adaptive & Policy & Overlap & Other \\
     & functional & allowed & data & known & assumption & assumptions \\
\midrule

\citet{1982laiwei}
& linear coef.
& no & yes & no
& implicit
& design stability \\[2pt]

\citet{10.1214/15-AOS1384}
& optimal value
& n/a & no & no
& uniform
& variance estimable \\[2pt]

\citet{pmlr-v80-deshpande18a,khamaru2021near,NEURIPS2023_a399456a}
& linear coef.
& no & yes & no
& implicit
& online debiasing \\[2pt]

\citet{2019hadad}
& arm means / policy value
& n/a & yes & yes
& vanishing
& variance converges \\[2pt]

\citet{bibault2021postcontext}
& policy value / contrasts
& n/a & yes & yes
& vanishing
& variance estimable \\[2pt]

\citet{NEURIPS2020_6fd86e0a}
& linear coef.
& no & batched & no
& uniform 
& fixed no. of batches \\

\citet{zhang2021mestimators}
& $M$-est.\ param.
& no & yes & yes
& uniform
& square-root IPW \\[2pt]

\citet{zhang2023statistical}
& $Z$-projection
& yes & yes & yes
& uniform
& longitudinal; fixed horizon \\[2pt]

\citet{10.1214/24-AOS2485}
& partial-linear coef.
& no & yes & moments
& vanishing
& nuisance consistency \\[2pt]

\citet{10.1214/24-AOS2485}
& GLM coef.
& no & yes & moments
& vanishing
& reference-arm floor; nuisance rates \\[2pt]

\citet{guo2025statisticalinferencemisspecifiedcontextual}
& $Z$-projection
& yes & yes & yes
& uniform
& deterministic policy limit \\[2pt]

\citet{leiner2026adaptiveoffpolicyinferencemestimators}
& $M$-projection
& yes & yes & yes
& uniform
& variance estimable \\

\midrule
This paper
& smooth $M$-functionals
& yes & yes & yes
& vanishing
& variance estimable \\
\bottomrule
\end{tabularx}

\caption{Positioning relative to related work. The target-functional column summarizes both the estimand and the generality of the working model. The overlap column distinguishes uniform positivity from
permitted overlap that is permitted to vanish at some rate.
The final column highlights some key assumptions for each method.}
\label{tab:positioning_review}
\end{table}

\section{Additional Simulation Details}  \label{appendix:pricing_details}
This section describes the policy-update regimes, implementation, and
sampling rules used in the dynamic pricing experiments. 

\subsection{Different levels of adaptivity}
We distinguish the \emph{sampling rule}, which proposes a distribution
over prices using past data, from the \emph{update regime}, which determines
when that proposal replaces the current distribution. Let
$m=\lfloor0.2T\rfloor$ be the pilot size and $n=T-m$ the number of
post-pilot observations. Index post-pilot rounds by $k=1,\ldots,n$.
The first 50 rounds use uniform prices; the first policy update occurs
at $k=51$. We consider four regimes:
\begin{enumerate}
\item \textbf{Sublinear adaptation followed by commitment.}
Update at every round through $k=51+\lceil\sqrt n\rceil$, then hold the
sampling distribution fixed. The updating period occupies a vanishing
fraction of the post-pilot horizon. Although the policy eventually stops
changing within a trajectory, its committed distribution depends on the
observed history, so its associated variance can remain path dependent.
\item \textbf{Linear adaptation followed by commitment.}
Update at every round through $k=51+\lfloor n/2\rfloor$, then freeze the
distribution. Feedback now continues for a nonvanishing fraction of the
sample, allowing us to assess how inference changes when adaptation
persists well beyond an initial learning period.
\item \textbf{Piecewise constant policies with random updates.}
After the first update, independently refresh the policy with probability
$0.3$ each round; otherwise retain the preceding distribution. Data continue
to accumulate between updates. This creates locally fixed sampling
distributions with repeated opportunities to change them. ``Random'' in
the figures refers to these update times, not to uniformly random prices.
\item \textbf{Fully adaptive policies.}
Refresh the distribution before every post-warm-up observation using all
available past data. This permits ongoing feedback between estimation
and sampling throughout the experiment.
\end{enumerate}
An update recomputes the entire distribution, including its exploration
probability and favored price. Freezing the distribution does not stop
data collection or estimation. 
\subsection{Implementation details}\label{appendix:pricing_details}\label{app:v9_pilot}\label{app:v9_estimators}

\paragraph{Demand model and target.}
We use the price menu $\mathcal P=\{0.5,1,\ldots,5\}$ and independent
customer attributes $C_j=B_jG_j$, with
$B_j\sim\mathrm{Bernoulli}(1/2)$ and $G_j\sim N(0,1)$.
The fitted model is
\[
m_\theta(C,p)=\Lambda\{H(C)^\top\alpha-\beta p\},\qquad
H(C)=(1,C_1C_2,\sin(C_2)C_3,C_4^2,C_5C_6)^\top,
\]
where $\theta=(\alpha^\top,\beta)^\top$ and
$\Lambda(u)=(1+e^{-u})^{-1}$. We estimate all coefficients while keeping
price sensitivity constant. The true nuisance coefficients are
$\alpha_0=(a,0.5,0.5,0.25,-0.125)^\top$.
Correct specification uses $\beta=1$; misspecification replaces it in the
true utility by $\beta(C)=b+1.2\tanh(C_1)$.
All methods target the price-sensitivity coefficient $\beta_e\approx1$
of the logistic-loss projection under uniform evaluation prices.
For the estimation formulas below, write $X(C,p)=(H(C)^\top,-p)^\top$.

\paragraph{Computing environment.}\label{subsec:computing}
Simulations ran locally on an Apple M4 MacBook Air (10 CPU cores,
16 GB RAM; macOS 15.7.7), using R 4.6.0 and C++ via Rcpp 1.1.1.1.1.
Replications ran in up to 16 CPU worker processes with fixed seeds.

\subsection{Sampling policies}\label{app:v4_policy}\label{app:revenue_pricing}\label{app:v9_controls}
At an eligible update, each rule below proposes a price distribution.
The current price is then drawn from that distribution before observing
the customer's attributes or purchase outcome. Between updates, prices
are drawn from the previously stored distribution. Define
\[
U(p)=\frac1{10},\qquad
E(p;q)=\begin{cases}0.90,&p=q,\\0.10/9,&p\ne q.\end{cases}
\]
Let $N_{t-1,p}$ and $S_{t-1,p}$ count previous offers and purchases at
price $p$. The smoothed empirical revenue and its maximizer are
\[
\widehat r_{t-1}(p)=p\frac{S_{t-1,p}+1/2}{N_{t-1,p}+1},\qquad
\widehat p_t=\arg\max_{p\in\mathcal P}\widehat r_{t-1}(p).
\]
We add one-half to the purchase and non-purchase counts, shrinking
small-sample purchase rates toward $1/2$ and ensuring that the estimate
is defined even for prices with no observations. Counts include the pilot except for
confidence deployment and probability matching, which start new revenue
counts after the pilot. This reset affects their allocation decisions;
the pilot nuisance fits and variance history are retained.

The relevant sampling policies we test are enumerated below. 

\paragraph{Revenue maximization.}
The decision maker concentrates on the current empirical revenue leader
while retaining exploration at every price. At an update, set
$\pi_t(p)=E(p;\widehat p_t)$. Thus the favored price receives probability
exactly $0.90$, and each of the other nine prices receives $0.10/9$.
This is the conventional revenue-driven benchmark for the statistical
allocation rules that follow.

\paragraph{Model auditing.}
A retailer begins with a planning model that assumes sensitivity
$\beta_{\rm ref}=1$. Fit its nuisance coefficients on the pilot with
this sensitivity held fixed, obtaining $\widehat\alpha_{\rm ref}$, and
choose the planning price
\[
p_{\rm ref}=\arg\max_{p\in\mathcal P}\frac p m\sum_{i=1}^m
\Lambda\{H(C_i)^\top\widehat\alpha_{\rm ref}-\beta_{\rm ref}p\}.
\]
The benchmark is a demand coefficient, not a prespecified price.
After $j$ post-pilot observations, form the running monitor
\[
B_j=\frac{\sum_{s=1}^j w_{m+s}F_{m+s}}{\sum_{s=1}^j w_{m+s}},\qquad
s_j=\frac{\{\sum_{s=1}^j w_{m+s}^2(F_{m+s}-B_j)^2\}^{1/2}}
{\sum_{s=1}^j w_{m+s}},\qquad
Z_j=\frac{B_j-\beta_{\rm ref}}{s_j}.
\]
At round $k$, an eligible update sets $\pi_t=U$ if
$|Z_{k-1}|>\Phi^{-1}(0.975)$ and $\pi_t=E(\cdot;p_{\rm ref})$ otherwise.
The retailer therefore uses its planning price while the benchmark remains
compatible with the data, and broadens data collection when the model
appears unreliable. The rule can switch back after later reviews.
The monitor $B_j$ is an ordinary weighted one-step calculation; the final
baseline labeled IPW is the weighted regression described above.

\paragraph{Test-and-deploy.}
This rule represents a different statistical objective: continue broad
experimentation until the benchmark can be rejected, then deploy a price
selected from the accumulated revenue data. Using the same monitor,
set $\pi_t=U$ while $|Z_{k-1}|\le\Phi^{-1}(0.975)$, and
$\pi_t=E(\cdot;\widehat p_t)$ otherwise. The rule is reversible, so a
later review can return the retailer to uniform exploration.

\paragraph{Permanent deployment.}
Use the same test-and-deploy rule until its first rejection at an eligible
update. Thereafter freeze the selected policy. Later updates can
still change $\widehat p_t$ as revenue estimates change, but cannot restore
uniform exploration. This represents a one-time decision to end broad
experimentation while continuing to improve the deployed price.

\paragraph{Confidence-based deployment.}
Here deployment is based directly on uncertainty about revenue rankings.
For each price, use a Jeffreys--Beta purchase-rate posterior, approximating
the associated revenue by independent normal variables
\[
R_p\sim N\!\left(pu_p,\frac{p^2u_p(1-u_p)}{N_p+2}\right),\qquad
u_p=\frac{S_p+1/2}{N_p+1},\qquad
\gamma_p=\Pr(R_p\ge R_q\text{ for every }q).
\]
At an update, offer uniform prices if $\gamma_{\widehat p_t}<0.95$;
otherwise set
$\pi_t(p)=0.02/10+0.98\mathbf1\{p=\widehat p_t\}$.
The rule can return to exploration if confidence subsequently falls.
It models experimentation until there is sufficient evidence to deploy
a revenue leader. The tied-revenue environment makes identifying a unique
leader difficult, without requiring a policy to switch prices by design.

\subsection{Results across adaptivity levels and sampling policies}\label{app:v9_cross}
Figures~\ref{fig:v9_cross_revenue}--\ref{fig:v9_cross_permanent}
cross all sampling rules with the four update regimes and both correct and incorrect
specifications at $T=50000$, using 2000 replications per setting.
Each policy has separate Correctly Specified and Misspecified plots,
with update regimes as columns. Coverage curves appear above 95\% CI
lengths.  CI boxes show medians and
interquartile ranges; black diamonds and bars show means and 95\% Monte
Carlo intervals. 

\begin{figure}[htbp]
\centering
\includegraphics[width=\linewidth]{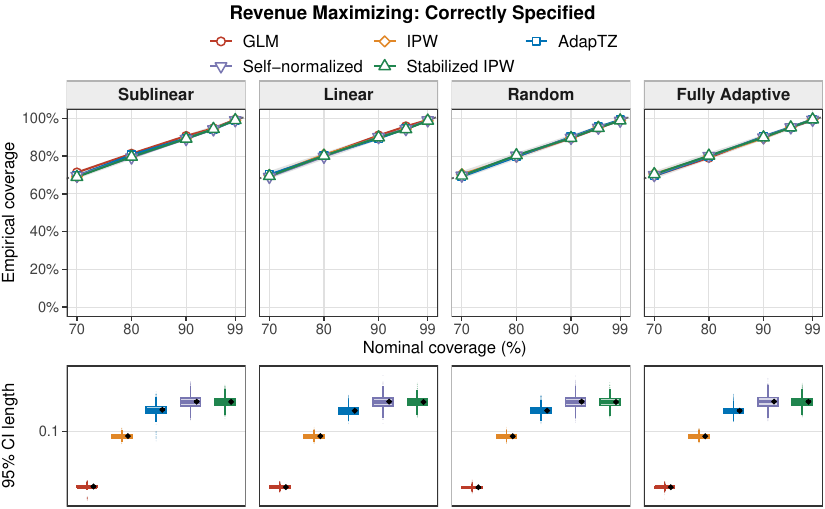}
\par\vspace{5pt}
\includegraphics[width=\linewidth]{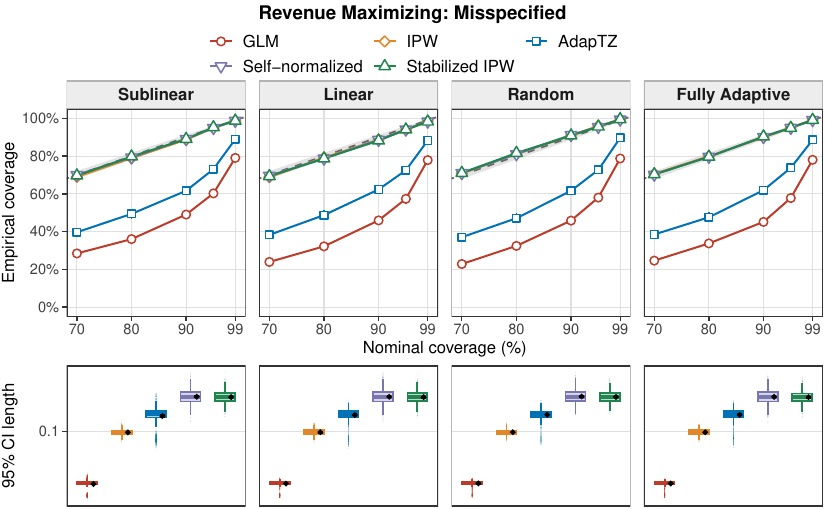}
\caption{Revenue Maximizing: Correctly Specified (top) and Misspecified (bottom), $T=50000$.}
\label{fig:v9_cross_revenue}
\end{figure}
\clearpage
\begin{figure}[htbp]
\centering
\includegraphics[width=\linewidth]{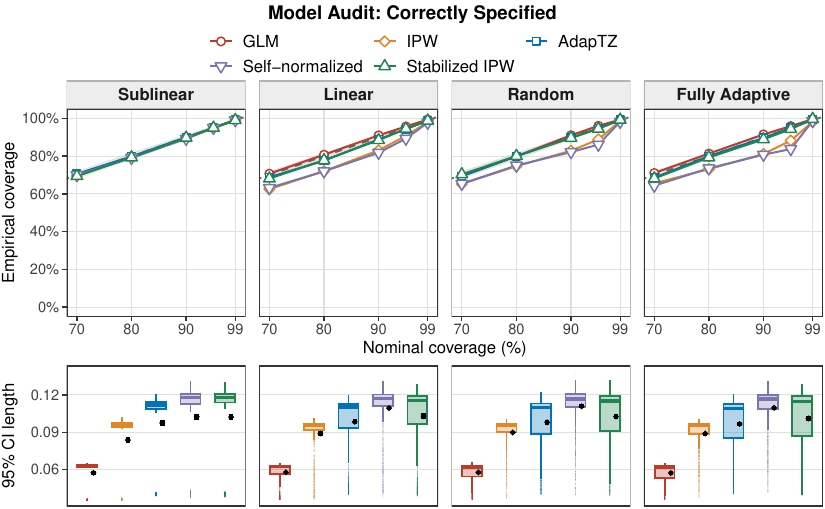}
\par\vspace{5pt}
\includegraphics[width=\linewidth]{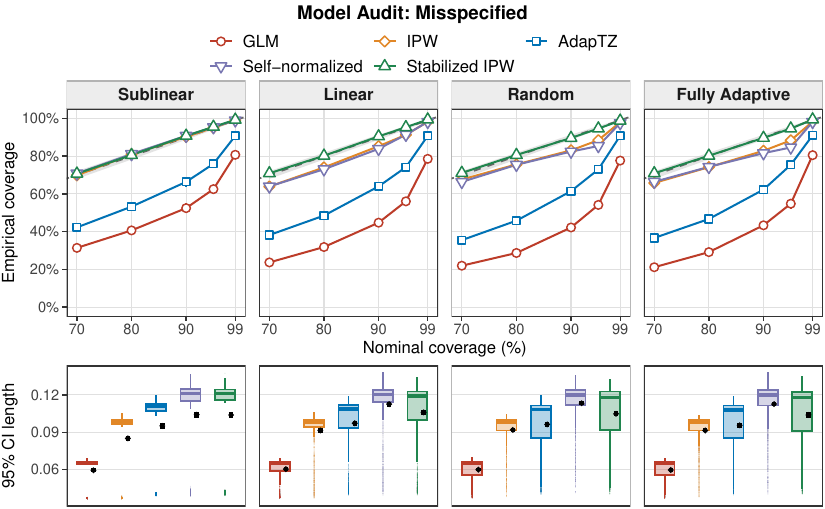}
\caption{Model Audit: Correctly Specified (top) and Misspecified (bottom), $T=50000$.}
\label{fig:v9_cross_audit}
\end{figure}
\clearpage
\begin{figure}[htbp]
\centering
\includegraphics[width=\linewidth]{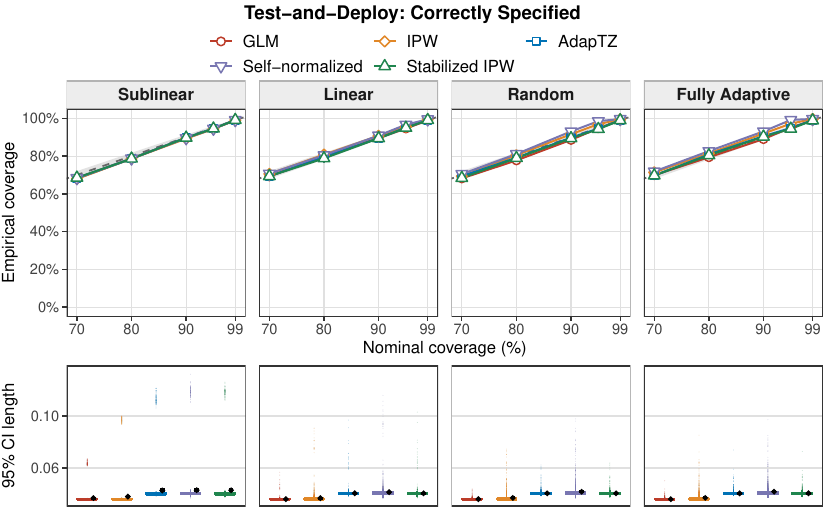}
\par\vspace{5pt}
\includegraphics[width=\linewidth]{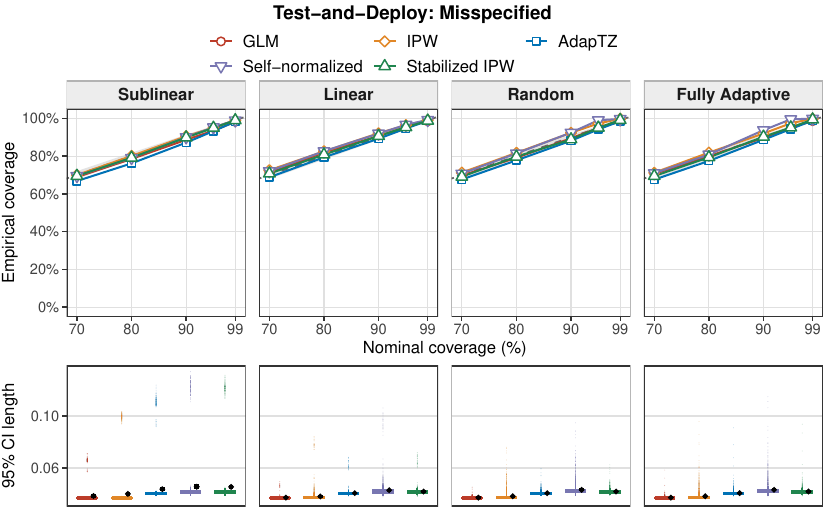}
\caption{Test-and-Deploy: Correctly Specified (top) and Misspecified (bottom), $T=50000$.}
\label{fig:v9_cross_deploy}
\end{figure}
\clearpage
\begin{figure}[htbp]
\centering
\includegraphics[width=\linewidth]{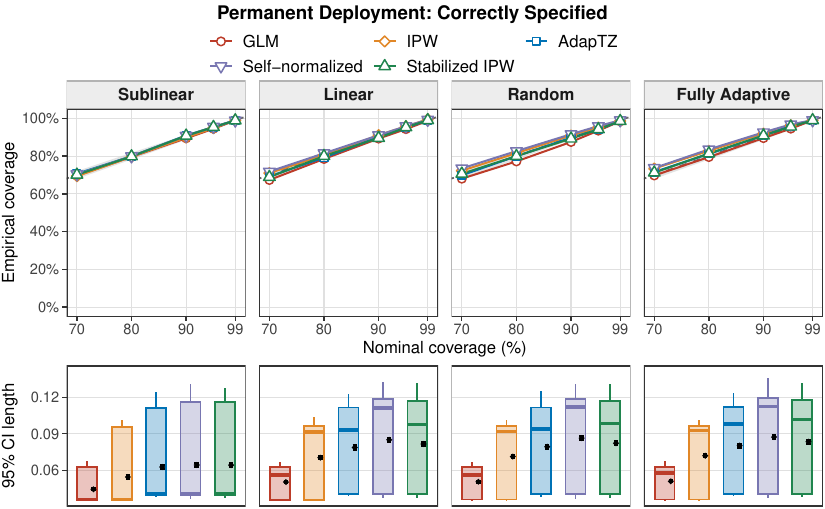}
\par\vspace{5pt}
\includegraphics[width=\linewidth]{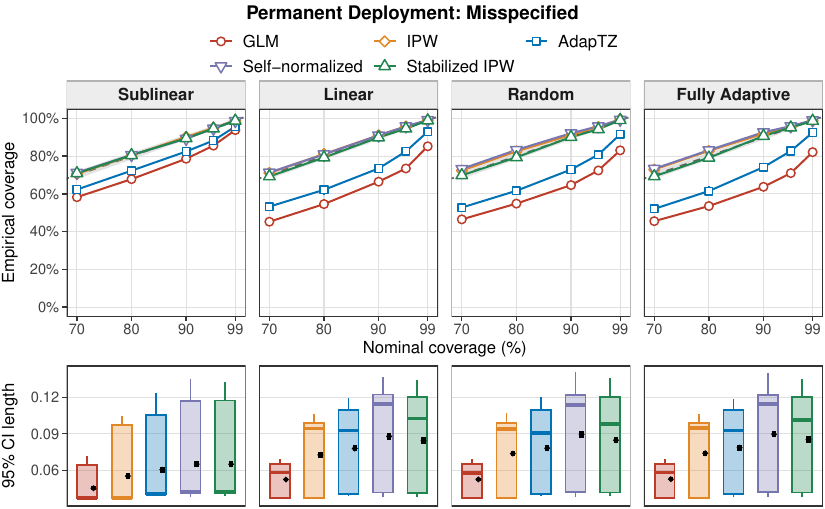}
\caption{Permanent Deployment: Correctly Specified (top) and Misspecified (bottom), $T=50000$.}
\label{fig:v9_cross_permanent}
\end{figure}

\section{Practical Considerations for Estimating Conditional Variances} \label{appendix:cond_var_detail}
We note that \cref{thm:var_estimate} requires the plug-in centering estimator \(\hat\Psi_t\) to satisfy mild boundedness and consistency conditions. A convenient way to guarantee these properties is through sample splitting. \cref{prop:plugin_samplesplit} shows that if the nuisance and parameter estimates are trained on an initial burn-in block and the centering term is estimated on a separate holdout block, then the required rate condition follows automatically under the regularity assumptions of the paper. This is not a particularly onerous requirement, as we already recommend a burn-in period when applying \cref{thm:clt}.

Moreover, if \(T_0\) is chosen to be sufficiently large, the resulting variance estimation problem becomes more stable. So long as the evaluation policy is chosen such that $\Pi_e(\cdot\mid c)\ll \Pi_t(\cdot\mid c)$, then $\hat\sigma_{t}^{2}$ will almost surely be bounded from below by a positive constant for all $t>T_0$. 

\begin{proposition}\label{prop:plugin_samplesplit}
Assume the conditions of \cref{thm:var_estimate}, except for the
requirements on the centering sequence $\hat\Psi_t$.
Let $T_0=\lfloor aT\rfloor$ and $T_1=\lfloor bT\rfloor$, where
$0<a<b<1$, and set $n=T_1-T_0$.
Suppose the $\mathcal F_{T_0}$-measurable pilot satisfies
$
\|\hat\theta_{T_0}-\theta_{P_e}\|_\Theta=o_p(T^{-1/4})$.
Define
\[
\hat\Psi
:=\frac{1}{T_1-T_{0}}\sum_{s=T_0+1}^{T_1}
w_s m_{\hat\theta_{T_0}}(Z_s).
\]
Then
\[
(\hat\Psi-\Psi_e)^2=o_p(T^{-1/2}).
\]
Consequently, setting $\hat\Psi_t=\hat\Psi$ for $t>T_1$
gives a predictable centering sequence on these rounds satisfying
\[
\frac1T\sum_{t=T_1+1}^T
(\hat\Psi_t-\Psi_e)^2=o_p(T^{-1/2}),
\qquad
\sup_{T_1<t\le T}|\hat\Psi_t|=O_p(1).
\]
\end{proposition}

\begin{proof}
Decompose
\[
\hat\Psi-\Psi_e
=
\left(\hat\Psi-\E_{P_e}[m_{\hat\theta_{T_0}}]\right)
+
\left(\E_{P_e}[m_{\hat\theta_{T_0}}]
      -\E_{P_e}[m_{\theta_{P_e}}]\right).
\]
By pilot consistency, the event
\[
\mathcal A_T
:=
\{\|\hat\theta_{T_0}-\theta_{P_e}\|_\Theta
   \le \epsilon_\Theta\}
\in\mathcal F_{T_0}
\]
has probability tending to one. We work on this event, where
\cref{assumption:envelope} applies.

Write
\[
X_s:=w_s m_{\hat\theta_{T_0}}(Z_s),
\qquad
\mu_0:=\E_{P_e}[m_{\hat\theta_{T_0}}].
\]
Because the pilot is fixed before the holdout block, change of
measure gives
\[
\E[X_s\mid\mathcal F_{s-1}]=\mu_0,
\qquad T_0<s\le T_1.
\]
Thus $\{X_s-\mu_0\}_{s=T_0+1}^{T_1}$ is a martingale
difference sequence. Its distinct increments are uncorrelated
conditionally on $\mathcal F_{T_0}$.

The overlap condition provides a bound $0\le w_s\le K_T$,
where $K_T=O(T^{1/8})$. Using $w_s^2\le K_Tw_s$
and change of measure again,
\[
\begin{aligned}
\E[X_s^2\mid\mathcal F_{s-1}]
&\le
K_T\E[
w_s m_{\hat\theta_{T_0}}(Z_s)^2
\mid\mathcal F_{s-1}]
\\
&=
K_T\E_{P_e}[m_{\hat\theta_{T_0}}^2]
\\
&\le
K_T\E_{P_e}[C(Z)^2].
\end{aligned}
\]
Consequently, on $\mathcal A_T$,
\[
\begin{aligned}
\Var(\hat\Psi\mid\mathcal F_{T_0})
&=
\frac1{n^2}
\sum_{s=T_0+1}^{T_1}
\Var(X_s\mid\mathcal F_{T_0})
\\
&\le
\frac{K_T}{n}\E_{P_e}[C(Z)^2]
=
O(T^{-7/8}).
\end{aligned}
\]
Conditional Chebyshev's inequality and
$\Pr(\mathcal A_T)\to1$ therefore yield
\[
\left(\hat\Psi-\E_{P_e}[m_{\hat\theta_{T_0}}]\right)^2
=
O_p(T^{-7/8})
=
o_p(T^{-1/2}).
\]

For the second term, the derivative envelope gives, on
$\mathcal A_T$,
\[
|m_{\hat\theta_{T_0}}(z)-m_{\theta_{P_e}}(z)|
\le
C(z)\|\hat\theta_{T_0}-\theta_{P_e}\|_\Theta.
\]
Hence, by Jensen's inequality and the pilot rate,
\[
\begin{aligned}
\left|
\E_{P_e}[m_{\hat\theta_{T_0}}-m_{\theta_{P_e}}]
\right|^2
&\le
\E_{P_e}\!\left[
(m_{\hat\theta_{T_0}}-m_{\theta_{P_e}})^2
\right]
\\
&\le
\E_{P_e}[C(Z)^2]\,
\|\hat\theta_{T_0}-\theta_{P_e}\|_\Theta^2
=
o_p(T^{-1/2}).
\end{aligned}
\]
Combining the two bounds proves the first conclusion.

Finally, $\hat\Psi$ is $\mathcal F_{T_1}$-measurable and
$\hat\Psi=\Psi_e+o_p(1)=O_p(1)$. Therefore, for $t>T_1$,
$\hat\Psi_t=\hat\Psi$ is predictable, uniformly bounded in
probability, and
\[
\frac1T\sum_{t=T_1+1}^T(\hat\Psi_t-\Psi_e)^2
=
\frac{T-T_1}{T}(\hat\Psi-\Psi_e)^2
=
o_p(T^{-1/2}).
\]
\end{proof}

\section{Deferred Proofs} \label{appendix:proofs}
\subsection{Deferred Assumptions} \label{subsec:deferred_assumptions}
Prior to presenting proofs, we briefly state, discuss and justify assumptions that are not discussed in the main body. 
\begin{assumption}[Envelope condition]
\label{assumption:envelope}

There exist constants $\epsilon_\Theta,\epsilon_\Upsilon > 0$
and a measurable function
$C : \mathcal Z \to \mathbb R_+$
such that
$\E_{P_e}[C(Z)^4] < \infty$,
and the following conditions hold whenever
$\|\theta-\theta_{P_e}\|_\Theta \le \epsilon_\Theta$,
$\|\theta'-\theta_{P_e}\|_\Theta \le \epsilon_\Theta$,
$\|\eta-\eta_{P_e}\|_\Upsilon \le \epsilon_\Upsilon$,
and
$\|\eta'-\eta_{P_e}\|_\Upsilon \le \epsilon_\Upsilon$.

\smallskip
\begin{enumerate}
\item
$|m_\theta(z)| \le C(z)$;

\item
$|\partial_\theta m_\theta(z)[h]|
\le
C(z)\|h\|_\Theta$;
\item $|\partial_\theta\ell_{\theta_{P_e},\eta_{P_e}}(z)[h]|\leq C(z)\|h\|_\Theta$ 
\item
$|\partial_\theta^2
\ell_{\theta,\eta}(z)[h,u]|
\le
C(z)\|h\|_\Theta\|u\|_\Theta$;

\item
$|\partial_\eta\partial_\theta
\ell_{\theta,\eta}(z)[h,v]|
\le
C(z)\|h\|_\Theta\|v\|_\Upsilon$.

\end{enumerate}

\smallskip
Furthermore, the derivatives are locally Lipschitz in this neighborhood:
\begin{enumerate}

\item
$\left|
\partial_\theta m_\theta(z)[h]
-
\partial_\theta m_{\theta'}(z)[h]
\right|
\le
C(z)\,
\|\theta-\theta'\|_\Theta
\|h\|_\Theta$;

\item
$\left|
\partial_\eta\partial_\theta
\ell_{\theta,\eta}(z)[h,v]
-
\partial_\eta\partial_\theta
\ell_{\theta',\eta'}(z)[h,v]
\right|
\le
C(z)
\left(
\|\theta-\theta'\|_\Theta
+
\|\eta-\eta'\|_\Upsilon
\right)
\|h\|_\Theta
\|v\|_\Upsilon$;

\item
$\left|
\partial_\theta^2
\ell_{\theta,\eta}(z)[h,u]
-
\partial_\theta^2
\ell_{\theta',\eta'}(z)[h,u]
\right|
\le
C(z)
\left(
\|\theta-\theta'\|_\Theta
+
\|\eta-\eta'\|_\Upsilon
\right)
\|h\|_\Theta
\|u\|_\Theta$.

\end{enumerate}
\end{assumption}

\cref{assumption:envelope} is used to translate consistency of the parameter estimates into asymptotic negligibility of the function-valued quantities that enter the one-step expansion. In particular, once \((\hat\theta,\hat\eta,\hat\alpha)\) lies in a local neighborhood of \((\theta_{P_e},\eta_{P_e},\alpha_{P_e})\), the envelope ensures that the associated influence function increments are dominated by an integrable random variable. This allows the population Taylor remainders in \cref{thm:pe_expansion} and the martingale remainder terms in the asymptotic normality proofs (\cref{thm:clt}, \cref{thm:self-normalized}) to be controlled by products of estimation errors. The local Lipschitz bounds play the analogous role for differences of derivative terms: they ensure that replacing \((\theta_{P_e},\eta_{P_e},\alpha_{P_e})\) by consistent estimates changes the induced functionals by an amount proportional to the corresponding estimation error. In finite-dimensional models, these Lipschitz conditions are implied by bounded third derivatives; here we state them directly in operator form so that the same argument applies to general nonparametric \(M\)-estimands. Similar sets of restrictions are common in standard proofs of asymptotic normality for $M$-estimators such as in \cite{van2000asymptotic}.
\begin{lemma}[Completion in the Hessian norm]\label{lemma:hilbert_completion}
Under Assumptions~\ref{assumption:reg_pe} and~\ref{assumption:envelope}, let
$\bar{\Theta}$ be the completion of $\Theta$ under
$\|h\|_{H_{P_e}}:=\{H_{P_e}(h,h)\}^{1/2}$.
The norm $\|\cdot\|_\Theta$ extends to an equivalent norm on $\bar{\Theta}$.
The derivatives entering the one-step expansion extend uniquely and continuously
in their $\Theta$-direction arguments, preserving their envelope and local
Lipschitz bounds. In particular,
$L(h):=\E_{P_e}[\partial_\theta m_{\theta_{P_e}}(Z)[h]]$
has a unique Riesz representer $\alpha_{P_e}\in\bar{\Theta}$, and
\[
\alpha_{P_e}
=\arg\min_{\alpha\in\bar{\Theta}}
\left\{\tfrac12 H_{P_e}(\alpha,\alpha)-L(\alpha)\right\}.
\]
\end{lemma}
\begin{proof}
Let $M:=\E_{P_e}[C(Z)]<\infty$. Coercivity and the Hessian envelope give
\[
c\|h\|_\Theta^2\leq\|h\|_{H_{P_e}}^2
\leq M\|h\|_\Theta^2,\qquad h\in\Theta.
\]
Thus both norms have the same Cauchy sequences, and
$\|h\|_\Theta:=\lim_n\|h_n\|_\Theta$ is well defined whenever
$h_n\in\Theta$ converges to $h\in\bar{\Theta}$. The Hessian is symmetric
by the stated differentiability and local Lipschitz assumptions, so its
positive definite extension makes $\bar{\Theta}$ a Hilbert space.
Moreover,
\[
|L(h)|\leq M\|h\|_\Theta
\leq Mc^{-1/2}\|h\|_{H_{P_e}},
\]
so $L$ extends continuously and Riesz representation applies.

For the derivative evaluations, each envelope bound makes the images of
approximating direction sequences Cauchy. For example,
$|\partial_\theta m_\theta(z)[h_n-h_k]|
\leq C(z)c^{-1/2}\|h_n-h_k\|_{H_{P_e}}$.
The bilinear maps extend by approximation in each $\Theta$-direction argument;
their product bounds make the limits independent of the approximating sequences.
The loss derivative at a nearby $(\theta,\eta)$ is also bounded: integrating
the Hessian and mixed derivative along the segment from
$(\theta_{P_e},\eta_{P_e})$ gives
\[
|\partial_\theta\ell_{\theta,\eta}(z)[h]|
\leq C(z)\bigl(1+\|\theta-\theta_{P_e}\|_\Theta
+\|\eta-\eta_{P_e}\|_\Upsilon\bigr)\|h\|_\Theta.
\]
Taking limits preserves the envelope and Lipschitz inequalities.
Their integrable bounds justify passing expectations to the limit, including
the first-order and orthogonality identities. These extensions act only on
direction arguments; $\theta$ and $\eta$ remain in their original spaces.
Finally, $L(\alpha)=H_{P_e}(\alpha_{P_e},\alpha)$ implies
\[
\tfrac12H_{P_e}(\alpha,\alpha)-L(\alpha)
=\tfrac12\|\alpha-\alpha_{P_e}\|_{H_{P_e}}^2
-\tfrac12\|\alpha_{P_e}\|_{H_{P_e}}^2,
\]
which proves the claimed unique minimizer.
\end{proof}

\subsection{Proof of \cref{thm:pe_expansion}}
\begin{proof}
Recalling that $H_{P_e}(h,u):=\E_{P_e}\!\left[\partial_\theta^2 \ell_{\theta_{P_e},\eta_{P_e}}(Z)[h,u]\right]$, we construct a functional Taylor expansion at the evaluation law. Since $m_\theta$ is Fréchet differentiable, we have the expansion, we have that
\begin{align*}
\Psi_e(\bar\theta) - \Psi_e(\theta_{P_e})&=\E_{P_e}[m_{\bar\theta} - m_{\theta_{P_e}}] \\
&=\E_{P_e}\left[\partial_\theta m_{\theta_{P_e}}(z)[\bar\theta - \theta_{P_e}]\right] + R_{1},
\end{align*}
where
$r_1(z)
=
\int_0^1
\left\{
\partial_\theta m_{\theta_{P_e}+u(\bar\theta - \theta_{P_e})}(z)
-
\partial_\theta m_{\theta_{P_e}}(z)
\right\}[\bar\theta - \theta_{P_e}]\,du.
$ and and $R_1 = \E_{P_e}[r_1(Z)]$.
By \cref{assumption:envelope},
\[
|r_1(z)|
\leq
\int_0^1 C(z)u\|\hat\theta_t - \theta_{P_e}\|_\Theta^2\,du
\leq
C(z)\|\hat\theta_t - \theta_{P_e}\|_\Theta^2.
\]
Therefore
\[R_{1} = \E_{P_e}[r_1(Z)] \leq
|\E_{P_e}[r_1(Z)]|
\leq
\E_{P_e}[C(Z)]\|\hat\theta_t - \theta_{P_e}\|_\Theta^2.
\]
If \(m_\theta\) is linear in \(\theta\), then \(r_1(z)=0\) and thus $R_1 = 0 $.

By the Riesz representation identity, $
\E_{P_e}\left[\partial_\theta m_{\theta_{P_e}}(Z)[h]\right]
=
H_{P_e}(\alpha_{P_e},h)$ for all $ h \in \Theta$,
so
\[
\Psi_e(\bar\theta) - \Psi_e(\theta_{P_e})
=
H_{P_e}(\alpha_{P_e},\bar\theta - \theta_{P_e}) + R_1.
\]

We now expand $\E_{P_e}\!\left[\partial_\theta \ell_{\theta,\eta}(Z)[\alpha]\right]$ around $(\theta_{P_e},\eta_{P_e})$:
\[
\partial_\theta \ell_{\bar\theta,\bar\eta}(z)[\bar\alpha]
=
\partial_\theta \ell_{\theta_{P_e},\eta_{P_e}}(z)[\bar\alpha]
+
\partial_\theta^2 \ell_{\theta_{P_e},\eta_{P_e}}(z)[\bar\theta - \theta_{P_e},\bar\alpha]
+
\partial_\eta\partial_\theta \ell_{\theta_{P_e},\eta_{P_e}}(z)[\bar\alpha,\bar\eta - \eta_{P_e}]
+
r_\ell(z),
\]
where {$\theta_u=\theta_{P_e}+u(\bar\theta-\theta_{P_e})$ and $\eta_u=\eta_{P_e}+u(\bar\eta-\eta_{P_e})$, and} the remainder is given by
\[
r_2(z)
=
\int_0^1\,
\Big[
\partial_\theta^2 \ell_{\theta_u,\eta_u}(z)[\bar\theta - \theta_{P_e},\bar\alpha]
-
\partial_\theta^2 \ell_{\theta_{P_e},\eta_{P_e}}(z)[\bar\theta - \theta_{P_e},\bar\alpha]
\Big] du
\]
\[
\quad +
{\int_0^1\,}
\Big[
\partial_\eta\partial_\theta \ell_{\theta_u,\eta_u}(z)[\bar\alpha,\bar\eta - \eta_{P_e}]
-
\partial_\eta\partial_\theta \ell_{\theta_{P_e},\eta_{P_e}}(z)[\bar\alpha,\bar\eta - \eta_{P_e}]
\Big] du.
\]
Taking expectation,
\[
\E_{P_e}\left[\partial_\theta \ell_{\bar\theta,\bar\eta}(z)[\bar\alpha]\right]
=
\E_{P_e}[\partial_\theta \ell_{\theta_{P_e},\eta_{P_e}}(Z)[\bar\alpha]]
+
H_{P_e}(\bar\theta - \theta_{P_e},\bar\alpha)
+
\E_{P_e}[\partial_\eta\partial_\theta \ell_{\theta_{P_e},\eta_{P_e}}(Z)[\bar\alpha,\bar\eta - \eta_{P_e}]]
+
R_2,
\]
where $R_2 := \E_{P_e}[r_2(Z)]$.

Under the envelope bound of \cref{assumption:envelope}, there exists $C(z)$ such that
\[
|\partial_\theta^2 \ell_{\theta,\eta}(z)[h,u]|
\le C(z)\|h\|_\Theta\|u\|_\Theta,
\]
\[
|\partial_\eta\partial_\theta \ell_{\theta,\eta}(z)[h,u]|
\le C(z)\|h\|_\Theta\|u\|_\Upsilon.
\]
Noting that the integrand is constant in the case that $\delta_{\mathrm{quad}} =0$, we have that
\[
|r_2(z)|
\lesssim
C(z)\|\bar\alpha\|_\Theta
\left(
\delta_{\mathrm{quad}}\|\bar\theta - \theta_{P_e}\|_\Theta^2
+
\|\bar\theta - \theta_{P_e}\|_\Theta\|\bar\eta - \eta_{P_e}\|_\Upsilon
+
\|\bar\eta - \eta_{P_e}\|_\Upsilon^2
\right).
\]
Taking expectations,
\[
|R_2|
\lesssim
\|\bar\alpha\|_\Theta
\left(
\delta_{\mathrm{quad}}\|\bar\theta - \theta_{P_e}\|_\Theta^2
+
\|\bar\theta - \theta_{P_e}\|_\Theta\|\bar\eta - \eta_{P_e}\|_\Upsilon
+
\|\bar\eta - \eta_{P_e}\|_\Upsilon^2
\right).
\]

By the first order conditions, we have that \cref{assumption:reg_pe}
\[
\E_{P_e}[\partial_\theta \ell_{\theta_{P_e},\eta_{P_e}}(Z)[h]] = 0, \quad \text{ and } \quad
\E_{P_e}[\partial_\eta\partial_\theta \ell_{\theta_{P_e},\eta_{P_e}}(Z)[h,u]] = 0.
\]
Thus,
\begin{align*}
\E_{P_e} \left[\partial_\theta \ell_{\bar\theta,\bar\eta}(z)[\bar\alpha]\right]
&=
H_{P_e}(\bar\theta - \theta_{P_e},\bar\alpha) + R_2 \\
&=H_{P_e}(\alpha_{P_e},\bar\theta - \theta_{P_e})
+
H_{P_e}(\bar\alpha-\alpha_{P_e},\bar\theta - \theta_{P_e})+ R_2 
\end{align*}
Note that under bilinearity $H_{P_e}(\alpha_{P_e},\bar\theta - \theta_{P_e})
=
\E_{P_e} \left[\partial_\theta \ell_{\bar\theta,\bar\eta}(z)[\bar\alpha]\right]
-
H_{P_e}(\bar\alpha-\alpha_{P_e},\bar\theta - \theta_{P_e})
-
R_2$
Putting everything together, we have that
\[
\Psi_e(\bar\theta) - \Psi_e(\theta_{P_e})
=
\E_{P_e}[\partial_\theta \ell_{\bar\theta,\bar\eta}(Z)[\bar\alpha]]
-
H_{P_e}(\bar\alpha - \alpha_{P_e},\bar\theta - \theta_{P_e})
+
R_1 - R_2.
\]

Combining the bounds on $R_1$ and $R_2$ yields the result.
\end{proof}

\subsection{Proof of \cref{thm:clt}}

\begin{proof}
Let us rewrite
\begin{align*}
\frac{ \hat{A}_T}{\sqrt T}(\hat\Psi_T^{\mathrm{os}} - \Psi_e(\theta_{P_e})) &= \frac{1}{\sqrt{T}} \sum_{t=1}^{T} \hat{\sigma}_{t}^{-1} w_{t} \left( {m_{\hat\theta_t}(Z_t)-\dot\ell_{\hat\theta_t,\hat\eta_t}(Z_t)[\hat\alpha_t]} - \Psi_{e}(\theta_{P_e})\right)     
\end{align*}
Let $F(\theta,\eta,\alpha)(z) :=m_{\theta}(z)-\dot\ell_{\theta,\eta}(z)[\alpha]$. Add and subtract $ F(\theta_{P_e},\eta_{P_e},\alpha_{P_e})$ inside the summand. We then have that
\begin{align*}\frac{ \hat{A}_T}{\sqrt T}(\hat\Psi_T^{\mathrm{os}} - \Psi(P_e))= \underbrace{\frac{1}{\sqrt T}\sum_{t=1}^T \hat{\sigma}_t^{-1} w_t \varphi_t}_{\overset{d}{\to}  N(0,1)} + \underbrace{\frac{1}{\sqrt{T}} \sum_{t=1}^{T}\hat \sigma_t^{-1} w_t \left(F(\hat\theta_t,\hat \eta_t,\hat \alpha_t)(Z_t) - F(\theta_{P_e},\eta_{P_e},\alpha_{P_e})(Z_t) \right)}_{\text{Remainder }} .
\end{align*}
The first term converges to $N(0,1)$ immediately by applying a martingale central limit theorem (\cref{prop:mds_clt}). We enumerate some sufficient conditions for the remainder to be asymptotically negligible in \cref{lemma:remainder}.

Conditions (a) of \cref{lemma:remainder} is true by assumption, so all that remains is verifying conditions (b) - (e). To this end, we can apply \cref{lemma:empirical_predictable}. We verify the conditions of \cref{lemma:empirical_predictable} for each term one-by one:
\paragraph{Showing $P_T^{a}\Big(
\dot{m}_{\theta_{P_e}}(\cdot)[\hat\theta_t - \theta_{P_e}]
-
\ddot{\ell}_{\theta_{P_e},\eta_{P_e}}(\cdot)
[\hat\theta_t - \theta_{P_e}, \alpha_{P_e}]
\Big) = o_p(T^{-1/2})$}

We note that because $\hat{\theta}_{t}$ is $\mathcal{F}_{t-1}$ measurable, it is a constant conditional on $\mathcal{F}_{t-1}$. Therefore, if we let $f_{t} := \dot{m}_{\theta_{P_e}}(\cdot)[\hat\theta_t - \theta_{P_e}]
-
\ddot{\ell}_{\theta_{P_e},\eta_{P_e}}(\cdot)
[\hat\theta_t - \theta_{P_e}, \alpha_{P_e}]$, we have:
\begin{align*}
\E\left[ w_t f_t(Z_t) \mid \mathcal F_{t-1}\right] &=\E\left[ w_t\left(\dot{m}_{\theta_{P_e}}(Z_t)[\hat\theta_t - \theta_{P_e}]
-
\ddot{\ell}_{\theta_{P_e},\eta_{P_e}}(Z_t)[\hat\theta_t - \theta_{P_e},\alpha_{P_e}]\right)  \mid \mathcal F_{t-1}\right] \\
&= \E_{p_{e}}\left[\dot{m}_{\theta_{P_e}}(Z)[\hat\theta_t - \theta_{P_e}] -
\ddot{\ell}_{\theta_{P_e},\eta_{P_e}}(Z) [\hat\theta_t - \theta_{P_e},\alpha_{P_e}] \right]  \\
&= 0,
\end{align*}
where for the last step we use the the properties of the Riesz representer $\alpha_{P_e}$ . Condition (b) is true by assumption. For condition c, we have that $\sup_{(c,a)} w_{t} < K_{T}$ for some $K_T = O(T^{1/3})$. Therefore, 

\begingroup
\begin{align*}
\frac1{T^2}\sum_{t=1}^T\left\{\E[w_t^2C(Z_t)^2\mid\mathcal F_{t-1}]\right\}^2
&\leq\frac1{T^2}\sum_{t=1}^T\left\{K_T\E[w_tC(Z_t)^2\mid\mathcal F_{t-1}]\right\}^2\\
&=\frac{K_T^2}{T}(\E_{P_e}[C^2])^2=O(T^{-1/3})=O(1).
\end{align*}
\endgroup

\paragraph{Showing $P_T^{a}\Big(
\,\partial_{\eta} \partial_{\theta}\ell_{\theta_{P_e},\eta_{P_e}}(\cdot)
{[\alpha_{P_e},\hat\eta_t - \eta_{P_e}]}
\Big) =  o_p(T^{-1/2})$} The arguments are  the same as for the previous term, with the only step differing in showing $\E\!\left[ w_t f_t(Z_t) \mid \mathcal F_{t-1}\right] = 0$ for each $t \in [T]$. Here, we have:
\begin{align*}
\E\left[ w_t f_t(Z_t) \mid \mathcal F_{t-1}\right] &=\E\left[ w_t\partial_{\eta} \partial_{\theta}\ell_{\theta_{P_e},\eta_{P_e}}(Z_t)
{[\alpha_{P_e},\hat\eta_t - \eta_{P_e}]}  \mid \mathcal F_{t-1}\right] \\
&= \E_{p_{e}}\left[\partial_{\eta} \partial_{\theta}\ell_{\theta_{P_e},\eta_{P_e}}(Z)
{[\alpha_{P_e},\hat\eta_t - \eta_{P_e}]} \right]  \\
&= 0,
\end{align*}
with the last line following from the Neyman orthogonality condition in \cref{assumption:reg_pe} that
 $
\E_{P_e}\!\left[\partial_\eta \partial_\theta \ell_{\theta_{P_e},\eta_{P_e}}(Z)[h,u]\right] = 0$ for all $h \in \Theta$ and $u \in \Upsilon$.
\paragraph{Showing $P_T^{a}\Big(
\,\partial_\theta \ell_{\theta_{P_e},\eta_{P_e}}(\cdot)
[\hat \alpha_t - \alpha_{P_e}]
\Big) = o_p(T^{-1/2})$} 
The arguments are  the same as for the previous term, with the only step differing in showing $\E\!\left[ w_t f_t(Z_t) \mid \mathcal F_{t-1}\right] = 0$ for each $t \in [T]$. Here, we have:
\begin{align*}
\E\left[ w_t f_t(Z_t) \mid \mathcal F_{t-1}\right] &=\E\left[ w_t\partial_\theta \ell_{\theta_{P_e},\eta_{P_e}}(Z_t)
[\hat \alpha_t - \alpha_{P_e}]  \mid \mathcal F_{t-1}\right] \\
&= \E_{p_{e}}\left[\partial_\theta \ell_{\theta_{P_e},\eta_{P_e}}(Z_t)
[\hat \alpha_t - \alpha_{P_e}] \right]  \\
&= 0,
\end{align*}
with the last line following from the fact that the first order condition of \cref{assumption:reg_pe} states that  $\E_{P_e}\!\left[
\partial \theta \ell_{\theta_{P_e},\eta_{P_e}}(Z)[h]
\right] = 0$ for all $h \in \Theta$.
Finally, all that remains is showing condition (e) of \cref{lemma:remainder} is satisfied. 

\paragraph{Showing$\frac{1}{T^2}\ \sum_{t=1}^{T} \E\left[ w_{t}^{4}C(Z_t)^{4} \mid \mathcal{F}_{t-1}\right] =O_p(1)$.}
This is almost the same as in the prior step. All we need to show is:

\begingroup
\begin{align*}
\frac1{T^2}\sum_{t=1}^T\E[w_t^4C(Z_t)^4\mid\mathcal F_{t-1}]
&\leq\frac{K_T^3}{T^2}\sum_{t=1}^T\E[w_tC(Z_t)^4\mid\mathcal F_{t-1}]\\
&=\frac{K_T^3}{T}\E_{P_e}[C^4]=O(1).
\end{align*}
\endgroup

Thus, the remainder is asymptotically negligible and asymptotic normality holds. 
\end{proof}

\subsection{Proof of \cref{thm:self-normalized}}

\begin{proof}

Define $F(\theta,\eta,\alpha)(z) :=
m_\theta(z)-\partial_\theta \ell_{\theta,\eta}(z)[\alpha]$, $F_0(z):=F(\theta_{P_e},\eta_{P_e},\alpha_{P_e})(z)$  and set
$\Psi_e := \Psi_e(\theta_{P_e})$, and $D_t := w_t\big(F_0(Z_t)-\Psi_e\big) $ Define the plug-in quantities
\[
\hat D_t := w_t\big({F(\hat\theta_t,\hat\eta_t,\hat\alpha_t)(Z_t)}-\bar\Psi_T\big),
\quad \text{for } t > 2m_T,
\]
and set $\hat D_t := 0$ for $t \le 2m_T$. Then we can write

\begingroup
\begin{align*}
\hat V_T^{-1/2}(\hat S_T-\Psi_e)
&=\frac{A_T^{-1}\sum_{t=1}^T\{w_tF(\hat\theta_t,\hat\eta_t,\hat\alpha_t)(Z_t)-w_t\Psi_e\}}
{A_T^{-1}(\sum_{t=2m_T+1}^T\hat D_t^2)^{1/2}}\\
&=\sqrt T\,\frac{T^{-1}\sum_{t=1}^T\{w_tF(\hat\theta_t,\hat\eta_t,\hat\alpha_t)(Z_t)-w_t\Psi_e\}}
{(T^{-1}\sum_{t=2m_T+1}^T\hat D_t^2)^{1/2}}.
\end{align*}
\endgroup

Define
\[
{\tilde S_T = \frac{1}{T}\sum_{t=1}^T\{w_tF(\hat\theta_t,\hat\eta_t,\hat\alpha_t)(Z_t)-w_t\Psi_e(\theta_{P_e})\},}
\qquad
{\widetilde V_T = \frac{1}{T}\sum_{t=2m_T+1}^T \hat D_t^2.}
\]
It now suffices to verify the conditions of \cref{prop:self_normalized_plugin} for $\{D_t,\hat D_t,\tilde{S}_{t}\}$ in order to demonstrate asymptotic normality.

By construction, $\{D_t\}_{t=1}^{T}$ is a martingale difference sequence with respect to $\{\mathcal F_{t-1}\}_{t=1}^{T}$. Assumptions (a) and (c) are true by {condition~\ref{cond:finite_rv} of \cref{thm:self-normalized}} and condition~\ref{cond:lindeberg} of \cref{thm:clt} respectively. We now verify the remaining pieces of the lemma, namely
\begin{enumerate}
\item (Realized Variance) $(TV_{T})^{-1}\sum_{t=1}^{T} D_{t}^{2} \overset{p}{\to} 1 $; 
\item (Plug-in Rate 1) $\tilde{S}_T - S_T = o_p(T^{-1/2})$
\item (Plug-in Rate 2) $\frac{1}{T}\sum_{t=1}^T (\hat D_t - D_t)^2 = o_p(1)$.
\end{enumerate}

\paragraph{Realized Variance} We wish to show $(TV_{T})^{-1}\sum_{t=1}^{T} D_{t}^{2} \overset{p}{\to} 1 $. Writing this out,
\[
\frac{1}{T}\sum_{t=1}^T D_t^2 - V_T
=
\frac{1}{T}\sum_{t=1}^T
\left\{
D_t^2 - E[D_t^2 \mid \mathcal F_{t-1}]
\right\}.
\]
We have that
\[
{\E[\{D_t^2-\E[D_t^2\mid\mathcal F_{t-1}]\}^2\mid\mathcal F_{t-1}]}
\le
E[D_t^4 \mid \mathcal F_{t-1}].
\]

\begingroup
By \cref{lemma:fourth_IF}, the conditional fourth moment is $O_p(K_T^3)$, so this untruncated variance bound is $O_p(K_T^3/T)$ and need not vanish at the allowed endpoint. We retain the centered-square argument but truncate its increments using the assumed Lindeberg condition. For $\epsilon>0$ define
\[
U_{T,t}=D_t^2\mathbf1\{|D_t|\leq\epsilon\sqrt T\},\qquad
L_T(\epsilon)=\frac1T\sum_t\E[D_t^2\mathbf1\{|D_t|>\epsilon\sqrt T\}\mid\mathcal F_{t-1}].
\]
Then $L_T(\epsilon)\to_p0$. For a nonnegative adapted sum $X_T$ with predictable sum $A_T$, stopping before the next predictable increment would make $A_T$ exceed $b$ and applying Markov's inequality gives
\[
\Pr(X_T>\eta)\leq b/\eta+\Pr(A_T>b).
\]
Apply this to $X_T=T^{-1}\sum_tD_t^2\mathbf1\{|D_t|>\epsilon\sqrt T\}$ and $A_T=L_T(\epsilon)$: the realized tail also tends to zero in probability. For the retained increments,
\begin{align*}
\frac1{T^2}\sum_t\E[\{U_{T,t}-\E[U_{T,t}\mid\mathcal F_{t-1}]\}^2\mid\mathcal F_{t-1}]
&\leq\frac1{T^2}\sum_t\E[D_t^4\mathbf1\{|D_t|\leq\epsilon\sqrt T\}\mid\mathcal F_{t-1}]\\
&\leq\epsilon^2\frac1T\sum_t\E[D_t^2\mid\mathcal F_{t-1}]=\epsilon^2V_T.
\end{align*}
The stopped martingale second-moment bound consequently gives
\[
\Pr\left(\left|T^{-1}\sum_t\{U_{T,t}-\E[U_{T,t}\mid\mathcal F_{t-1}]\}\right|>\eta\right)
\leq\epsilon^2L/\eta^2+\Pr(V_T>L).
\]
First choose $L$ large using $V_T\to_p\kappa<\infty$, then take $T\to\infty$ and $\epsilon\downarrow0$. Combining retained increments, their predictable tails, and their realized tails proves the centered-square average is $o_p(1)$.
\endgroup
{The preceding truncated second-moment argument yields}
\[
\frac{1}{T}\sum_{t=1}^T
\left\{
D_t^2 - E[D_t^2 \mid \mathcal F_{t-1}]
\right\}
= o_p(1).
\]
Therefore,
$\frac{1}{T}\sum_{t=1}^T D_t^2
=
V_T + o_p(1)$. Since \(V_T \xrightarrow{p} \kappa\) with \(\kappa>0\) almost surely, it follows that
$
\frac{1}{T V_T}\sum_{t=1}^T D_t^2
=
1 + o_p(1)$.

\paragraph{Showing $\tilde{S}_T - S_T = o_p(T^{-1/2})$}
Note that {$S_T=T^{-1}\sum_{t=1}^T\{w_tF_0(Z_t)-w_t\Psi_e\}$}. Therefore, 
\begin{align*} 
\tilde S_T-S_T
=
\frac1T\sum_{t=1}^T
w_t\{F(\hat\theta_t,\hat\eta_t,\hat\alpha_t)(Z_t)-F_0(Z_t)\}.
\end{align*}
Demonstrating this quantity is $o_p(T^{-1/2})$ is precisely the same argument as showing the remainder in \cref{thm:clt} is $o_p(T^{-1/2})$. The only difference is that in this case we let $\hat{\sigma}_{t} = 1$ for all $t\in[T]$ which is allowed because the only requirement is that the weights are almost surely bounded. Repeating these arguments with this choice of $\hat{\sigma}_{t}$ yields the result.

\paragraph{Showing $\frac{1}{T}\sum_{t=1}^T (\hat D_t - D_t)^2 = o_p(1)$} We first show that
$\bar\Psi_T-\Psi_e=o_p(1)$.
Note that 
\begin{align*}
\bar\Psi_T - \Psi_{e}
&=
\frac1{m_T}
\sum_{t=m_T+1}^{2m_T}
w_tF(\bar\theta,\bar\eta,\bar\alpha)(Z_t) - \Psi_{e} \\
&=\underbrace{\left(\frac1{m_T}
\sum_{t=m_T+1}^{2m_T}
w_tF(\bar\theta,\bar\eta,\bar\alpha)(Z_t)
-
\E_{P_e}\left[F(\bar\theta,\bar\eta,\bar\alpha)(Z)\right]\right)}_{R_1} + \underbrace{\left(\E_{P_e}\left[F(\bar\theta,\bar\eta,\bar\alpha)(Z)\right]-\E_{P_e}\left[F_0(Z)\right]\right)}_{R_2}
\end{align*}
To bound $R_1$, note that $\E \left[ w_tF(\bar\theta,\bar\eta,\bar\alpha)(Z_t) \mid \mathcal{F}_{t-1}\right] = \E_{P_e}\left[F(\bar\theta,\bar\eta,\bar\alpha)(Z)\right]$ {so the centered summands are martingale differences. On the predictable localization event of Lemma~\ref{lemma:fourth_IF}, $|F(\bar\theta,\bar\eta,\bar\alpha)|\leq L C$. Thus $\E[(w_t\bar F-\E_{P_e}[\bar F])^2\mid\mathcal F_{t-1}]\leq K_T\E_{P_e}[\bar F^2]\leq K_TL^2\E_{P_e}[C^2]$. Summing over the $m_T$ centering observations and dividing by $m_T^2$ gives $R_1=O_p(\sqrt{K_T/m_T})$. There is no shrinking-function factor here to justify a little-o root-$m_T$ bound.}

To bound $R_{2}$, we can apply \cref{thm:pe_expansion} because this term is precisely
\begin{align*}
\E_{P_e}\left[F(\bar\theta,\bar\eta,\bar\alpha)(Z)\right]-\E_{P_e}\left[F_0(Z)\right] &= \Psi_e(\bar\theta)-\Psi_e(\theta_{P_e}) -  \E_{P_e}\left[\partial_\theta\ell_{\bar\theta,\bar\eta}[\bar\alpha]\right]\\
&=
-
H_{P_e}[\bar\alpha-\alpha_{P_e},\,\bar\theta-\theta_{P_e}]   + O\left(
{(1+\|\bar\alpha\|_\Theta)}
\|\bar\theta - \theta_{P_e}\|_\Theta^2
\right)\\
& \qquad +
O\left(\|\bar\alpha\|_\Theta\left(\|\bar\theta-\theta_{P_e}\|_\Theta
\|\bar\eta-\eta_{P_e}\|_\Upsilon+\|\bar\eta-\eta_{P_e}\|_\Upsilon^2\right)\right).
\end{align*}
The requirements that $ {\|\bar\theta - \theta_{P_e}\|_\Theta =o_p(m_{T}^{-1/4})}$,$\|\bar\eta-\eta_{P_e}\|_\Upsilon=o_p(m_{T}^{-1/4})$, ${\|\bar\alpha-\alpha_{P_e}\|_\Theta}=o_p(m_{T}^{-1/4})$  is sufficient when combined with the results of \cref{lemma:bounded_riesz} which shows that $\|\bar\alpha\|_\Theta$ is {bounded in probability} to conclude that $R_2 = o_p(m_{T}^{-1/2})$.
Together, this shows {$\bar\Psi_T-\Psi_e=O_p(\sqrt{K_T/m_T})+o_p(m_T^{-1/2})$, and hence $K_T(\bar\Psi_T-\Psi_e)^2=O_p(K_T^2/m_T)+o_p(K_T/m_T)=o_p(1)$}.

We now move to showing the result. First, note that $\frac1T\sum_{t=1}^{2m_T}D_t^2=o_p(1)$,
{because, for the stated fixed martingale sequence, condition~\ref{cond:finite_rv} also gives $V_{2m_T}\to_p\kappa$. Therefore $T^{-1}\sum_{t\leq2m_T}\sigma_t^2=(2m_T/T)V_{2m_T}=o_p(1)$, since $m_T=o(T)$. The nonnegative-process bound established above gives the realized conclusion. This derives the needed negligibility from the existing assumptions; no additional discarded-block condition is required. For a horizon-dependent array, terminal variance convergence alone would not justify substituting $2m_T$ for $T$ in this way.} 

For \(t>2m_T\),
\[
\hat D_t-D_t
=
w_t\{{F(\hat\theta_t,\hat\eta_t,\hat\alpha_t)}(Z_t)-F_0(Z_t)\}
+
w_t(\Psi_e-\bar\Psi_T).
\]
Therefore,
\[
(\hat D_t-D_t)^2
\lesssim
w_t^2\{{F(\hat\theta_t,\hat\eta_t,\hat\alpha_t)}-F_0\}^2(Z_t)
+
w_t^2(\bar\Psi_T-\Psi_e)^2.
\]
Denote {$\sup_{t,c,a}\pi_e/\pi_t\leq K_T$}. {Since $\E[w_t^2\mid\mathcal F_{t-1}]\leq K_T$, Markov\textquotesingle s inequality gives $T^{-1}\sum_tw_t^2=O_p(K_T)$.} We can therefore write.

\begin{align*}
 \frac1T\sum_{t=2m_T+1}^T(\hat D_t-D_t)^2
&\lesssim
\frac1T\sum_{t=2m_T+1}^T
w_t^2\{{F(\hat\theta_t,\hat\eta_t,\hat\alpha_t)}-F_0\}^2(Z_t) + (\bar\Psi_T-\Psi_e)^2
\frac1T\sum_{t=2m_T+1}^T w_t^2 \\
& \lesssim \left(\frac{K_T}{T} \sum_{t=2m_T+1}^T
w_t\{{F(\hat\theta_t,\hat\eta_t,\hat\alpha_t)}-F_0\}^2(Z_t)\right) + {O_p(K_T)(\bar\Psi_T-\Psi_e)^2}.\\
\end{align*}

\begingroup
To control the first term, use Lemma~\ref{lemma:epe_f}, whose retained Taylor calculation gives
\[
\frac1T\sum_t\E_{P_e}\!\left[\{F(\hat\theta_t,\hat\eta_t,\hat\alpha_t)-F_0\}^2\right]=o_p(T^{-1/2}).
\]
The summands are predictable as functions of $Z$, so the nonnegative-process bound applied after multiplying by $\sqrt T$ yields
\[
\frac1T\sum_{t>2m_T}w_t\{F(\hat\theta_t,\hat\eta_t,\hat\alpha_t)-F_0\}^2(Z_t)=o_p(T^{-1/2}).
\]
Combining this with the centering calculation gives
\[
\frac1T\sum_{t>2m_T}(\hat D_t-D_t)^2
=o_p(K_TT^{-1/2})+O_p(K_T^2/m_T)+o_p(K_T/m_T).
\]
\endgroup
Thus,  $\frac{1}{T}\sum_{t=1}^T (\hat D_t - D_t)^2 = o_p(1)$ so long as $K_{T} = o(m_{T}^{1/2})$. 

\end{proof}

\subsection{Proof of \cref{thm:var_estimate}}
\begin{proof}

\begingroup
First, for each fixed candidate $\pi$, define the predictable random function
\[
G_s(\pi):=\E\!\left[\frac{\pi_e^2(A_s\mid C_s)}{\pi(A_s\mid C_s)\pi_s(A_s\mid C_s)}\hat\varphi_s^2(Z_s)\,\middle|\,\mathcal F_{s-1}\right]
=\E_{P_e}\!\left[\frac{\pi_e}{\pi}\hat\varphi_s^2\right].
\]
Set $\bar\sigma_t^2=(t-1)^{-1}\sum_{s<t}G_s(\pi_t)$ for $t\geq2$. In general $G_s(\pi_t)$ is not the conditional expectation obtained by leaving the future random $\pi_t$ inside the conditioning on $\mathcal F_{s-1}$. The initialized first-round squared error is $O_p(K_T^2/T)=o_p(1)$, since $\sigma_1^2\leq K_T\E_{P_e}[\varphi_{P_e}^2]$; below all historical averages start at $t=2$.
\endgroup

Consider the following decomposition
$$\hat{\sigma}_{t}^{2} - \sigma_{t}^{2} 
=\left( \hat{\sigma}_{t}^{2} -\E\left[\hat{\varphi}_{t}^{2} w_{t}^{2}|\Filt \right] \right)+ \left(\E\left[w_{t}^{2}\left(\hat{\varphi}_{t}^{2}-\varphi_{t}^{2}\right)|\Filt \right]\right).$$
Define $w_{t,s} = \frac{\pi_e(A_s\mid C_s)}
{\pi_t(A_s\mid C_s)}$to be the transported importance weights defined at time $t$ transported to observation $Z_s$. Note that  $\frac{\pi_e(A_s\mid C_s)^2}
{\pi_t(A_s\mid C_s)\pi_s(A_s\mid C_s)}=w_{t,s}w_{s}$. Summing over $t\in [T]$ and applying the Cauchy--Schwarz inequality twice yields 

\begingroup
\begin{align*}
\frac1T\sum_{t=2}^T(\hat\sigma_t^2-\sigma_t^2)^2
&=\frac1T\sum_{t=2}^T\left\{(\hat\sigma_t^2-\E[w_t^2\hat\varphi_t^2\mid\Filt])
+\E[w_t^2(\hat\varphi_t^2-\varphi_t^2)\mid\Filt]\right\}^2\\
&\leq\frac2T\sum_{t=2}^T(\hat\sigma_t^2-\E[w_t^2\hat\varphi_t^2\mid\Filt])^2
+\frac2T\sum_{t=2}^T\{\E[w_t^2(\hat\varphi_t^2-\varphi_t^2)\mid\Filt]\}^2\\
&\leq\underbrace{\frac4T\sum_{t=2}^T\left\{\frac1{t-1}\sum_{s<t}
[\hat\varphi_s^2(Z_s)w_sw_{t,s}-G_s(\pi_t)]\right\}^2}_{R_1}\\
&\quad+\underbrace{\frac4T\sum_{t=2}^T(\bar\sigma_t^2-\E[w_t^2\hat\varphi_t^2\mid\Filt])^2}_{R_2}
+\underbrace{\frac2T\sum_{t=2}^T\{\E[w_t^2(\hat\varphi_t^2-\varphi_t^2)\mid\Filt]\}^2}_{R_3}.
\end{align*}
\endgroup
 
$R_1$ is an empirical process term whose summands we can bound uniformly over time by an appeal to the maximal inequality in {Fact~\ref{fact:EMI}, with a log-policy entropy argument as in \cite{BIC_regret}}. $R_{2}$ is an approximation error drift term. $R_{3}$ can be bound with roughly the same asymptotic arguments as described in \cref{thm:clt}. We now turn to each piece separately. 

\paragraph{ Showing $R_1 =o_p(1)$}

\begingroup
Choose a deterministic $t_0=o(T)$ with $t_0\to\infty$. Let $R_1=4T^{-1}\sum_{t=2}^T A_t$, where
\[
A_t=\left\{\frac1{t-1}\sum_{s<t}\bigl[\hat\varphi_s^2(Z_s)w_sw_{t,s}-G_s(\pi_t)\bigr]\right\}^2.
\]
We use the envelope established in Lemma~\ref{lemma:fourth_IF}. Its bounds are in probability, so before applying unconditional moment inequalities we stop predictably at the first round whose fitted norms or absolute centering value exceed a fixed bound $L$, or whose nuisance estimates leave the stated neighborhoods. Set subsequent fitted scores to zero. The stopping indicator at round $s$ is $\mathcal F_{s-1}$-measurable. For every $\eta>0$, $L$ can be chosen so that the probability of altering any score has limit superior at most $\eta$. For this stopped process, after enlarging the constant $L$ if necessary,
\[
|\hat\varphi_s(z)|\leq L C(z),\qquad C\geq1,\qquad \E_{P_e}[C^4]<\infty
\]
holds for every $s,z$. Constants below may depend on fixed $L$, but not on $T$. We first prove the bounds for this process and then remove the stopping.

We also choose a truncation threshold $M_T$ and decompose
\[
\hat\varphi_s^2w_sw_{t,s}
=\hat\varphi_s^2w_sw_{t,s}\mathbf1\{w_s\hat\varphi_s^2(Z_s)\leq M_T\}
+\hat\varphi_s^2w_sw_{t,s}\mathbf1\{w_s\hat\varphi_s^2(Z_s)>M_T\}.
\]
For a fixed candidate policy $\pi$, write $w_{\pi,s}=\pi_e(A_s\mid C_s)/\pi(A_s\mid C_s)$ and define
\[
\begin{aligned}
C_t^{(M)}(\pi)
:=&\frac1{t-1}\sum_{s<t}\Bigl(
\hat\varphi_s^2(Z_s)w_sw_{\pi,s}\mathbf1\{w_s\hat\varphi_s^2(Z_s)\leq M_T\}
\\&\hspace{1em}-\E[\hat\varphi_s^2(Z_s)w_sw_{\pi,s}\mathbf1\{w_s\hat\varphi_s^2(Z_s)\leq M_T\}\mid\mathcal F_{s-1}]
\Bigr),
\end{aligned}
\]
\[
\begin{aligned}
C_t^{(\mathrm{tail})}(\pi)
:=&\frac1{t-1}\sum_{s<t}\Bigl(
\hat\varphi_s^2(Z_s)w_sw_{\pi,s}\mathbf1\{w_s\hat\varphi_s^2(Z_s)>M_T\}
\\&\hspace{1em}-\E[\hat\varphi_s^2(Z_s)w_sw_{\pi,s}\mathbf1\{w_s\hat\varphi_s^2(Z_s)>M_T\}\mid\mathcal F_{s-1}]
\Bigr).
\end{aligned}
\]
All compensators are defined for fixed $\pi$; $C_t^{(M)}(\pi_t)$ and $C_t^{(\mathrm{tail})}(\pi_t)$ denote evaluation of the resulting random maps. This avoids treating $\pi_t$ as known at time $s-1$. The two indicators are complementary in every empirical term and compensator.

Since $\pi_s\in\Pi$ and $\sup_{\pi,c,a}\pi_e(a\mid c)/\pi(a\mid c)\leq K_T=O(T^{1/8})$, we have
\[
0\leq\hat\varphi_s^2w_sw_{\pi,s}\mathbf1\{w_s\hat\varphi_s^2(Z_s)\leq M_T\}
\leq K_TM_T=:B_T.
\]
For $d(\pi,\pi')=\|\log\pi-\log\pi'\|_\infty\leq1$,
\begin{align*}
|w_{\pi,s}-w_{\pi',s}|
&=w_{\pi,s}|1-\exp\{\log\pi(A_s\mid C_s)-\log\pi'(A_s\mid C_s)\}|\\
&\leq eK_Td(\pi,\pi').
\end{align*}
Multiplication by $w_s\hat\varphi_s^2\mathbf1\{w_s\hat\varphi_s^2\leq M_T\}\leq M_T$ gives the Lipschitz bound $eB_Td(\pi,\pi')$ needed in Fact~\ref{fact:EMI}. That fact, proved below with its moment bound, yields
\[
\E\sup_{\pi\in\Pi}|C_t^{(M)}(\pi)|^2\lesssim\frac{B_T^2}{t-1},
\qquad
\Pr\left\{\sup_{\pi\in\Pi}|C_t^{(M)}(\pi)|>
\frac{cB_T}{\sqrt{t-1}}(1+\sqrt{\log(1/\delta)})\right\}\leq\delta.
\]
Applying the probability bound with $\delta=T^{-2}$ and a union bound over $t\in\{t_0,\ldots,T\}$ gives
\[
\sup_{t\geq t_0}\sup_{\pi\in\Pi}|C_t^{(M)}(\pi)|
=O_p\!\left(B_T\sqrt{\frac{\log T}{t_0}}\right).
\]
Thus the late-round supremum is $o_p(1)$ when $B_T^2\log T/t_0\to0$. The factor $K_T$ in $B_T=K_TM_T$ must remain in this condition.

We have $A_t\leq2\{C_t^{(M)}(\pi_t)\}^2+2\{C_t^{(\mathrm{tail})}(\pi_t)\}^2$. Therefore,
\begin{align*}
R_1
&\leq\frac8T\sum_{t=2}^T\{C_t^{(M)}(\pi_t)\}^2
+\frac8T\sum_{t=2}^T\{C_t^{(\mathrm{tail})}(\pi_t)\}^2\\
&\leq\frac8T\sum_{2\leq t<t_0}\{C_t^{(M)}(\pi_t)\}^2
+8\sup_{t\geq t_0,\pi\in\Pi}|C_t^{(M)}(\pi)|^2
+\frac8T\sum_{t=2}^T\{C_t^{(\mathrm{tail})}(\pi_t)\}^2.
\end{align*}
For the first term, the deterministic bound $|C_t^{(M)}|\leq2B_T$ is too crude at the desired overlap rate. The second-moment part of the same maximal inequality instead gives
\begin{align*}
\E\left[\frac1T\sum_{2\leq t<t_0}\{C_t^{(M)}(\pi_t)\}^2\right]
&\leq\frac1T\sum_{2\leq t<t_0}\E\sup_{\pi\in\Pi}|C_t^{(M)}(\pi)|^2\\
&\lesssim\frac{B_T^2}{T}\sum_{n=1}^{t_0-2}\frac1n
\lesssim\frac{B_T^2\log(et_0)}{T}.
\end{align*}
Markov's inequality now proves that the early-round term is $o_p(1)$ when $B_T^2\log T/T\to0$. The second term is controlled by the preceding late-round supremum. It remains to bound the third, tail term.

For the tail term, note that
\[
0\leq\hat\varphi_s^2w_sw_{\pi,s}\mathbf1\{w_s\hat\varphi_s^2(Z_s)>M_T\}
\leq K_Tw_s\hat\varphi_s^2(Z_s)\mathbf1\{w_s\hat\varphi_s^2(Z_s)>M_T\}=: \zeta_s.
\]
Because the increments are nonnegative, with $\mu_s=\E[\zeta_s\mid\mathcal F_{s-1}]$,
\[
|C_t^{(\mathrm{tail})}(\pi_t)|
\leq\frac1{t-1}\sum_{s<t}\zeta_s+\frac1{t-1}\sum_{s<t}\mu_s.
\]
Therefore,
\[
\frac1T\sum_{t=2}^T\{C_t^{(\mathrm{tail})}(\pi_t)\}^2
\leq\frac2T\sum_{t=2}^T\left(\frac1{t-1}\sum_{s<t}\zeta_s\right)^2
+\frac2T\sum_{t=2}^T\left(\frac1{t-1}\sum_{s<t}\mu_s\right)^2.
\]
By change of measure and the fourth moment of the envelope,
\begin{align*}
\mu_s
&\leq\frac{K_T}{M_T}\E[w_s^2\hat\varphi_s^4(Z_s)\mid\mathcal F_{s-1}]\\
&\leq\frac{K_T^2}{M_T}\E_{P_e}[\hat\varphi_s^4]
\leq\frac{L^4K_T^2}{M_T}\E_{P_e}[C^4]=:b_T.
\end{align*}
Consequently the predictable part is bounded by $2b_T^2=O(K_T^4/M_T^2)$.

It remains to control the empirical average. Expanding the square,
\[
\frac1T\sum_{t=2}^T\left(\frac1{t-1}\sum_{s<t}\zeta_s\right)^2
=\frac1T\sum_{t=2}^T\frac1{(t-1)^2}\sum_{s,r<t}\zeta_s\zeta_r.
\]
We first consider the off-diagonal terms. For $s<r$, $\zeta_s$ is $\mathcal F_{r-1}$-measurable, so
\[
\E[\zeta_s\zeta_r]=\E[\zeta_s\mu_r]\leq b_T\E\zeta_s\leq b_T^2.
\]
It follows that
\begin{align*}
\E\left[\frac1T\sum_{t=2}^T\frac1{(t-1)^2}\sum_{\substack{s,r<t\\s\ne r}}\zeta_s\zeta_r\right]
&=\frac2T\sum_{t=2}^T\frac1{(t-1)^2}\sum_{s<r<t}\E[\zeta_s\mu_r]\\
&\leq\frac{b_T^2}{T}\sum_{t=2}^T\frac{(t-1)(t-2)}{(t-1)^2}\\
&\leq b_T^2=O(K_T^4/M_T^2).
\end{align*}
For the diagonal terms, all overlap factors must again be retained:
\[
\zeta_s^2\leq L^4K_T^2w_s^2C(Z_s)^4
\leq L^4K_T^3w_sC(Z_s)^4.
\]
Thus $\E\zeta_s^2\leq L^4K_T^3\E_{P_e}[C^4]$, and
\begin{align*}
\E\left[\frac1T\sum_{t=2}^T\frac1{(t-1)^2}\sum_{s<t}\zeta_s^2\right]
&\leq\frac{L^4K_T^3\E_{P_e}[C^4]}{T}\sum_{t=2}^T\frac1{t-1}\\
&=O\!\left(\frac{K_T^3\log T}{T}\right).
\end{align*}
Combining the diagonal and off-diagonal parts and applying Markov's inequality gives
\[
\frac1T\sum_{t=2}^T\left(\frac1{t-1}\sum_{s<t}\zeta_s\right)^2
=O_p\!\left(\frac{K_T^4}{M_T^2}+\frac{K_T^3\log T}{T}\right).
\]
The same rate, with a different constant, bounds the average squared tail process.

Collecting the rate requirements up to this point, we need
\[
\begin{gathered}
M_T\to\infty,\quad t_0\to\infty,\quad t_0/T\to0,\\
\frac{K_T^2M_T^2\log T}{T}\to0,\quad
\frac{K_T^2M_T^2\log T}{t_0}\to0,\quad
\frac{K_T^4}{M_T^2}\to0,\quad\frac{K_T^3\log T}{T}\to0.
\end{gathered}
\]
For $K_T=O(T^{1/8})$, choose $M_T=T^{5/16}$ and $t_0=\lceil T^{15/16}\rceil$. In the same order, the four nontrivial bounds are
\[
O(T^{-1/8}\log T),\qquad O(T^{-1/16}\log T),\qquad
O(T^{-1/8}),\qquad O(T^{-5/8}\log T),
\]
all tending to zero. Hence $R_1=o_p(1)$ for the stopped process. The original process agrees with it except on an event of asymptotic probability at most $\eta$; letting $\eta\downarrow0$ removes the stopping and proves the claim.
\endgroup

\paragraph{ Showing $R_2 = o_p(1)$}
Note that by a change of measure argument, we can rewrite
\begin{align*}
\bar{\sigma}_{t}^{2}
&= {\frac1{t-1}\sum_{s<t}G_s(\pi_t)} \\
&=
\frac{1}{t-1}\sum_{s=1}^{t-1}
\int_{\mathcal C}
\int_{\mathcal A}
\int_{\mathcal Y}
\frac{\pi_{e}^{2}(a \mid c)}
{\pi_{t}(a \mid c)\pi_{s}(a \mid c)}
\hat{\varphi}_{s}^{2}(c,a,y)
\,p_Y(dy \mid c,a)\,
\pi_s(a \mid c)\,\mu(da \mid c)\,
P_C(dc) \\
&=
\frac{1}{t-1}\sum_{s=1}^{t-1}
\int_{\mathcal C}
\int_{\mathcal A}
\int_{\mathcal Y}
\frac{\pi_{e}^{2}(a \mid c)}
{\pi_{t}(a \mid c)}
\hat{\varphi}_{s}^{2}(c,a,y)
\,p_Y(dy \mid c,a)\,
\mu(da \mid c)\,
P_C(dc) \\
&=
\frac{1}{t-1}\sum_{s=1}^{t-1}
\int_{\mathcal C}
\int_{\mathcal A}
\int_{\mathcal Y}
\left(
\frac{\pi_e(a \mid c)}
{\pi_t(a \mid c)}
\right)^2
\hat{\varphi}_{s}^{2}(c,a,y)
\,p_Y(dy \mid c,a)\,
\pi_t(a \mid c)\,\mu(da \mid c)\,
P_C(dc) \\
&=
\frac{1}{t-1}\sum_{s=1}^{t-1}
\E\left[
w_t^2\hat{\varphi}_{s}^{2}
\,\middle|\, \mathcal F_{t-1}
\right].
\end{align*}
{In the last conditional expectation, the historical fitted function $\hat\varphi_s$ is evaluated at the current observation $Z_t$; its fitted coefficients are already $\mathcal F_{t-1}$-measurable.} Therefore, we have that
\begin{align*}
\frac{1}{T} {\sum_{t=2}^{T}} ({\bar{\sigma}_{t}^{2}} -\E\left[\hat{\varphi}_{t}^{2} w_{t}^{2}|\Filt \right])^{2}
&= \frac{1}{T} {\sum_{t=2}^{T}} \left(\frac{1}{t-1}\sum_{s=1}^{t-1}\E\left[
w_t^2\hat{\varphi}_{s}^{2}
\,\middle|\, \mathcal F_{t-1}
\right] -\E\left[\hat{\varphi}_{t}^{2} w_{t}^{2}|\Filt \right]\right)^{2}\\
&= \frac{1}{T} {\sum_{t=2}^{T}} \left(\frac{1}{t-1}\sum_{s=1}^{t-1}\E\left[
w_t^2(\hat{\varphi}_{s}^{2}-\hat{\varphi}_{t}^{2})
\,\middle|\, \mathcal F_{t-1}
\right]\right)^{2}\\
&\le  \frac{1}{T}{\sum_{t=2}^{T}}  \left(\frac{1}{t-1}\sum_{s=1}^{t-1}{\norm{w_t}_\infty\E_{P_e}[|\hat\varphi_s^2-\hat\varphi_t^2|]}\right)^{2}\\
&\le  \frac{K_{T}^{2}}{T}{\sum_{t=2}^{T}}  \left(\frac{1}{t-1}\sum_{s=1}^{t-1}\E_{P_e}\left[
|\hat{\varphi}_{s}^{2}-\hat{\varphi}_{t}^{2}|
\right]\right)^{2}\\
\end{align*}
By the triangle inequality,
\[
{\E_{P_e}[|\hat\varphi_s^2-\hat\varphi_t^2|]}
\le
\E_{P_e}\big[|\hat\varphi_s^2 - \varphi_{P_e}^2|\big]
+
\E_{P_e}\big[|\hat\varphi_t^2 - \varphi_{P_e}^2|\big].
\]
We now apply Hardy's inequality \citep{Hardy} which we recount here for our reader's convenience. 
\begin{fact}[Hardy's Inequality] \label{fact:hardy} If $\{a_{n}\}_{n\in \mathbb{N}}$ is a sequence of non-negative real numbers, then for every real number $p>1$, 
\[
\sum_{n=1}^{\infty} \left( \frac{a_1 + a_2 + \cdots + a_n}{n} \right)^p
\le
\left( \frac{p}{p - 1} \right)^p \sum_{n=1}^{\infty} a_n^p,
\quad \text{for } p > 1.
\]
\end{fact}
\cref{fact:hardy} implies that
$
{\frac{1}{T}\sum_{t=2}^T}
\left(
\frac{1}{t-1}\sum_{s=1}^{t-1} a_s
\right)^2
\le
\frac{4}{T}\sum_{t=1}^T a_t^2
$. Therefore, applying Cauchy--Schwarz once more

\begingroup
\begin{align*}
&\frac{K_T^2}{T}\sum_{t=2}^T\left\{\frac1{t-1}\sum_{s<t}\E_{P_e}[|\hat\varphi_s^2-\hat\varphi_t^2|]\right\}^2\\
&\leq\frac{K_T^2}{T}\sum_{t=2}^T\left\{\E_{P_e}[|\hat\varphi_t^2-\varphi_{P_e}^2|]
+\frac1{t-1}\sum_{s<t}\E_{P_e}[|\hat\varphi_s^2-\varphi_{P_e}^2|]\right\}^2\\
&\leq\frac{2K_T^2}{T}\sum_{t=2}^T\{\E_{P_e}[|\hat\varphi_t^2-\varphi_{P_e}^2|]\}^2\\
&\quad+\frac{2K_T^2}{T}\sum_{t=2}^T\left\{\frac1{t-1}\sum_{s<t}\E_{P_e}[|\hat\varphi_s^2-\varphi_{P_e}^2|]\right\}^2\\
&\leq\frac{2K_T^2+8K_T^2}{T}\sum_{t=1}^T\{\E_{P_e}[|\hat\varphi_t^2-\varphi_{P_e}^2|]\}^2\\
&=\frac{10K_T^2}{T}\sum_{t=1}^T\{\E_{P_e}[|\hat\varphi_t^2-\varphi_{P_e}^2|]\}^2\\
&\leq10K_T^2\sup_{t\leq T}\E_{P_e}[(\hat\varphi_t+\varphi_{P_e})^2]
\cdot\frac1T\sum_{t=1}^T\E_{P_e}[(\hat\varphi_t-\varphi_{P_e})^2].
\end{align*}
\endgroup

Since $K_{T} = O(T^{1/8})$ it suffices to show that
$\frac{1}{T}\sum_{t=1}^T
\E_{P_e}\left[(\hat\varphi_t-\varphi_t)^2\right]
{=o_p(T^{-1/4})}$, and $\sup_{t\le T}
\E_{P_e}\!\left[(\hat\varphi_t+\varphi_t)^2\right]
=O_p(1)
$ in order to show the entire term is $o_p(1)$. However,  \cref{lemma:epe_f} combined with the requirement that ${\frac1T\sum_{t=1}^T(\hat\Psi_t-\Psi_e(\theta_{P_e}))^2} =o_p(T^{-1/2})$ immediately yield the stronger claim that $\frac{1}{T}\sum_{t=1}^T
\E_{P_e}\left[(\hat\varphi_t-\varphi_t)^2\right]
=o_p(T^{-1/2})$.
\iffalse
For the first display, by local smoothness of the influence function,
\[
|\hat\varphi_t(z)-\varphi_{P_e}(z)|
\lesssim
C(z)\left(
\|\hat\theta_t-\theta_{P_e}\|_\Theta
+
\|\hat\eta_t-\eta_{P_e}\|_\Upsilon
+
\|\hat\alpha_t-\alpha_{P_e}\|_\Theta
\right).
\]
Therefore,
\[
\E_{P_e}\!\left[(\hat\varphi_t-\varphi_{P_e})^2\right]
\lesssim
\|\hat\theta_t-\theta_{P_e}\|_\Theta^2
+
\|\hat\eta_t-\eta_{P_e}\|_\Upsilon^2
+
\|\hat\alpha_t-\alpha_{P_e}\|_\Theta^2,
\]
since \(\E{P_e}[C(Z)^2]<\infty\). Averaging over \(t\) and applying Cauchy--Schwarz,
\begin{align*}
\frac{1}{T}\sum_{t=1}^T
\E_{P_e}\!\left[(\hat\varphi_t-\varphi_t)^2\right]
&\lesssim
\left(
\frac{1}{T}\sum_{t=1}^T
\|\hat\theta_t-\theta_{P_e}\|_\Theta^4
\right)^{1/2}
+
\left(
\frac{1}{T}\sum_{t=1}^T
\|\hat\eta_t-\eta_{P_e}\|_\Upsilon^4
\right)^{1/2}
+
\left(
\frac{1}{T}\sum_{t=1}^T
\|\hat\alpha_t-\alpha_{P_e}\|_\Theta^4
\right)^{1/2}\\
&=o_p(1).
\end{align*}
\fi 
For the second term, note that by Cauchy-Schwarz and Jensen's inequality that, 
\begin{align*}
\sup_{t\in[T]} \E_{P_e}\left[ (\hat\varphi_t + \varphi_t)^{2}\right]&\le \left(\sup_{t\in[T]} \E_{P_e}\left[ (\hat\varphi_t + \varphi_t)^{4} \right]\right)^{1/2}
 \le \left(\sup_{t\in[T]} 8\E_{P_e}\left[\hat\varphi_t^{4}\right] + \sup_{t\in[T]} 8\E_{P_e}\left[\varphi_t^{4}\right]\right)^{1/2}.
\end{align*}
We can then invoke \cref{lemma:fourth_IF} to show $\sup_{t\in[T]}\E_{P_e}\left[\hat\varphi_t^{4}\right] =O_p(1)$ and $\sup_{t\in[T]}\E_{P_e}\left[\varphi_t^{4}\right]  =O(1)$.

\paragraph{ Showing $R_3 = o_p(1)$}
Using $\hat\varphi_t^2-\varphi_t^2
=
(\hat\varphi_t-\varphi_t)(\hat\varphi_t+\varphi_t)$ and Cauchy--Schwarz inequality,
\begin{align*}
\frac{2}{T}\sum_{t=1}^T
\left(
\E\left[
w_t^2(\hat\varphi_t-\varphi_t)(\hat\varphi_t+\varphi_t)
\mid \mathcal F_{t-1}
\right]
\right)^2 
&\le
\frac{2}{T}\sum_{t=1}^T
\E\left[
w_t^2(\hat\varphi_t-\varphi_t)^2
\mid \mathcal F_{t-1}
\right]
\E\left[
w_t^2(\hat\varphi_t+\varphi_t)^2
\mid \mathcal F_{t-1}
\right]\\
&\le {\frac{2K_T^2}{T}\sum_{t=1}^T} \E_{P_e}\left[\left(\hat\varphi_t -  \varphi_{t} \right)^{2}\right]\E_{P_e}\left[(\hat\varphi_t +  \varphi_{t})^{2} \right]
\end{align*}
The remaining argument is the same as what was shown in demonstrating $R_{2} = o_p(1)$

On the event that $\min_{t\in [T]}\hat{\sigma}_{t}^{2} >c$, observe that
\[
\frac{\sigma_t^2}{\hat\sigma_t^2}-1
=
\frac{\sigma_t^2-\hat\sigma_t^2}{\hat\sigma_t^2}.
\]
Hence, on the event $\min_{t\in[T]}\hat\sigma_t^2\ge c$,
\[
\left|
\frac1T\sum_{t=1}^T \left(\frac{\sigma_t^2}{\hat\sigma_t^2}-1\right)
\right|
\le
\frac{1}{c}\cdot \frac1T\sum_{t=1}^T |\hat\sigma_t^2-\sigma_t^2|.\]
Applying Cauchy--Schwarz gives
\[
\frac1T\sum_{t=1}^T |\hat\sigma_t^2-\sigma_t^2|
\le
\left(\frac1T\sum_{t=1}^T (\hat\sigma_t^2-\sigma_t^2)^2\right)^{1/2}.
\]
Therefore,
\[
\left|
\frac1T\sum_{t=1}^T \left(\frac{\sigma_t^2}{\hat\sigma_t^2}-1\right)
\right|
\le
\frac{1}{c}
\left(\frac1T\sum_{t=1}^T (\hat\sigma_t^2-\sigma_t^2)^2\right)^{1/2}
=o_p(1),
\]

\end{proof}

\section{Auxiliary Lemmas}
\iffalse
\begin{proposition} \label{fact:sigma}
Suppose there exist a constant $c >0$ such that $c \le \min_{t\in[T]} \sigma_t^2$ and $c \le \min_{t\in[T]} \hat{\sigma}_t^2$.

If 
\[
\frac1T\sum_{t=1}^T (\hat\sigma_t^2-\sigma_t^2)^2 = o_p(1),
\]
then
\[
\frac1T\sum_{t=1}^T \left(\frac{\sigma_t^2}{\hat\sigma_t^2}-1\right)=o_p(1).
\]
\end{proposition}

\begin{proof}
Observe that
\[
\frac{\sigma_t^2}{\hat\sigma_t^2}-1
=
\frac{\sigma_t^2-\hat\sigma_t^2}{\hat\sigma_t^2}.
\]
Hence, on the event $\min_{t\in[T]}\hat\sigma_t^2\ge c$,
\[
\left|
\frac1T\sum_{t=1}^T \left(\frac{\sigma_t^2}{\hat\sigma_t^2}-1\right)
\right|
\le
\frac{1}{c}\cdot \frac1T\sum_{t=1}^T |\hat\sigma_t^2-\sigma_t^2|.\]
Applying Cauchy--Schwarz gives
\[
\frac1T\sum_{t=1}^T |\hat\sigma_t^2-\sigma_t^2|
\le
\left(\frac1T\sum_{t=1}^T (\hat\sigma_t^2-\sigma_t^2)^2\right)^{1/2}.
\]
Therefore,
\[
\left|
\frac1T\sum_{t=1}^T \left(\frac{\sigma_t^2}{\hat\sigma_t^2}-1\right)
\right|
\le
\frac{1}{c}
\left(\frac1T\sum_{t=1}^T (\hat\sigma_t^2-\sigma_t^2)^2\right)^{1/2}
=o_p(1),
\]
which proves the claim.
\end{proof}
\fi

\begin{fact}[Theorem 6.6 of \cite{BIC_regret}] \label{fact:EMI}
Let $\Pi$ satisfy
\[
\log N(\epsilon,\log\Pi,\|\cdot\|_\infty)
\lesssim 
\begin{cases}
\log(e/\epsilon), & p=0,\\
\epsilon^{-p}, & p\in(0,2).
\end{cases}
\]
For each $s\le t$, let $g_s:\mathcal Z\to[0,B]$ be $\mathcal F_{s-1}$-measurable, and define
\[
f_{\pi,s}(Z_s)
=
\frac{g_s(Z_s)}
{\pi(A_s\mid C_s)}.
\]
Then there exists a universal constant $C>0$ such that, for any $t\ge 2$ and $\delta\in(0,1)$, with probability at least $1-\delta$,
\[
\sup_{\pi\in\Pi}
\left|
\frac1t\sum_{s=1}^t
\{f_{\pi,s}(Z_s)-\mathbb E[f_{\pi,s}(Z_s)\mid\mathcal F_{s-1}]\}
\right|
\le
C\left[
r_{B,t}^2
+
\tilde\rho_t(\pi)\, t^{-1/2}
+
\hat\rho_t(\pi)\sqrt{\frac{\log(et/\delta)}{t}}
+
B\frac{\log(et/\delta)}{t}
\right],
\]
uniformly over $\pi\in\Pi$, where
\[
\tilde\rho_t(\pi):=
\begin{cases}
\left(\frac{1}{t}\sum_{s=1}^t f_{\pi,s}(Z_s)^2\right)^{\frac{1-p/2}{2}} & \text{ when }  p \in (0,2),\\[10pt]
\left(\frac{1}{t}\sum_{s=1}^t f_{\pi,s}(Z_s)^2\right)^{1/2}\sqrt{\log t} & \text{ when }   p = 0.
\end{cases}
\]and $$r_{B,t}
\asymp
\begin{cases}
\left(\dfrac{B^{2-p}}{t}\right)^{\frac{2}{2+p}}  & \text{ when } p \in (0,2),\\[10pt]
B\sqrt{\dfrac{\log t}{t}} & \text{ when } p = 0.
\end{cases}$$

\end{fact}
\begin{proof}
This is a direct application of Theorem 6.6 of \cite{BIC_regret} with only minor notation changes to match our setting. 
\end{proof}

\begin{lemma} \label{prop:mds_clt} Let $\sigma_t$ be defined as in \cref{eqn:cond_var}. Assume that the following conditions hold. 
\begin{enumerate}[label=(\alph*)] 
\item A sequence of estimators $\{\hat{\sigma}_{t}\}_{t=1}^{T}$ adapted to $\mathcal{F}_{t-1}$ can be constructed such that there exists a constant $c>0$ such that $\min_{t\in [T]}\hat{\sigma}_{t}^{2} > c$ almost surely and $\frac{1}{T}\sum_{t=1}^{T}  \sigma_{t}^{2}/\hat{\sigma_{t}}^{2} -1 =o_{p}(1)$.\label{cond:estimator}
\item For every $\varepsilon > 0 $,
\begin{equation}  \label{eqn:lindeberg}
\frac{1}{T} \sum_{t=1}^T
\E\!\left[w_{t}^{2}\varphi_t^2\mathbf{1}\{|w_{t}\varphi_t| > \varepsilon \sqrt{T}\}
\mid \mathcal{F}_{t-1}\right] \xrightarrow{p} 0.
\end{equation}
\end{enumerate}
Then under Assumptions~\ref{assumption:distribution}-\ref{assumption:densities}, 
\begin{align}\label{eqn:martingale_CLT}
    \frac{1}{\sqrt{T}} \sum_{t=1}^{T} \hat{\sigma}_{t}^{-1}w_{t} \varphi_{t} \overset{d}{\to} N(0,1).
\end{align}
\end{lemma}

\begin{proof}
Define $\zeta_{t} := \hat{\sigma}_{t}^{-1}w_{t} \varphi_{t}$. By Theorem 3.4 of \cite{hall2014martingale}, we need to demonstrate three conditions:
\begin{enumerate}[label=(\alph*)]
    \item \textbf{Conditional Expectation} $\E[\zeta_{t} | \mathcal{F}_{t-1}] = 0$ for all $t\in [T]$;
    \item \textbf{Conditional Variance} $\frac{1}{T} \sum_{t=1}^{T}\E[\zeta_{t}^{2} | \mathcal{F}_{t-1}] = 1$
    \item \textbf{Conditional Lindeberg} $\frac{1}{T};\sum_{t=1}^{T}\E[\zeta_{t}^{2} \mathbf{1}\!\left\{|\zeta_{t}> \varepsilon \sqrt{T}\right\} | \mathcal{F}_{t-1}] =  o_p(1)$.
\end{enumerate}
The first condition follows from the fact that $\hat{\sigma_{t}}$ is $\mathcal{F}_{t-1}$ measurable so 
\begin{align*}
    E[\hat{\sigma}^{-1}w_{t}\varphi_t \mid \mathcal{F}_{t-1}] 
   =  \hat{\sigma}^{-1}E[w_{t}\varphi_t \mid \mathcal{F}_{t-1}] 
    =\hat{\sigma}^{-1} E_{P_e}[\varphi_{P_e}(Z)] 
    = 0.
\end{align*}
For the second condition, we have
\begin{align*}
    \frac{1}{T} \sum_{t=1}^{T}\E[\zeta_{t}^{2}\mid \mathcal{F}_{t-1}] =\frac{1}{T} \sum_{t=1}^{T}\E[\hat{\sigma}^{-2}w^{2}_{t}\varphi_t \mid \mathcal{F}_{t-1}] 
    =\frac{1}{T} \sum_{t=1}^{T} \frac{\sigma_{t}^{2}}{\hat\sigma_{t}^{2}} = 1 + o_{p}(1).
\end{align*}
For the third condition, we have a constant $c>0$ such that 
$\zeta_{t}^{2}  = \hat{\sigma}_{t}^{-2}w_{t}^{2} \varphi_{t}^{2} \le \frac{w_t^2}{c^2}\varphi_{t}^{2}$.  We therefore have for some $\varepsilon >0$ that, 

$$\frac{1}{T} \sum_{t=1}^T 
\mathbb{E}\!\left[
\zeta_t^2 \mathbf{1}\{|\zeta_t| > \varepsilon \sqrt{T}\}
\,\middle|\, \mathcal{F}_{t-1}
\right]
\;\le\;
c^{-2} \cdot \frac{1}{T} \sum_{t=1}^T 
\mathbb{E}\!\left[
w_t^{2}\varphi_t^2 \mathbf{1}\left\{|w_t\varphi_t| > \varepsilon\sqrt{T}\right\}
\,\middle|\, \mathcal{F}_{t-1}
\right] = o_p(1).$$

\end{proof}

\begin{lemma}[Self-normalized CLT for plug-in quantities]
\label{prop:self_normalized_plugin}
Let $\{D_{t}\}_{t=1}$ be a martingale difference sequence and $\hat{D}_{t}$ be a $\mathcal{F}_{t-1}$ measurable estimate of $D_{t}$. Define 
\[
S_T := \frac{1}{T}\sum_{t=1}^T D_t, \quad 
V_T := \frac{1}{T}\sum_{t=1}^T \E[D_t^2 \mid \mathcal F_{t-1}], \quad
\hat V_T := \frac{1}{T}\sum_{t=1}^T \hat D_t^2.
\]
Let $\hat{S}_{T}$ be a $\mathcal{F}_{t-1}$ measurable estimator of $S_T$. Suppose that the following conditions hold,  
\begin{enumerate}[label=(\alph*)]
\item (Conditional Variance) $V_T \overset{p}{\to} \kappa$ to some non-zero random variable $\kappa$ which is finite almost surely; \label{cond:cond_var_random}
\item (Realized Variance) $(TV_{T})^{-1}\sum_{t=1}^{T} D_{t}^{2} \overset{p}{\to} 1 $; \label{cond:self_relized_variance}
\item (Conditional Lindeberg)  For every $\varepsilon > 0$; \label{cond:self_relized_lindeberg}
$
\frac{1}{T}\sum_{t=1}^T
\E\!\left[
D_t^2 \mathbf{1}\{|D_t|>\varepsilon\sqrt{T}\}
\,\middle|\, \mathcal F_{t-1}
\right]
\to_p 0$;  
\item (Plug-In Rates) $\hat S_T - S_T = o_p(T^{-1/2})$, and $\frac{1}{T}\sum_{t=1}^T (\hat D_t - D_t)^2 = o_p(1)$. \label{cond:plugin_average}
\end{enumerate}

Then $\sqrt{T}\frac{\hat S_T}{\sqrt{\hat V_T}}
\overset{d}{\to}  N(0,1)$.
\end{lemma} 
\begin{proof}
Let $R_T := \frac{1}{T}\sum_{t=1}^T D_t^2$. By assumption, $\frac{R_T}{V_T} \pto 1$. 
Since condition (a) states that \(V_T \to_p \kappa\), with \(\kappa\) almost surely finite $R_{T} \pto \kappa $. Now define the martingale difference array $X_{T,t} := \frac{D_t}{\sqrt{T}}$. Note that 
\[
\sum_{t=1}^T \E[X_{T,t}^2 \mid \mathcal{F}_{t-1}]
=
\frac{1}{T}\sum_{t=1}^T \E[D_t^2 \mid \mathcal{F}_{t-1}]
=
V_T \pto \kappa.
\]

Moreover, the conditional Lindeberg assumption ensure that 
\[
\sum_{t=1}^T \E\!\left[X_{T,t}^2 \mathbf{1}\{|X_{T,t}| > \varepsilon\} \mid \mathcal{F}_{t-1}\right]
=
\frac{1}{T}\sum_{t=1}^T \E\!\left[D_t^2 \mathbf{1}\{|D_t| > \varepsilon \sqrt{T}\} \mid \mathcal{F}_{t-1}\right]
\pto 0.
\]

Hence, by Theorem 3.2 of \cite{hall2014martingale} we have that,
\[
\frac{1}{\sqrt{T}}\sum_{t=1}^T D_t \overset{st}{\rightarrow} \sqrt{\kappa}\, Z,
\]
where \(Z \sim N(0,1)\) and $\overset{st}{\rightarrow}$ denotes stable convergence. In particular, conditional on \(\kappa\), the limit is Gaussian with variance \(\kappa\), so the unconditional characteristic function is$\E\!\left[\exp\!\left(-\tfrac{1}{2}\kappa u^2\right)\right]$, corresponding to a mixture of centered Gaussian distributions.

Since $R_T \pto \kappa$, stable convergence and convergence in probability imply joint convergence:
$$
\left(\frac{1}{\sqrt{T}}\sum_{t=1}^T D_t,\; R_T\right)
\dto
(\sqrt{\kappa}\, Z,\; \kappa).
$$
Because \(\kappa > 0\) almost surely, the continuous mapping theorem yields
\[
\sqrt{T}\,\frac{S_T}{\sqrt{R_T}}
=
\frac{T^{-1/2}\sum_{t=1}^T D_t}{\sqrt{R_T}}
\dto
\frac{\sqrt{\kappa}\, Z}{\sqrt{\kappa}}
=
Z \sim N(0,1).
\]

It remains to replace \((S_T, R_T)\) with \((\hat{S}_T, \hat{V}_T)\). By condition~\ref{cond:plugin_average}, we have that$ \sqrt{T}\,(\hat{S}_T - S_T) \pto 0$. Next, observe that
\[
\hat{V}_T - R_T
=
\frac{1}{T}\sum_{t=1}^T (\hat{D}_t^2 - D_t^2)
=
\frac{1}{T}\sum_{t=1}^T (\hat{D}_t - D_t)^2
+
\frac{2}{T}\sum_{t=1}^T D_t(\hat{D}_t - D_t).
\]
The first term is \(o_p(1)\) by condition~\ref{cond:plugin_average}. For the second term, by Cauchy--Schwarz,
\[
\frac{1}{T}\left|\sum_{t=1}^T D_t(\hat{D}_t - D_t)\right|
\le
\left(\frac{1}{T}\sum_{t=1}^T D_t^2\right)^{1/2}
\left(\frac{1}{T}\sum_{t=1}^T (\hat{D}_t - D_t)^2\right)^{1/2}
=
R_T^{1/2} \cdot o_p(1)
=
o_p(1),
\]
since $R_T \pto \kappa$. Hence, $\hat{V}_T - R_T = o_p(1)$ and $\hat{V}_T \pto \kappa$. Finally,
\[
\sqrt{T}\,\frac{\hat{S}_T}{\sqrt{\hat{V}_T}}
=
\sqrt{T}\,\frac{S_T}{\sqrt{R_T}}
\left(\frac{R_T}{\hat{V}_T}\right)^{1/2}
+
\sqrt{T}\,\frac{\hat{S}_T - S_T}{\sqrt{\hat{V}_T}}.
\]
We have shown that $\sqrt{T}\,\frac{S_T}{\sqrt{R_T}} \dto N(0,1)$, $\frac{R_T}{\hat{V}_T} \pto 1$, and $\sqrt{T}\,\frac{\hat{S}_T - S_T}{\sqrt{\hat{V}_T}} \pto 0$. Thus, by Slutsky's theorem,
$\sqrt{T}\,\frac{\hat{S}_T}{\sqrt{\hat{V}_T}} \dto N(0,1)$.
\end{proof}

\begin{lemma} \label{lemma:remainder}
For any measurable
$f$, define the operator
$
P_T^{a}f
:=
\frac{1}{T}\sum_{t=1}^T a_t^{-1}w_tf(Z_t)$,
where the weights $\{a_t\}_{t=1}^T$ are
$\mathcal F_{t-1}$ measurable and satisfy $a_t\ge c>0$
uniformly in $t$, almost surely. Let
\[
F(\theta,\eta,g)(z)
:=
m_\theta(z)-\dot\ell_{\theta,\eta}(z)[g]
\]
and define
\[
R_T
:=
P_T^a\left(
F(\hat\theta_t,\hat\eta_t,\hat\alpha_t)
-
F(\theta_{P_e},\eta_{P_e},\alpha_{P_e})
\right).
\]
Suppose the following conditions hold:
\begin{enumerate}[label=(\alph*)]
\item
$\frac{1}{T}\sum_{t=1}^T
\|\hat\theta_t-\theta_{P_e}\|_\Theta^4=o_p(T^{-1})$,
$\frac{1}{T}\sum_{t=1}^T
\|\hat\eta_t-\eta_{P_e}\|_\Upsilon^4=o_p(T^{-1})$, and
$\frac{1}{T}\sum_{t=1}^T
\|\hat\alpha_t-\alpha_{P_e}\|_\Theta^4=o_p(T^{-1})$.

\item
\[
P_T^a\Big(
\dot m_{\theta_{P_e}}(\cdot)
[\hat\theta_t-\theta_{P_e}]
-
\ddot\ell_{\theta_{P_e},\eta_{P_e}}(\cdot)
[\hat\theta_t-\theta_{P_e},\alpha_{P_e}]
\Big)
=o_p(T^{-1/2}).
\]

\item
\[
P_T^a\Big(
\partial_\eta\partial_\theta
\ell_{\theta_{P_e},\eta_{P_e}}(\cdot)
[\alpha_{P_e},\hat\eta_t-\eta_{P_e}]
\Big)
=o_p(T^{-1/2}).
\]

\item
\[
P_T^a\Big(
\partial_\theta
\ell_{\theta_{P_e},\eta_{P_e}}(\cdot)
[\hat\alpha_t-\alpha_{P_e}]
\Big)
=o_p(T^{-1/2}).
\]

\item
\[
\frac{1}{T^2}\sum_{t=1}^T
\E\left[
w_t^4C(Z_t)^4\mid\mathcal F_{t-1}
\right]
=O_p(1).
\]
\end{enumerate}
Then, under
Assumptions~\ref{assumption:distribution}--\ref{assumption:envelope},
$R_T=o_p(T^{-1/2})$.
\end{lemma}

\begin{proof}
A first-order functional Taylor expansion of $F$ around
$(\theta_{P_e},\eta_{P_e},\alpha_{P_e})$ yields
\begin{align*}
&F(\hat\theta_t,\hat\eta_t,\hat\alpha_t)(z)
-
F(\theta_{P_e},\eta_{P_e},\alpha_{P_e})(z)
\\
&\quad=
\partial_\theta
F(\theta_{P_e},\eta_{P_e},\alpha_{P_e})(z)
[\hat\theta_t-\theta_{P_e}]
\\
&\qquad+
\partial_\eta
F(\theta_{P_e},\eta_{P_e},\alpha_{P_e})(z)
[\hat\eta_t-\eta_{P_e}]
\\
&\qquad+
\partial_g
F(\theta_{P_e},\eta_{P_e},\alpha_{P_e})(z)
[\hat\alpha_t-\alpha_{P_e}]
+
R_{2,t}(z),
\end{align*}
where $R_{2,t}(z)$ is a second-order remainder term.
In particular,
\[
\partial_\theta
F(\theta_{P_e},\eta_{P_e},\alpha_{P_e})(z)[h]
=
\dot m_{\theta_{P_e}}(z)[h]
-
\ddot\ell_{\theta_{P_e},\eta_{P_e}}(z)
[h,\alpha_{P_e}],
\]
\[
\partial_\eta
F(\theta_{P_e},\eta_{P_e},\alpha_{P_e})(z)[u]
=
-
\partial_\eta\partial_\theta
\ell_{\theta_{P_e},\eta_{P_e}}(z)
[\alpha_{P_e},u],
\]
and
\[
\partial_g
F(\theta_{P_e},\eta_{P_e},\alpha_{P_e})(z)[h]
=
-\dot\ell_{\theta_{P_e},\eta_{P_e}}(z)[h].
\]

Condition (a) implies that the maximum estimation error over
$t\le T$ converges to zero in probability for each of the three
estimators. Thus, with probability tending to one, the local bounds
in \cref{assumption:envelope} apply simultaneously for all $t\le T$.
On this event, after enlarging $C$ by a fixed constant if necessary,
\begin{align*}
|R_{2,t}(Z_t)|
\le C(Z_t)\Big(
&\|\hat\theta_t-\theta_{P_e}\|_\Theta^2
+
\|\hat\eta_t-\eta_{P_e}\|_\Upsilon^2
\\
&+
\|\hat\theta_t-\theta_{P_e}\|_\Theta
\|\hat\eta_t-\eta_{P_e}\|_\Upsilon
\\
&+
\|\hat\theta_t-\theta_{P_e}\|_\Theta
\|\hat\alpha_t-\alpha_{P_e}\|_\Theta
\\
&+
\|\hat\eta_t-\eta_{P_e}\|_\Upsilon
\|\hat\alpha_t-\alpha_{P_e}\|_\Theta
\Big).
\end{align*}
Note that there is no quadratic term in
$\|\hat\alpha_t-\alpha_{P_e}\|_\Theta$, since
$F(\theta,\eta,g)$ is affine in $g$.

Applying the operator $P_T^a$ to the Taylor expansion gives
\begin{align*}
R_T
&=
P_T^a\left[
\dot m_{\theta_{P_e}}(\cdot)
[\hat\theta_t-\theta_{P_e}]
-
\ddot\ell_{\theta_{P_e},\eta_{P_e}}(\cdot)
[\hat\theta_t-\theta_{P_e},\alpha_{P_e}]
\right]
\\
&\quad-
P_T^a\left[
\partial_\eta\partial_\theta
\ell_{\theta_{P_e},\eta_{P_e}}(\cdot)
[\alpha_{P_e},\hat\eta_t-\eta_{P_e}]
\right]
\\
&\quad-
P_T^a\left[
\dot\ell_{\theta_{P_e},\eta_{P_e}}(\cdot)
[\hat\alpha_t-\alpha_{P_e}]
\right]
+
P_T^aR_{2,t}.
\end{align*}
The first three terms on the right-hand side are each
$o_p(T^{-1/2})$ by conditions (b)--(d).
It therefore remains to control the final term.

Using the above bound on $R_{2,t}$,
\[
|P_T^aR_{2,t}|
\le
\frac{1}{T}\sum_{t=1}^T
a_t^{-1}w_tC(Z_t)\delta_t,
\]
where
\begin{align*}
\delta_t
:={}&
\|\hat\theta_t-\theta_{P_e}\|_\Theta^2
+
\|\hat\eta_t-\eta_{P_e}\|_\Upsilon^2
\\
&+
\|\hat\theta_t-\theta_{P_e}\|_\Theta
\|\hat\eta_t-\eta_{P_e}\|_\Upsilon
\\
&+
\|\hat\theta_t-\theta_{P_e}\|_\Theta
\|\hat\alpha_t-\alpha_{P_e}\|_\Theta
\\
&+
\|\hat\eta_t-\eta_{P_e}\|_\Upsilon
\|\hat\alpha_t-\alpha_{P_e}\|_\Theta.
\end{align*}
Since $w_t\ge0$, the Cauchy--Schwarz inequality gives
\[
|P_T^aR_{2,t}|
\le
\left(
\frac{1}{T}\sum_{t=1}^T
a_t^{-2}w_tC(Z_t)^2
\right)^{1/2}
\left(
\frac{1}{T}\sum_{t=1}^T
w_t\delta_t^2
\right)^{1/2}.
\]
It therefore suffices to show that the first average is $O_p(1)$
and the second is $o_p(T^{-1})$.

\paragraph{Showing
$\frac{1}{T}\sum_{t=1}^T a_t^{-2}w_tC(Z_t)^2=O_p(1)$.}
By predictability and the change-of-measure identity,
\[
\E\left[
a_t^{-2}w_tC(Z_t)^2
\mid\mathcal F_{t-1}
\right]
=
a_t^{-2}\E_{P_e}[C(Z)^2]
\le
c^{-2}\E_{P_e}[C(Z)^2].
\]
Since $C\in L^2(P_e)$, taking expectations and applying Markov's
inequality gives
\[
\frac{1}{T}\sum_{t=1}^T
a_t^{-2}w_tC(Z_t)^2=O_p(1).
\]

\paragraph{Showing
$\frac{1}{T}\sum_{t=1}^T w_t\delta_t^2=o_p(T^{-1})$.}
Since $(a_1+\cdots+a_5)^2\le5(a_1^2+\cdots+a_5^2)$,
\begin{align*}
\delta_t^2
\lesssim{}&
\|\hat\theta_t-\theta_{P_e}\|_\Theta^4
+
\|\hat\eta_t-\eta_{P_e}\|_\Upsilon^4
\\
&+
\|\hat\theta_t-\theta_{P_e}\|_\Theta^2
\|\hat\eta_t-\eta_{P_e}\|_\Upsilon^2
\\
&+
\|\hat\theta_t-\theta_{P_e}\|_\Theta^2
\|\hat\alpha_t-\alpha_{P_e}\|_\Theta^2
\\
&+
\|\hat\eta_t-\eta_{P_e}\|_\Upsilon^2
\|\hat\alpha_t-\alpha_{P_e}\|_\Theta^2.
\end{align*}
By condition (a), the averages of the fourth-power terms are
$o_p(T^{-1})$. The averages of the mixed terms are also
$o_p(T^{-1})$ by Cauchy--Schwarz. For example,
\begin{align*}
&\frac{1}{T}\sum_{t=1}^T
\|\hat\theta_t-\theta_{P_e}\|_\Theta^2
\|\hat\eta_t-\eta_{P_e}\|_\Upsilon^2
\\
&\quad\le
\left(
\frac{1}{T}\sum_{t=1}^T
\|\hat\theta_t-\theta_{P_e}\|_\Theta^4
\right)^{1/2}
\left(
\frac{1}{T}\sum_{t=1}^T
\|\hat\eta_t-\eta_{P_e}\|_\Upsilon^4
\right)^{1/2}
=
o_p(T^{-1}),
\end{align*}
and similarly for the remaining mixed terms. Hence
\[
\frac{1}{T}\sum_{t=1}^T\delta_t^2=o_p(T^{-1}),
\qquad\text{or equivalently}\qquad
\sum_{t=1}^T\delta_t^2=o_p(1).
\]

To obtain the corresponding weighted bound, note that $\delta_t$
is $\mathcal F_{t-1}$ measurable and
$\E[w_t\mid\mathcal F_{t-1}]=1$.
For any $b>0$, predictability therefore gives
\begin{align*}
&\E\left[
\sum_{t=1}^T w_t\delta_t^2
\mathbf{1}\left\{\sum_{s=1}^t\delta_s^2\le b\right\}
\right]
\\
&\quad=
\E\left[
\sum_{t=1}^T \delta_t^2
\mathbf{1}\left\{\sum_{s=1}^t\delta_s^2\le b\right\}
\right]
\le b.
\end{align*}
Consequently, for any $\varepsilon,b>0$, Markov's inequality yields
\[
\Pr\left(
\sum_{t=1}^T w_t\delta_t^2>\varepsilon
\right)
\le
\Pr\left(
\sum_{t=1}^T\delta_t^2>b
\right)
+
\frac{b}{\varepsilon}.
\]
Taking $T\to\infty$ and then $b\downarrow0$ proves
$\sum_{t=1}^T w_t\delta_t^2=o_p(1)$, and thus
\[
\frac{1}{T}\sum_{t=1}^T w_t\delta_t^2=o_p(T^{-1}).
\]

Combining the two bounds gives
$P_T^aR_{2,t}=o_p(T^{-1/2})$.
Together with conditions (b)--(d), this proves
$R_T=o_p(T^{-1/2})$.
\end{proof}
\begin{lemma} \label{lemma:empirical_predictable}
Let $\{f_t\}_{t=1}^T$ be a predictable sequence of measurable
maps $f_t:\mathcal Z\to\mathbb R$ such that
\begin{enumerate}[label=(\alph*)]
\item
$\E[w_tf_t(Z_t)\mid\mathcal F_{t-1}]=0$
for each $t\in[T]$ almost surely.

\item
There exists a measurable function
$C:\mathcal Z\to\mathbb R_+$ and nonnegative
$\mathcal F_{t-1}$ measurable variables $\delta_t$ such that
$|f_t(z)|\le C(z)\delta_t$ for all $z\in\mathcal Z$ and
$\frac{1}{T}\sum_{t=1}^T\delta_t^4=o_p(T^{-1})$.

\item
$\frac{1}{T^2}\sum_{t=1}^T
\big(\E[w_t^2C(Z_t)^2\mid\mathcal F_{t-1}]\big)^2
=O_p(1)$.
\end{enumerate}
Let $\{a_t\}_{t=1}^T$ be an $\mathcal F_{t-1}$ measurable
sequence such that $\inf_{t\in[T]}a_t>c$ almost surely for
some constant $c>0$.
Then, under
Assumptions~\ref{assumption:distribution}--\ref{assumption:reg_pe},
\[
\frac{1}{T}\sum_{t=1}^T a_t^{-1}w_tf_t(Z_t)
=
o_p(T^{-1/2}).
\]
\end{lemma}

\begin{proof}
Define $\zeta_t:=a_t^{-1}w_tf_t(Z_t)$.
Since $a_t^{-1}$ is $\mathcal F_{t-1}$ measurable,
condition (a) gives
$\E[\zeta_t\mid\mathcal F_{t-1}]=0$.
Thus $\{\zeta_t\}_{t=1}^T$ is a martingale difference sequence.
We first show that its normalized predictable variance satisfies
\[
V_T
:=
\frac{1}{T}\sum_{t=1}^T
\E[\zeta_t^2\mid\mathcal F_{t-1}]
=
o_p(1).
\]

By condition (b),
\[
|\zeta_t|^2
=
a_t^{-2}w_t^2f_t(Z_t)^2
\le
a_t^{-2}w_t^2C(Z_t)^2\delta_t^2.
\]
Taking conditional expectations and using predictability gives
\[
\E[\zeta_t^2\mid\mathcal F_{t-1}]
\le
a_t^{-2}\delta_t^2
\E[w_t^2C(Z_t)^2\mid\mathcal F_{t-1}].
\]
Moreover, $a_t^{-2}\le c^{-2}$ uniformly, so
\[
\E[\zeta_t^2\mid\mathcal F_{t-1}]
\le
c^{-2}\delta_t^2
\E[w_t^2C(Z_t)^2\mid\mathcal F_{t-1}].
\]

Applying Cauchy--Schwarz,
\begin{align*}
V_T
&\le
\frac{c^{-2}}{T}\sum_{t=1}^T
\delta_t^2
\E[w_t^2C(Z_t)^2\mid\mathcal F_{t-1}]
\\
&\le
c^{-2}
\left(\sum_{t=1}^T\delta_t^4\right)^{1/2}
\left(
\frac{1}{T^2}\sum_{t=1}^T
\big(
\E[w_t^2C(Z_t)^2\mid\mathcal F_{t-1}]
\big)^2
\right)^{1/2}.
\end{align*}
By condition (b), the first factor is $o_p(1)$, while
condition (c) makes the second factor $O_p(1)$.
Hence $V_T=o_p(1)$.

Finally, for any $\varepsilon,b>0$, a stopped martingale
second-moment bound gives
\[
\Pr\left(
\left|\frac{1}{\sqrt T}\sum_{t=1}^T\zeta_t\right|
>\varepsilon
\right)
\le
\Pr(V_T>b)+\frac{b}{\varepsilon^2}.
\]
Indeed, retain the increment $\zeta_t$ only when
$T^{-1}\sum_{s=1}^t
\E[\zeta_s^2\mid\mathcal F_{s-1}]\le b$.
This restriction is predictable, and the retained sum divided
by $\sqrt T$ has second moment at most $b$.
It agrees with the original sum whenever $V_T\le b$,
giving the displayed bound by Chebyshev's inequality.

Taking $T\to\infty$ and then $b\downarrow0$ yields
$T^{-1/2}\sum_{t=1}^T\zeta_t=o_p(1)$.
Therefore,
\[
\frac{1}{T}\sum_{t=1}^T a_t^{-1}w_tf_t(Z_t)
=
o_p(T^{-1/2}),
\]
as required.
\end{proof}
\begin{lemma} \label{lemma:bounded_riesz}
Under the conditions of \cref{thm:clt}, if $\|\alpha_{P_e}\|_\Theta<\infty$, then
\[
\sup_{t \in[ T]}\|\hat\alpha_t\|_\Theta=O_p(1).
\]
\end{lemma}

\begin{proof}
Since all terms are nonnegative,
\[
\sup_{t \in [T]}\|\hat\alpha_t-\alpha_{P_e}\|_\Theta^4
\le
\sum_{t=1}^T
\|\hat\alpha_t-\alpha_{P_e}\|_\Theta^4
=
T\cdot
\frac1T\sum_{t=1}^T
\|\hat\alpha_t-\alpha_{P_e}\|_\Theta^4
=o_p(1).
\]
Taking fourth roots gives that $\sup_{t \in [T]}\|\hat\alpha_t-\alpha_{P_e}\| =o_p(1)$. The final claim follows from
\[
\sup_{t\le T}\|\hat\alpha_t\|_\Theta
\le
\|\alpha_{P_e}\|_\Theta
+
\sup_{t\in [T]}\|\hat\alpha_t-\alpha_{P_e}\|_\Theta .
\]
\end{proof}
\begin{lemma} \label{lemma:fourth_IF}
Assume the conditions of \cref{thm:clt}. Let
$\hat\varphi_t(z):=
m_{\hat\theta_t}(z)-\hat\Psi_t
-\partial_\theta\ell_{\hat\theta_t,\hat\eta_t}(z)[\hat\alpha_t]$,
where $\hat\Psi_t$ is an $\mathcal F_{t-1}$ measurable estimator
such that $\sup_{t\in[T]}|\hat\Psi_t|=O_p(1)$.
Assume the overlap condition
$\sup_{\pi\in\Pi}\sup_{c,a}
\frac{\pi_e(a\mid c)}{\pi(a\mid c)}\le K_T$
for some $K_T<\infty$.
Writing $\varphi_t=\varphi_{P_e}(Z_t)$ for the oracle influence
function evaluated at $Z_t$, for any fixed $k\ge1$,
\[
\sup_{t\in[T]}
\E\left[w_t^k\hat\varphi_t^4(Z_t)\mid\mathcal F_{t-1}\right]
=O_p(K_T^{k-1})
\qquad\text{and}\qquad
\sup_{t\in[T]}
\E\left[w_t^k\varphi_t^4\mid\mathcal F_{t-1}\right]
=O_p(K_T^{k-1}).
\]
\end{lemma}

\begin{proof}
By \cref{lemma:bounded_riesz},
\[
\sup_{t\in[T]}\|\hat\alpha_t\|_\Theta=O_p(1).
\]
Furthermore, by condition~\ref{cond:rates} of \cref{thm:clt},
\[
\sup_{t\in[T]}
\|\hat\theta_t-\theta_{P_e}\|_\Theta^4
\le
\sum_{t=1}^T
\|\hat\theta_t-\theta_{P_e}\|_\Theta^4
=o_p(1),
\]
and
\[
\sup_{t\in[T]}
\|\hat\eta_t-\eta_{P_e}\|_\Upsilon^4
\le
\sum_{t=1}^T
\|\hat\eta_t-\eta_{P_e}\|_\Upsilon^4
=o_p(1).
\]
Taking fourth roots gives uniform consistency of both estimators.
Therefore, with probability tending to one,
$(\hat\theta_t,\hat\eta_t)$ lies in the neighborhoods where
\cref{assumption:envelope} holds, simultaneously for all
$t\in[T]$. In particular, on this event,
\[
|m_{\hat\theta_t}(z)|\le C(z),
\qquad \text{for all }t\in[T].
\]

Next, decompose
\[
\partial_\theta
\ell_{\hat\theta_t,\hat\eta_t}(z)[\hat\alpha_t]
=
\partial_\theta
\ell_{\theta_{P_e},\eta_{P_e}}(z)[\hat\alpha_t]
+
R_t(z),
\]
where, by the derivative bounds in \cref{assumption:envelope},
\[
|R_t(z)|
\le
C(z)\|\hat\alpha_t\|_\Theta
\left(
\|\hat\theta_t-\theta_{P_e}\|_\Theta
+
\|\hat\eta_t-\eta_{P_e}\|_\Upsilon
\right).
\]

For any $\varepsilon>0$, we can choose a fixed $L<\infty$
such that, for all sufficiently large $T$, with probability
at least $1-\varepsilon$, the estimators lie in these
neighborhoods and, uniformly over $t\in[T]$,
\[
\|\hat\alpha_t\|_\Theta\le L,
\quad
\|\hat\theta_t-\theta_{P_e}\|_\Theta\le1,
\quad
\|\hat\eta_t-\eta_{P_e}\|_\Upsilon\le1,
\quad
|\hat\Psi_t|\le L.
\]
On this event, $|R_t(z)|\le2LC(z)$ and
\[
\left|
\partial_\theta
\ell_{\theta_{P_e},\eta_{P_e}}(z)[\hat\alpha_t]
\right|
\le
C(z)\|\hat\alpha_t\|_\Theta
\le LC(z).
\]
Combining these bounds gives
\[
|\hat\varphi_t(z)|
\le
L'(1+C(z))
\]
for some finite constant $L'$ depending only on $L$,
uniformly over $t\in[T]$.

Since the estimators are $\mathcal F_{t-1}$ measurable,
the overlap bound and change-of-measure identity give
\[
\E\left[
w_t^k\hat\varphi_t^4(Z_t)\mid\mathcal F_{t-1}
\right]
\le
K_T^{k-1}
\E\left[
w_t\hat\varphi_t^4(Z_t)\mid\mathcal F_{t-1}
\right]
=
K_T^{k-1}\E_{P_e}[\hat\varphi_t^4(Z)].
\]
Consequently, on the preceding event,
\begin{align*}
\sup_{t\in[T]}
\E\left[
w_t^k\hat\varphi_t^4(Z_t)\mid\mathcal F_{t-1}
\right]
&\le
K_T^{k-1}
\sup_{t\in[T]}\E_{P_e}[\hat\varphi_t^4(Z)]
\\
&\le
(L')^4K_T^{k-1}
\E_{P_e}[(1+C(Z))^4].
\end{align*}
By \cref{assumption:envelope},
$\E_{P_e}[(1+C(Z))^4]<\infty$.
Since $\varepsilon>0$ was arbitrary, this proves
\[
\sup_{t\in[T]}
\E\left[
w_t^k\hat\varphi_t^4(Z_t)\mid\mathcal F_{t-1}
\right]
=
O_p(K_T^{k-1}).
\]

For the oracle influence function, the same envelope bounds give
\[
|\varphi_{P_e}(z)|
=
\left|
m_{\theta_{P_e}}(z)
-
\Psi_e
-
\partial_\theta
\ell_{\theta_{P_e},\eta_{P_e}}(z)[\alpha_{P_e}]
\right|
\le
L_0(1+C(z))
\]
for some fixed finite constant $L_0$.
Repeating the change-of-measure argument yields
\[
\sup_{t\in[T]}
\E\left[
w_t^k\varphi_t^4\mid\mathcal F_{t-1}
\right]
\le
L_0^4K_T^{k-1}\E_{P_e}[(1+C(Z))^4]
=
O_p(K_T^{k-1}),
\]
which completes the proof.
\end{proof}

\begin{lemma} \label{lemma:fourth_IF} Assume the conditions of \cref{thm:clt}. Let
$
\hat\varphi_t(z)
:=
m_{\hat\theta_t}(z)
-
\hat\Psi_t
-
\partial_\theta\ell_{\hat\theta_t,\hat\eta_t}(z)[\hat\alpha_t]$ where $\hat\Psi_t$ is a $\mathcal{F}_{t-1}$ measurable estimator such that $\sup_{t \in [T]} \hat\Psi_{t} = O_p(1)$. Assume the overlap condition $\sup_{\pi \in \Pi}\sup_{c,a} \frac{\pi_{e}}{\pi} \le K_{T}$ for some $K_{T}< \infty$.
Then, for any fixed $k\ge 1$, 
\[
\sup_{t\in[T]}
\mathbb E[w_{t}^{k}
\hat\varphi_t^4(Z_t)\mid\mathcal F_{t-1}
]
=
O_p(1) \qquad  \text{ and } \qquad \sup_{t\in[T]}
\mathbb E[w_{t}^{k}
\varphi_t^4(Z_t)\mid\mathcal F_{t-1}]=
O_p(1) .
\]
\end{lemma}

\begin{proof}
By \cref{lemma:bounded_riesz},
\[
\sup_{t\in[T]}\|\hat\alpha_t\|_\Theta=O_p(1).
\]
Furthermore, repeating the arguments in \cref{lemma:bounded_riesz} note that by condition~\ref{cond:rates} of \cref{thm:clt}, we have that
$$\sup_{t \in [T]} \norm{\hat\theta_t - \theta_{P_e}}_{\Theta}^{4} \le \sum_{t=1}^{T} \norm{\hat\theta_t - \theta_{P_e}}_{\Theta}^{4} = o_p(1)$$
$$\sup_{t \in [T]} \norm{\hat\eta_t - \eta_{P_e}}_{\Upsilon}^{4} \le \sum_{t=1}^{T} \norm{\hat\eta_{t} - \eta_{P_e}}^{4}_{\Upsilon} = o_p(1)$$
by the same arguments. Then, by taking fourth roots, we have that $\sup_{t \in [T]} \norm{\hat\theta_t - \theta_{P_e}}_{\Theta} =o_p(1)$. Therefore,  $(\hat\theta_t,\hat\eta_t)$ lies in a neighborhood where the envelope condition of \cref{assumption:reg_pe} holds.

In particular,
\[
|m_{\hat\theta_t}(z)| \le C(z), \qquad \text{for all } t \in [T].
\]

Next, decompose
\[
\partial_\theta \ell_{\hat\theta_t,\hat\eta_t}(z)[\hat\alpha_t]
=
\partial_\theta \ell_{\theta_{P_e},\eta_{P_e}}(z)[\hat\alpha_t]
+
R_t(z),
\]
where, by the same arguments as \cref{thm:pe_expansion}, 
\[
|R_t(z)|
\le
C(z)\|\hat\alpha_t\|_\Theta
\left(
\|\hat\theta_t - \theta_{P_e}\|_\Theta
+
\|\hat\eta_t - \eta_{P_e}\|_\Upsilon
\right).
\]

Therefore, on an event with probability tending to $1$, there exists a constant $L < \infty$ such that, uniformly over $t \in [T]$,
\[
\|\hat\alpha_t\|_\Theta \le L,
\quad
\|\hat\theta_t - \theta_{P_e}\|_\Theta \le 1,
\quad
\|\hat\eta_t - \eta_{P_e}\|_\Upsilon \le 1,
\quad
|\hat\Psi_t| \le L.
\]

Thus, $|R_t(z)| \le L C(z)$ and $|\partial_\theta \ell_{\theta_{P_e},\eta_{P_e}}(z)[\hat\alpha_t]|
\le C(z)\|\hat\alpha_t\|_\Theta \le L C(z)$. Combining everything, we have that 
\[
|\hat\varphi_t(z)|
=
|m_{\hat\theta_t}(z) - \hat\Psi_t - \partial_\theta \ell_{\hat\theta_t,\hat\eta_t}(z)[\hat\alpha_t]|
\le
L' (1 + C(z))
\]
for some finite constant $L'$, uniformly over $t \in [T]$ on an event with probability tending to $1$. Therefore, on this event, 
\begin{align*}
\sup_{t \in [T]}\mathbb{E}\left[w_t^k \hat\varphi_t^4(Z_t)\mid \mathcal{F}_{t-1}\right]
&\le 
\sup_{t \in [T]}(L')^4 \mathbb{E}\!\left[w_t^k (1 + C(Z_t))^4 \mid \mathcal{F}_{t-1}\right] \\
& \le \sup_{t \in [T]}(L')^4 \mathbb{E}\!\left[ \left(\sup_{t \in [T]} w_t^{k-1}\right)w_t (1 + C(Z_t))^4 \mid \mathcal{F}_{t-1}\right] \\ 
&\le \sup_{t \in [T]}(L')^4 K_{T}^{k-1} \mathbb{E}_{P_e}\!\left[ (1 + C(Z))^4 \right].
\end{align*}
However, $\mathbb{E}_{P_e}\!\left[ (1 + C(Z))^4 \right]$ is bounded by assumption so the entire term is $O_p(1)$. For the oracle influence function, applying the envelope condition again gives us 
\[
|\varphi_t|
=
|m_{\theta_{P_e}}(Z_t)-\Psi_e-\partial_\theta\ell_{\theta_{P_e},\eta_{P_e}}(Z_t)[\alpha_{P_e}]|
\le
L(1+C(Z_t))
\]
for some finite constant \(L\). Repeating the same steps as above gives us that this term is $O_{p}(1)$.

\end{proof}

\begin{lemma} \label{lemma:epe_f}
Let $F(\theta,\eta,\alpha)(z)
:=
m_{\theta_t}(z)
-
\partial_\theta \ell_{\theta,\eta}(z)[\alpha]$
and 
$F_0(z):= F(\theta_{P_e},\eta_{P_e},\alpha_{P_e})$. 
Then under the conditions of \cref{thm:clt},
\[
\frac1T\sum_{t=1}^T
\E_{P_e}\left[ \left(F(\hat\theta_t,\hat\eta_t,\hat\alpha_t)(Z) - F_{0}(Z)\right)^2 \right]
=
o_p(T^{-1/2}).
\]
\end{lemma}
\begin{proof}
We decompose
\[
F(\hat\theta_t,\hat\eta_t,\hat\alpha_t)(z) - F_{0}(z)
=
\{m_{\hat\theta_t}(z)-m_{\theta_{P_e}}(z)\}
-
\left\{
\partial_\theta\ell_{\hat\theta_t,\hat\eta_t}[\hat\alpha_t](z)
-
\partial_\theta\ell_{\theta_{P_e},\eta_{P_e}}[\alpha_{P_e}](z)
\right\}.
\]
We have using \cref{assumption:envelope} that 
\begin{align*}
    |m_{\hat\theta_t}(z)-m_{\theta_{P_e}}(z)|
\leq
\int_0^1 C(z)\|\hat\theta_t - \theta_{P_e}\|_\Theta\,ds
=
C(z)\|\hat\theta_t - \theta_{P_e}\|_\Theta.
\end{align*}
Therefore, 
\[
\E_{P_e}[(m_{\hat\theta_t}-m_{\theta_{P_e}})^2]
\leq
\E_{P_e}[C(Z)^2]\|\hat\theta_t - \theta_{P_e}\|_\Theta^2 .
\] 
Next, write
\begin{align*}
    \partial_\theta\ell_{\hat\theta_t,\hat\eta_t}[\hat\alpha_t]
-
\partial_\theta\ell_{\theta_{P_e},\eta_{P_e}}[\alpha_{P_e}]
&=
\partial_\theta\ell_{\hat\theta_t,\hat\eta_t}[\Delta_{\alpha,t}] +\left\{
\partial_\theta\ell_{\hat\theta_t,\hat\eta_t}
-
\partial_\theta\ell_{\theta_{P_e},\eta_{P_e}}
\right\}[\alpha_{P_e}] \\
&=\partial_\theta\ell_{\theta_{P_e},\eta_{P_e}}[\Delta_{\alpha,t}]
+
\left\{
\partial_\theta\ell_{\hat\theta_t,\hat\eta_t}
-
\partial_\theta\ell_{\theta_{P_e},\eta_{P_e}}
\right\}[\Delta_{\alpha,t}] +\left\{
\partial_\theta\ell_{\hat\theta_t,\hat\eta_t}
-
\partial_\theta\ell_{\theta_{P_e},\eta_{P_e}}
\right\}[\alpha_{P_e}]\\
\end{align*}
For the first term, we have by \cref{assumption:envelope} that
$
\E_{P_e}\left[
\left\{
\partial_\theta\ell_{\theta_{P_e},\eta_{P_e}}[\alpha_t - \alpha_{P_e}]
\right\}^2
\right]
\lesssim
\|\alpha_t - \alpha_{P_e}\|_\Theta^2$. For the second and third term, consider that

\begin{align*}
\left\{
\partial_\theta\ell_{\hat\theta_t,\hat\eta_t}
-
\partial_\theta\ell_{\theta_{P_e},\eta_{P_e}}
\right\}[h]
&=
\underbrace{\int_0^1
\partial_\theta^2
\ell_{\theta_{P_e}+s\Delta_{\theta,t},
      \eta_{P_e}+s\Delta_{\eta,t}}
[\Delta_{\theta,t},h]
\,ds}_{R_1t} \\
& \qquad +
\underbrace{\int_0^1
\partial_\eta\partial_\theta
\ell_{\theta_{P_e}+s\Delta_{\theta,t},
      \eta_{P_e}+s\Delta_{\eta,t}}
[h,\Delta_{\eta,t}]
\,ds}_{R_{2t}} .
\end{align*}
Again applying the envelope condition \cref{assumption:envelope},
\[
\left|
\left\{
\partial_\theta\ell_{\hat\theta_t,\hat\eta_t}
-
\partial_\theta\ell_{\theta_{P_e},\eta_{P_e}}
\right\}[h]
\right|
\lesssim
C(Z)
\left(
\|\hat\theta_t-\theta_{P_e}\|_\Theta
+
\|\hat\eta_t-\eta_{P_e}\|_\Upsilon
\right)
\|h\|_\Theta.
\]

Taking \(h=\hat\alpha_t-\alpha_{P_e}\) yields
\[
\begin{aligned}
\E_{P_e}[R_{1t}^2]
&\lesssim
\|\hat\alpha_t-\alpha_{P_e}\|_\Theta^2
\\
&\quad+
\|\hat\alpha_t-\alpha_{P_e}\|_\Theta^2
\left(
\|\hat\theta_t-\theta_{P_e}\|_\Theta
+
\|\hat\eta_t-\eta_{P_e}\|_\Upsilon
\right)^2.
\end{aligned}
\]

Similarly, taking \(h=\alpha_{P_e}\) gives
\[
\begin{aligned}
\E_{P_e}[R_{2t}^2]
&\lesssim
\|\alpha_{P_e}\|_\Theta^2
\left(
\|\hat\theta_t-\theta_{P_e}\|_\Theta
+
\|\hat\eta_t-\eta_{P_e}\|_\Upsilon
\right)^2.
\end{aligned}
\]

Combining the preceding bounds,
\[
\begin{aligned}
\E_{P_e}
\!\left[
\left(
F_t(\hat\theta_t,\hat\eta_t,\hat\alpha_t)(Z)
-
F_0(Z)
\right)^2
\right]
&\lesssim
\|\hat\theta_t-\theta_{P_e}\|_\Theta^2
+
\|\hat\eta_t-\eta_{P_e}\|_\Upsilon^2
+
\|\hat\alpha_t-\alpha_{P_e}\|_\Theta^2
\\
& \qquad+
\|\hat\alpha_t-\alpha_{P_e}\|_\Theta^2
\left(
\|\hat\theta_t-\theta_{P_e}\|_\Theta
+
\|\hat\eta_t-\eta_{P_e}\|_\Upsilon
\right)^2 .
\end{aligned}
\]

Averaging over \(t\), condition \ref{cond:rates} of \cref{thm:clt} implies
\[
\frac1T\sum_{t=1}^T
\|\hat\theta_t-\theta_{P_e}\|_\Theta^2
\leq
\left(
\frac1T\sum_{t=1}^T
\|\hat\theta_t-\theta_{P_e}\|_\Theta^4
\right)^{1/2}=
o_p(T^{-1/2}),
\]
and similarly for \(\hat\eta_t\) and \(\hat\alpha_t\). Moreover,

\begin{align*}
\frac1T\sum_{t=1}^T
\|\hat\alpha_t-\alpha_{P_e}\|_\Theta^2
\|\hat\theta_t-\theta_{P_e}\|_\Theta^2
&\leq
\left(
\frac1T\sum_{t=1}^T
\|\hat\alpha_t-\alpha_{P_e}\|_\Theta^4
\right)^{1/2}
\left(
\frac1T\sum_{t=1}^T
\|\hat\theta_t-\theta_{P_e}\|_\Theta^4
\right)^{1/2}
\\
&=
o_p(T^{-1}),
\end{align*}
and the same argument yields
$
\frac1T\sum_{t=1}^T
\|\hat\alpha_t-\alpha_{P_e}\|_\Theta^2
\|\hat\eta_t-\eta_{P_e}\|_\Upsilon^2
=
o_p(T^{-1})$.

Therefore,
\[
\frac1T\sum_{t=1}^T
\E_{P_e}
\!\left[
\left(
F_t(\hat\theta_t,\hat\eta_t,\hat\alpha_t)(Z)
-
F_0(Z)
\right)^2
\right]
=
o_p(T^{-1/2}).
\]

\end{proof}

%% file: ref.bib
@book{bickel1993efficient,
  title={Efficient and adaptive estimation for semiparametric models},
  author={Bickel, Peter J and Klaassen, Chris AJ},
  volume={4},
  publisher={Springer},
  year={1993}
}

@article{https://doi.org/10.3982/ECTA18515,
author = {Chernozhukov, Victor and Newey, Whitney K. and Singh, Rahul},
title = {Automatic Debiased Machine Learning of Causal and Structural Effects},
journal = {Econometrica},
volume = {90},
number = {3},
pages = {967-1027},
doi = {https://doi.org/10.3982/ECTA18515},
url = {https://onlinelibrary.wiley.com/doi/abs/10.3982/ECTA18515},
eprint = {https://onlinelibrary.wiley.com/doi/pdf/10.3982/ECTA18515},
year = {2022}
}

@inproceedings{NEURIPS2020_6fd86e0a,
 author = {Zhang, Kelly and Janson, Lucas and Murphy, Susan},
 booktitle = {Advances in Neural Information Processing Systems},
 pages = {9818--9829},
 title = {Inference for Batched Bandits},
 url = {https://proceedings.neurips.cc/paper_files/paper/2020/file/6fd86e0ad726b778e37cf270fa0247d7-Paper.pdf},
 volume = {33},
 year = {2020}
}

@article{
2019hadad,
author = {Vitor Hadad  and David A. Hirshberg  and Ruohan Zhan  and Stefan Wager  and Susan Athey },
title = {Confidence intervals for policy evaluation in adaptive experiments},
journal = {Proceedings of the National Academy of Sciences},
volume = {118},
number = {15},
pages = {e2014602118},
year = {2021},
doi = {10.1073/pnas.2014602118},
URL = {https://www.pnas.org/doi/abs/10.1073/pnas.2014602118},
eprint = {https://www.pnas.org/doi/pdf/10.1073/pnas.2014602118}}

@article{10.1093/biomet/asaa054,
    author = {Rotnitzky, A and Smucler, E and Robins, J M},
    title = {Characterization of parameters with a mixed bias property},
    journal = {Biometrika},
    volume = {108},
    number = {1},
    pages = {231-238},
    year = {2021},
    month = {03},
    issn = {0006-3444},
    doi = {10.1093/biomet/asaa054},
    url = {https://doi.org/10.1093/biomet/asaa054},
    eprint = {https://academic.oup.com/biomet/article-pdf/108/1/231/36441108/asaa054.pdf},
}

@article{10.1214/15-AOS1384,
author = {Alexander R. Luedtke and Mark J. van der Laan},
title = {{Statistical inference for the mean outcome under a possibly non-unique optimal treatment strategy}},
volume = {44},
journal = {The Annals of Statistics},
number = {2},
publisher = {Institute of Mathematical Statistics},
pages = {713 -- 742},
year = {2016},
doi = {10.1214/15-AOS1384},
URL = {https://doi.org/10.1214/15-AOS1384}
}

@article{dvoretzky1972,
      title={Asymptotic normality for sums of dependent random variables}, 
      author={Aryeh Dvoretzky},
      year={1972},
      journal={Berkeley Symp. on Math. Statist. and Prob.},
      primaryClass={cs.LG}
}

@article{zhang2023statistical,
      title={Statistical Inference After Adaptive Sampling for Longitudinal Data}, 
      author={Kelly W. Zhang and Lucas Janson and Susan A. Murphy},
      year={2023},
      journal={arXiv:2202.07098},
      primaryClass={cs.LG}
}

@inproceedings{cook2023semiparametric,
 author = {Cook, Thomas and Mishler, Alan and Ramdas, Aaditya},
 booktitle = {Proceedings of the Third Conference on Causal Learning and Reasoning},
 pages = {1033--1064},
 series = {Proceedings of Machine Learning Research},
 title = {Semiparametric Efficient Inference in Adaptive Experiments},
 url = {https://proceedings.mlr.press/v236/cook24a.html},
 volume = {236},
 year = {2024}
}

@inproceedings{NEURIPS2023_a399456a,
 author = {Ying, Mufang and Khamaru, Koulik and Zhang, Cun-Hui},
 booktitle = {Advances in Neural Information Processing Systems},
 pages = {52051--52072},
 title = {Adaptive Linear Estimating Equations},
 url = {https://proceedings.neurips.cc/paper_files/paper/2023/file/a399456a191ca36c7c78dff367887f0a-Paper-Conference.pdf},
 volume = {36},
 year = {2023}
}

@article{Hardy, title={Inequalities}, volume={37},author = {G. Hardy and J.E. Littlewood and G. Pólya},  DOI={10.1017/S0025557200027455}, number={321}, journal={The Mathematical Gazette}, author={E. C. T.}, year={1953}, pages={236–236}}

@book{van2000asymptotic,
author={Aad van der Vaart},
  title={Asymptotic statistics},
  volume={3},
  year={2000},
  publisher={Cambridge university press}
}

@book{hall2014martingale,
  title={Martingale limit theory and its application},
  author={Hall, Peter and Heyde, Christopher C},
  year={1980},
  publisher={Academic press}
}

@article{afsar2022reinforcement,
  title={Reinforcement learning based recommender systems: A survey},
  author={Afsar, M Mehdi and Crump, Trafford and Far, Behrouz},
  journal={ACM Computing Surveys},
  volume={55},
  number={7},
  pages={1--38},
  year={2022},
  publisher={ACM New York, NY}
}

@article{info:doi/10.2196/18477,
author="Liu, Siqi
and See, Kay Choong
and Ngiam, Kee Yuan
and Celi, Leo Anthony
and Sun, Xingzhi
and Feng, Mengling",
title="Reinforcement Learning for Clinical Decision Support in Critical Care: Comprehensive Review",
journal="J Med Internet Res",
year="2020",
month="Jul",
day="20",
volume="22",
number="7",
pages="e18477",
issn="1438-8871",
doi="10.2196/18477",
url="https://www.jmir.org/2020/7/e18477",
url="https://doi.org/10.2196/18477",
url="http://www.ncbi.nlm.nih.gov/pubmed/32706670"
}

@article{https://doi.org/10.1002/sim.3720,
author = {Zhao, Yufan and Kosorok, Michael R. and Zeng, Donglin},
title = {Reinforcement learning design for cancer clinical trials},
journal = {Statistics in Medicine},
volume = {28},
number = {26},
pages = {3294-3315},
doi = {https://doi.org/10.1002/sim.3720},
url = {https://onlinelibrary.wiley.com/doi/abs/10.1002/sim.3720},
eprint = {https://onlinelibrary.wiley.com/doi/pdf/10.1002/sim.3720},
year = {2009}
}

@article{info:doi/10.2196/jmir.7994,
author="Yom-Tov, Elad
and Feraru, Guy
and Kozdoba, Mark
and Mannor, Shie
and Tennenholtz, Moshe
and Hochberg, Irit",
title="Encouraging Physical Activity in Patients With Diabetes: Intervention Using a Reinforcement Learning System",
journal="J Med Internet Res",
year="2017",
month="Oct",
day="10",
volume="19",
number="10",
pages="e338",
issn="1438-8871",
doi="10.2196/jmir.7994",
url="http://www.jmir.org/2017/10/e338/",
url="https://doi.org/10.2196/jmir.7994",
url="http://www.ncbi.nlm.nih.gov/pubmed/29017988"
}

@article{https://doi.org/10.3982/ECTA17527,
author = {Kasy, Maximilian and Sautmann, Anja},
title = {Adaptive Treatment Assignment in Experiments for Policy Choice},
journal = {Econometrica},
volume = {89},
number = {1},
pages = {113-132},
doi = {https://doi.org/10.3982/ECTA17527},
url = {https://onlinelibrary.wiley.com/doi/abs/10.3982/ECTA17527},
eprint = {https://onlinelibrary.wiley.com/doi/pdf/10.3982/ECTA17527},
year = {2021}
}

@article{10.1093/jeea/jvad067,
    author = {Caria, A Stefano and Gordon, Grant and Kasy, Maximilian and Quinn, Simon and Shami, Soha Osman and Teytelboym, Alexander},
    title = {An Adaptive Targeted Field Experiment: Job Search Assistance for Refugees in Jordan},
    journal = {Journal of the European Economic Association},
    volume = {22},
    number = {2},
    pages = {781-836},
    year = {2024},
    month = {04},
    issn = {1542-4766},
    doi = {10.1093/jeea/jvad067},
    url = {https://doi.org/10.1093/jeea/jvad067},
    eprint = {https://academic.oup.com/jeea/article-pdf/22/2/781/57146869/jvad067.pdf},
}

@article{10.1111/ectj.12097,
    author = {Chernozhukov, Victor and Chetverikov, Denis and Demirer, Mert and Duflo, Esther and Hansen, Christian and Newey, Whitney and Robins, James},
    title = {Double/debiased machine learning for treatment and structural parameters},
    journal = {The Econometrics Journal},
    volume = {21},
    number = {1},
    pages = {C1-C68},
    year = {2018},
    month = {01},
    issn = {1368-4221},
    doi = {10.1111/ectj.12097},
    url = {https://doi.org/10.1111/ectj.12097},
    eprint = {https://academic.oup.com/ectj/article-pdf/21/1/C1/27684918/ectj00c1.pdf},
}

@article{3f41bc42-f438-35b7-8504-3723ce0fea5f,
 ISSN = {00029939, 10886826},
 URL = {http://www.jstor.org/stable/2034876},
 author = {Patrick Billingsley},
 journal = {Proceedings of the American Mathematical Society},
 number = {5},
 pages = {788--792},
 publisher = {American Mathematical Society},
 title = {The Lindeberg-Lévy Theorem for Martingales},
 urldate = {2026-05-06},
 volume = {12},
 year = {1961}
}

@article{bibault2021postcontext,
author = {Bibaut, Aurelien and Chambaz, Antoine and Dimakopoulou, Maria and Kallus, Nathan and Laan, Mark},
year = {2021},
month = {12},
pages = {28548-28559},
title = {Post-Contextual-Bandit Inference},
volume = {34},
journal = {Advances in Neural Information Processing Systems}
}

@article{kato2020standardized,
    title={Confidence Interval for Off-Policy Evaluation from Dependent Samples via Bandit Algorithm: Approach from Standardized Martingales},
      author={Masahiro Kato},
      year={2020},
        journal = {arXiv:2006.06982},
      archivePrefix={arXiv},
      primaryClass={stat.ML}
}

@inproceedings{NEURIPS2019_65b1e92c,
 author = {Shin, Jaehyeok and Ramdas, Aaditya and Rinaldo, Alessandro},
 booktitle = {Advances in Neural Information Processing Systems},
 title = {Are sample means in multi-armed bandits positively or negatively biased?},
 year = {2019}
}

@InProceedings{kato2021efficient,
  title = 	 {Active Adaptive Experimental Design for Treatment Effect Estimation with Covariate Choice},
  author =       {Kato, Masahiro and Oga, Akihiro and Komatsubara, Wataru and Inokuchi, Ryo},
  booktitle = 	 {Proceedings of the 41st International Conference on Machine Learning},
  pages = 	 {23291--23323},
  year = 	 {2024},
  url = 	 {https://proceedings.mlr.press/v235/kato24a.html}
}

@article{vanderlaan2026automaticdebiasedmachinelearning,
      title={Automatic Debiased Machine Learning for Smooth Functionals of Nonparametric M-Estimands}, 
      author={Lars van der Laan and Aurelien Bibaut and Nathan Kallus and Alex Luedtke},
      year={2026},
      journal={arXiv:2501.11868},
      primaryClass={stat.ME},
      url={https://arXiv.org/abs/2501.11868}, 
}

@article{10.1214/24-AOS2485,
author = {Licong Lin and Koulik Khamaru and Martin J. Wainwright},
title = {{Semiparametric inference based on adaptively collected data}},
volume = {53},
journal = {The Annals of Statistics},
number = {3},
publisher = {Institute of Mathematical Statistics},
pages = {989 -- 1014},
year = {2025},
doi = {10.1214/24-AOS2485},
URL = {https://doi.org/10.1214/24-AOS2485}
}

@article{guo2025statisticalinferencemisspecifiedcontextual,
      title={Statistical Inference for Misspecified Contextual Bandits}, 
      author={Yongyi Guo and Ziping Xu},
      year={2025},
      journal={arXiv:2509.06287},
      archivePrefix={arXiv},
      primaryClass={math.ST},
      url={https://arXiv.org/abs/2509.06287}, 
}

@InProceedings{pmlr-v80-deshpande18a,
  title = 	 {Accurate Inference for Adaptive Linear Models},
  author =       {Deshpande, Yash and Mackey, Lester and Syrgkanis, Vasilis and Taddy, Matt},
  booktitle = 	 {Proceedings of the 35th International Conference on Machine Learning},
  pages = 	 {1194--1203},
  year = 	 {2018},
  volume = 	 {80},
  series = 	 {Proceedings of Machine Learning Research},
  url = 	 {https://proceedings.mlr.press/v80/deshpande18a.html}
}

@inproceedings{
BIC_regret,
title={Fast Best-in-Class Regret for Contextual Bandits},
author={Samuel Girard and Aurelien Bibaut and Jill-J{\^e}nn Vie and Arthur Gretton and Nathan Kallus and Houssam Zenati},
booktitle={Forty-Second Annual Conference on Uncertainty in Artificial Intelligence},
year={2026},
url={https://openreview.net/forum?id=pmPgWIs4qb}
}

@article{leiner2026adaptiveoffpolicyinferencemestimators,
      title={Adaptive Off-Policy Inference for M-Estimators Under Model Misspecification}, 
      author={James Leiner and Robin Dunn and Aaditya Ramdas},
      year={2026},
      journal={arXiv:2509.14218},
      archivePrefix={arXiv},
      primaryClass={stat.ME},
      url={https://arXiv.org/abs/2509.14218}, 
}

@article{1982laiwei, 
 ISSN = {00905364},
 URL = {http://www.jstor.org/stable/2240506},
 author = {Tze Leung Lai and Ching Zong Wei},
 journal = {Annals of Statistics},
 number = {1},
 pages = {154--166},
 publisher = {Institute of Mathematical Statistics},
 title = {Least Squares Estimates in Stochastic Regression Models with Applications to Identification and Control of Dynamic Systems},
 urldate = {2024-03-18},
 volume = {10},
 year = {1982}
}

@article{whitehouse2026inferenceoptimalpolicyvalues,
      title={Inference on Optimal Policy Values and Other Irregular Functionals via Softmax Smoothing}, 
      author={Justin Whitehouse and Qizhao Chen and Morgane Austern and Vasilis Syrgkanis},
      year={2026},
      journal = {arXiv:2507.11780},
      archivePrefix={arXiv},
      primaryClass={econ.EM},
      url={https://arxiv.org/abs/2507.11780}, 
}

@article{whitehouse2026inferencevaluepolicyimprovement,
      title={Inference on the Value of Policy Improvement in Adaptive Experiments}, 
      author={Justin Whitehouse and Emma Brunskill and Vasilis Syrgkanis},
      year={2026},
      journal = {arXiv:2610.15789},
      archivePrefix={arXiv},
      primaryClass={stat.ML},
      url={https://arxiv.org/abs/2610.15789}, 
}

@inproceedings{zhang2021mestimators,
 author = {Zhang, Kelly and Janson, Lucas and Murphy, Susan},
 booktitle = {Advances in Neural Information Processing Systems},
 pages = {7460--7471},
 title = {Statistical Inference with {M}-Estimators on Adaptively Collected Data},
 url = {https://proceedings.neurips.cc/paper_files/paper/2021/file/3d7d9461075eb7c37fbbfcad1d7042c1-Paper.pdf},
 volume = {34},
 year = {2021}
}

@article{khamaru2021near,
author = {Koulik Khamaru and Yash Deshpande and Tor Lattimore and Lester Mackey and Martin J. Wainwright},
title = {{Near-optimal inference in adaptive linear regression}},
volume = {53},
journal = {The Annals of Statistics},
number = {6},
publisher = {Institute of Mathematical Statistics},
pages = {2329 -- 2355},
year = {2025},
doi = {10.1214/24-AOS2450},
URL = {https://doi.org/10.1214/24-AOS2450}
}
